\documentclass{amsart}

\usepackage{subfigure}
\usepackage{amsmath,amssymb, amsthm}
\usepackage{enumerate}
\usepackage{fancybox}
\usepackage{graphicx,color}
\usepackage{amsfonts}
\usepackage{multicol}
\usepackage{hyperref}
\usepackage{cleveref}
\usepackage{thmtools,thm-restate}
\usepackage[all]{xy}
\usepackage{amscd}

\numberwithin{equation}{section}

\newcommand{\V}[1]{\vspace{#1}}

\newcommand{\PP}{\mathbb{P}}
\newcommand{\C}{\mathbb{C}}

\newcommand{\Z}{\mathbb{Z}}
\newcommand{\Pa}{\partial}
\newcommand{\Diff}{\operatorname{Diff}}
\newcommand{\Div}{\operatorname{Div}}

\newcommand{\Crit}{\operatorname{Crit}}

\newcommand{\MCG}{\operatorname{MCG}}

\theoremstyle{plain}
\newtheorem{theorem}{Theorem}[section]

\newtheorem{lemma}[theorem]{Lemma}
\newtheorem{proposition}[theorem]{Proposition}

\theoremstyle{definition}

\newtheorem{remark}[theorem]{Remark}
\newtheorem{question}[theorem]{Question}

\title{Classification of genus-$1$ holomorphic Lefschetz pencils}

\author{Noriyuki Hamada}
\email{hamada@math.umass.edu}
\address{Department of Mathematics and Statistics,
	University of Massachusetts Amherst, 
	Lederle Graduate Research Tower,
	710 North Pleasant Street,
	Amherst, MA 01003-9305, USA
}

\author{Kenta Hayano}
\email{k-hayano@math.keio.ac.jp}
\address{Department of Mathematics, Faculty of Science and Technology, Keio University,
Yagami Campus, 3-14-1, Hiyoshi, Kohoku-ku, Yokohama, 223-8522, Japan}

\thanks{
The second author was supported by JSPS KAKENHI Grant Number JP17K14194.
This research was supported by Global Station for Big Data and Cybersecurity, a project of Global Institution for Collaborative Research and Education at Hokkaido University.
}

\begin{document}

\begin{abstract}

In this paper, we classify relatively minimal genus-$1$ holomorphic Lefschetz pencils up to smooth isomorphism. 
We first show that such a pencil is isomorphic to either the pencil on $\PP^1\times \PP^1$ of bi-degree $(2,2)$ or a blow-up of the pencil on $\PP^2$ of degree $3$, provided that no fiber of a pencil contains an embedded sphere. (Note that one can easily classify genus-$1$ Lefschetz pencils with an embedded sphere in a fiber.) 
We further determine the monodromy factorizations of these pencils and show that the isomorphism class of a blow-up of the pencil on $\PP^2$ of degree $3$ does not depend on the choice of blown-up base points. 
We also show that the genus-$1$ Lefschetz pencils constructed by Korkmaz-Ozbagci (with nine base points) and Tanaka (with eight base points) are respectively isomorphic to the pencils on $\PP^2$ and $\PP^1\times \PP^1$ above, in particular these are both holomorphic. 

\end{abstract}

\maketitle

\section{Introduction}

Classification problems of Lefschetz fibrations up to smooth isomorphism have attracted a lot of interest since around 1980. 
The first result concerning the problems was given in \cite{KasDeformationEllipticsurf,Moishezon_classification_genus1LF}, in which Kas and Moishezon independently classified genus-$1$ Lefschetz fibrations over the $2$-sphere. 
This classification result was extended to more general genus-$1$ fibrations: those with general base spaces and achiral singularities \cite{IwaseTorusFibration,MatsumotoTorusFibration,MatsumotoDiffeoEllipticsurf}.
Furthermore, Siebert and Tian \cite{SiebertTiangenus2LF} classified genus-$2$ Lefschetz fibrations over the $2$-sphere with transitive monodromies and no reducible fibers by showing that such fibrations are always holomorphic. 
Classifications up to stabilizations by fiber sums have also been studied in \cite{AurouxFibersumgenus2LF,AurouxStableclassificationLF,EndoKamadaChartHypLF,EndoKamadaDiracBraid,EHKTChartStabilizationLF}.

Whereas there are various results on classifications of Lefschetz \textit{fibrations}, very little is known about those of Lefschetz \textit{pencils}, except for the classification of genus-$0$ pencils that is given implicitly in \cite{PlamenevskayaVHMorris}.
In this paper, we will deal with the classification problem of genus-$1$ Lefschetz pencils. 
We first show that a genus-$1$ holomorphic Lefschetz pencil is isomorphic to either of the standard ones given below:

\begin{theorem}\label{T:genus-1 LP is blow-up of N_9 or S_8}

Let $f:X\dashrightarrow \PP^1$ be a genus-$1$ relatively minimal holomorphic Lefschetz pencil.
Suppose that no fibers of $f$ contain an embedded sphere. 
Then either of the following holds: 

\begin{itemize}

\item 
$f$ is smoothly isomorphic to the one obtained by blowing-up the Lefschetz pencil $f_n:\PP^2\dashrightarrow \PP^1$, which is the composition of the Veronese embedding $v_3:\PP^2\hookrightarrow \PP^9$ of degree $3$ and a generic projection $\PP^9 \dashrightarrow \PP^1$.

\item 
$f$ is smoothly isomorphic to the Lefschetz pencil $f_s:\PP^1\times \PP^1\dashrightarrow \PP^1$, which is the composition of the Segre embedding $\sigma:\PP^1\times \PP^1\hookrightarrow \PP^3$, the Veronese embedding $v_2:\PP^3\hookrightarrow \PP^9$ of degree $2$, and a generic projection $\PP^9\dashrightarrow \PP^1$. 

\end{itemize}

\end{theorem}

\noindent
The subscripts "n" and "s" for the Lefschetz pencils $f_n$ and $f_s$ represent the properties "non-spin" and "spin", respectively. 
Note that, needless to say, the blow-ups of $f_s$ also give Lefschetz pencils. Theorem~\ref{T:genus-1 LP is blow-up of N_9 or S_8} implies that such pencils are isomorphic to the blow-ups of $f_n$.
The assumption of relative minimality and the additional requirement that no fibers contain an embedded sphere with any self-intersection number should not be confused. The latter is required to exclude inessential Lefschetz pencils. For more detail, see Remark~\ref{rem:nosphere}.

Although the isomorphism classes of the pencils $f_n$ and $f_s$ do not depend on the choice of generic projections $\PP^9\dashrightarrow \PP^1$ (cf.~\cref{R:rel LP proj}), one cannot deduce immediately from \cref{T:genus-1 LP is blow-up of N_9 or S_8} that the isomorphism class of a blow-up of $f_n$ does not depend on the choice of blown-up base points (one can indeed find  in \cite{Hamada_MCK_preprint} examples of a pair of non-isomorphic pencils that are obtained by blowing-up a common pencil at the same number but different combinations of base points).
We next address this issue by examining the monodromies of $f_n$ and $f_s$. 

It is a standard fact in the literature that there is one-to-one correspondence between the isomorphism classes of genus-$g$ Lefschetz pencils with $m$ critical points and $k$ base points and the Hurwitz equivalence classes of factorizations 
\[
t_{c_m} \cdots t_{c_1} = t_{\delta_{1}} \cdots t_{\delta_{k}}
\]
of the boundary multi-twist $t_{\delta_{1}} \cdots t_{\delta_{k}}$ as products of positive Dehn twists in the mapping class group of a $k$-holed surface of genus $g$.
Here, $\delta_i$ stands for a simple closed curve parallel to the $i$-th boundary component.
Such a factorization is called a \textit{monodromy factorization} in general, or also a \textit{$k$-holed torus relation} when $g=1$.
Relying on the theory of braid monodromies due to Moishezon-Teicher \cite{MTI,MTII,MTIII,MTIV,MRT} we determine the monodromy factorizations of $f_n$ and $f_s$.
We further analyze the Hurwitz equivalence classes of the factorizations, and eventually show the following: 

\begin{theorem}\label{T:(main)smooth classification genus1 holLP}

Let $f:X\dashrightarrow \PP^1$ be a relatively minimal genus-$1$ holomorphic Lefschetz pencil without embedded spheres in fibers. 
The monodromy factorization of $f$ is Hurwitz equivalent to that of one of the pencils in table \ref{Tbl:list hol genus1 LP}. 
In particular, the isomorphism class of a blow-up of $f_n$ does not depend on the choice of blown-up base points.

\end{theorem}

\renewcommand{\arraystretch}{1.1}
\begin{table}[htpb]\phantomsection\label{Tbl:list hol genus1 LP}
	\begin{center}
	\begin{tabular}{|c||ccc|} \hline 
		pencil & 
\begin{minipage}[c]{18mm}
\centering
\V{.2em}

number of 

\V{-.2em}

base points
\end{minipage} &
monodromy factorization
& total space \\ \hline\hline
		$f_n$ & $9$ & $t_{a_1} t_{b_1} t_{b_2} t_{b_3} t_{a_4} t_{b_4} t_{b_5} t_{b_6} t_{a_7} t_{b_7} t_{b_8} t_{b_9} =\partial_9$ & $\PP^2$ \\ \hline
		$f_s$ & $8$ & $t_{a_1} t_{b_1} t_{b_2} t_{a_3} t_{b_3} t_{b_4} t_{a_5} t_{b_5} t_{b_6} t_{a_7} t_{b_7} t_{b_8} =\partial_8$ & $\PP^1 \times \PP^1$ \\ \hline
		$f_n \sharp \overline{\PP}{}^{2}$ & $8$ &$t_{a_1} t_{b_1} t_{a_2} t_{b_2} t_{b_3} t_{a_4} t_{b_4} t_{b_5} t_{b_6} t_{a_7} t_{b_7} t_{b_8} =\partial_8$ & $\PP^2 \sharp \overline{\PP}{}^{2}$ \\ \hline
		$f_n \sharp 2\overline{\PP}{}^{2}$ & $7$ & $t_{a_1} t_{b_1} t_{a_2} t_{b_2} t_{a_3} t_{b_3} t_{b_4} t_{a_5} t_{b_5} t_{b_6} t_{a_7} t_{b_7} =\partial_7$ & $\PP^2 \sharp 2\overline{\PP}{}^{2}$ \\ \hline
		$f_n \sharp 3\overline{\PP}{}^{2}$ & $6$ & $t_{a_1} t_{b_1} t_{a_2} t_{b_2} t_{a_3} t_{b_3} t_{a_4} t_{b_4} t_{a_5} t_{b_5} t_{a_6} t_{b_6} =\partial_6$ & $\PP^2 \sharp 3\overline{\PP}{}^{2}$ \\ \hline
		$f_n \sharp 4\overline{\PP}{}^{2}$ & $5$ & $t_{a_1}^2 t_{b_1} t_{a_2}^2t_{b_2} t_{a_3} t_{b_3} t_{a_4} t_{b_4} t_{a_5} t_{b_5} =\partial_5$ & $\PP^2 \sharp 4\overline{\PP}{}^{2}$ \\ \hline
		$f_n \sharp 5\overline{\PP}{}^{2}$ & $4$ &$t_{a_1}^2 t_{b_1} t_{a_2} ^2 t_{b_2} t_{a_3}^2 t_{b_3} t_{a_4}^2 t_{b_4} =\partial_4$ & $\PP^2 \sharp 5\overline{\PP}{}^{2}$ \\
		& & $\sim$ $(t_{a_1} t_{a_3} t_{b} t_{a_2} t_{a_4} t_{b})^2 =\partial_4$ & \\ \hline
		$f_n \sharp 6\overline{\PP}{}^{2}$ & $3$ &$t_{a_1}^3t_{b_1} t_{a_2}^3 t_{b_2} t_{a_3} ^3 t_{b_3} =\partial_3$ & $\PP^2 \sharp 6\overline{\PP}{}^{2}$ \\ 
		& & $\sim$ $(t_{a_1} t_{a_2} t_{a_3} t_{b})^3 =\partial_3$ & \\ \hline
		$f_n \sharp 7\overline{\PP}{}^{2}$ & $2$ & $(t_{a_1} t_{b} t_{a_2})^4 =\partial_2$ & $\PP^2 \sharp 7\overline{\PP}{}^{2}$ \\ \hline
		$f_n \sharp 8\overline{\PP}{}^{2}$ & $1$ & $(t_{a_1} t_{b})^6 =\partial_1$ & $\PP^2 \sharp 8\overline{\PP}{}^{2}$ \\
		\hline
	\end{tabular}
	\end{center}
	\caption{Classification of the genus-$1$ holomorphic Lefschetz pencils.
The curves in the table are given in \cref{F:k=general} and $\partial_k$ represents the boundary multi-twist $t_{\delta_1}\cdots t_{\delta_k}$.}
	\label{T:classification}
\end{table}

\begin{figure}[htbp]
	\centering
	\includegraphics[height=170pt]{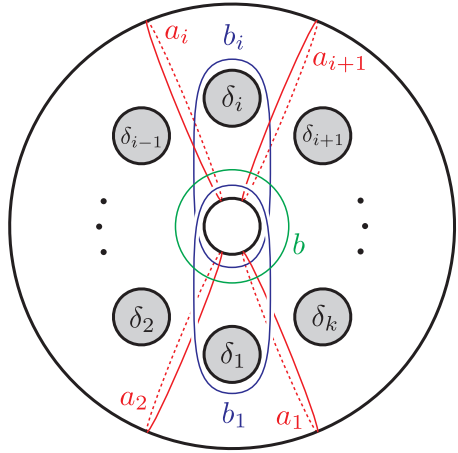}
	\caption{The curves on the $k$-holed torus $\Sigma_1^k$.} 	
	\label{F:k=general}
\end{figure}

Note that according to the aforementioned works of Kas and Moishezon~\cite{KasDeformationEllipticsurf,Moishezon_classification_genus1LF} the only genus-$1$ Lefschetz fibration that admits a $(-1)$-section is the well-known rational elliptic fibration $E(1) \to \PP^1$, whose monodromy factorization is $(t_a t_b)^6 =1$ where $a$ is the meridian and $b$ is the longitude of the torus.
Thus, any genus-$1$ Lefschetz pencil, even a non-holomorphic one (if exists), must descend to this fibration after blowing-up all the base points.
This is clearly reflected in table~\ref{T:classification}, where once more blowing-up of the pencil $f_n \sharp 8\overline{\PP}{}^{2}$ results in $E(1)= \PP^2 \sharp 9\overline{\PP}{}^{2}$ and $(t_at_b)^6=1$.

Examples of explicit $k$-holed torus relations were first discovered by Korkmaz and Ozbagci \cite{KorkmazOzbagci2008}, and then by Tanaka \cite{Tanaka2012}. In both of the works, the authors constructed those relations by combining the known relations (i.e.~the $2$-chain relation and the lantern relation) in the mapping class groups.\footnote{The first author also found factorizations in~\cite{Hamada2014, Hamada_MCK_preprint} in different contexts, which in fact can be shown to be Hurwitz equivalent to either Korkmaz-Ozbagci's or Tanaka's.}
We will show that the $k$-holed torus relations of Korkmaz-Ozbagci and Tanaka are Hurwitz equivalent to the monodromy factorizations in table~\ref{Tbl:list hol genus1 LP}, in particular we conclude that the Lefschetz pencils associated with their relations are holomorphic (\cref{T:isom Korkmaz-Ozbagci,T:isom Tanaka}).

The virtue of our presentations of the $k$-holed torus relations in table~\ref{Tbl:list hol genus1 LP} is that 
the curves involved are remarkably simple as they are well-organized lifts of the meridian and longitude of a closed torus.
As the $k$-holed torus relations are fundamentally important to construct relations in the mapping class groups of even higher genera, having simpler expressions may help those who try to use them.

As our results in the present paper take care of holomorphic pencils, the next step shall be the ultimate classification of genus-$1$ {\it smooth} Lefschetz pencils.
Although we speculate that any genus-$1$ Lefschetz pencil is isomorphic to one of the holomorphic ones, we do not have the machinery to prove this.
This leaves the following open question.
\begin{question}
Is there a non-holomorphic genus-$1$ Lefschetz pencil? In other words, is there a $k$-holed torus relation that is not Hurwitz equivalent to any of the $k$-holed torus relations in Table~\ref{Tbl:list hol genus1 LP}? 
\end{question}

The paper is organized as follows. 
In \cref{sec:preliminaries}, we briefly review basic properties of holomorphic Lefschetz pencils and monodromy factorizations. 
\Cref{sec:LP on complex surf} is devoted to proving \cref{T:genus-1 LP is blow-up of N_9 or S_8}. 
In \cref{sec:van cyc LP}, we determine monodromy factorizations of the pencils $f_n$ and $f_s$. 
We analyze combinatorial properties of the monodromy factorizations of $f_n$ and $f_s$ in \cref{sec:combinatorial structure}, completing the proof of \cref{T:(main)smooth classification genus1 holLP}.

\section{Preliminaries}\label{sec:preliminaries}

Throughout this paper, we will assume that manifolds are smooth, connected, oriented and closed unless otherwise noted. 
We denote the $n$-dimensional complex projective space by $\PP^n$. 
Let $X$ be a $4$-manifold. 
A \emph{Lefschetz pencil} on $X$ is a smooth mapping $f:X\setminus B \to \PP^1$ defined on the complement of a non-empty finite subset $B\subset X$ satisfying the following conditions: 

\begin{itemize}

\item 
for any critical point $p\in X$ of $f$, there exists a complex coordinate neighborhood $(U,\varphi:U\to \C^2)$ (resp.~$(V,\psi:V\to \C)$) at $p$ (resp.~$f(p)$) compatible with the orientation such that $\psi\circ f\circ \varphi^{-1}(z,w)$ is equal to $z^2+w^2$, 

\item 
for any $b\in B$, there exist a complex coordinate neighborhood $(U,\varphi)$ of $b$ compatible with the orientation and an orientation preserving diffeomorphism $\xi:\PP^1\to \PP^1$ such that $\xi \circ f \circ \varphi^{-1}(z,w)$ is equal to $[z:w]$, 

\item 
the restriction $f|_{\Crit(f)}$ is injective. 

\end{itemize}

\noindent
The set $B$ is called the \emph{base point set} of $f$. 
In this paper we will use the dashed arrow $\dashrightarrow$ to represent Lefschetz pencils, e.g.~$f:X\dashrightarrow \PP^1$, when we do not need to represent the base point set explicitly. 
(Note that this symbol will be also used to represent meromorphic mappings.)
For a Lefschetz pencil $f$, the genus of the closure of a regular fiber is called the \emph{genus} of $f$.

A Lefschetz pencil $f$ is said to be \emph{relatively minimal} if no fiber of $f$ contains a $(-1)$-sphere. 
Let $f:X\dashrightarrow \PP^1$ be a Lefschetz pencil, $\tilde{X}$ be a blow-up of $X$ at a point and $\pi:\tilde{X}\to X$ be the blow-down mapping. 
One can construct a Lefschetz pencil $\tilde{f}:\tilde{X}\dashrightarrow \PP^1$ so that $\tilde{f} = f\circ \pi$ on the complement of the exceptional sphere. 
Conversely, any relatively non-minimal Lefschetz pencil can be obtained from a relatively minimal one by this construction. 
In particular, relatively non-minimal Lefschetz pencils are inessential in the context of classification, and thus, \textit{we will assume that Lefschetz pencils are relatively minimal unless otherwise noted.}

\subsection{Holomorphic Lefschetz pencils}

A Lefschetz pencil $f:X\dashrightarrow \PP^1$ is said to be \emph{holomorphic} if there exists a complex structure of $X$ such that $f$ is holomorphic and we can take biholomorphic $\varphi, \psi$ and $\xi$ in the conditions in the definition above. 
A Lefschetz pencil on a complex surface $S$ is said to be \emph{holomorphic} if it is holomorphic with respect to the given complex structure. 
For a complex surface $S$, it is well-known that a divisor $D \in \Div(S)$ gives rise to a line bundle over $S$, which we denote by $[D]$ (see \cite{GH} for details). 

\begin{proposition}\label{T:properties holLP}

Let $S$ be a complex surface, $f:S\dashrightarrow \PP^1$ be a holomorphic Lefschetz pencil and $F\subset S$ be the closure of a fiber of $f$. 

\begin{enumerate}

\item 
The genus of $f$ is equal to $(2+\mathcal{F}^2+K_S(\mathcal{F}))/2$, where $K_S\in H^2(S;\Z)$ is the canonical class of $S$ and $\mathcal{F}\in H_2(S;\Z)$ is the homology class represented by $F$.

\item 
There exist sections $s_0,s_1$ of the line bundle $[F]$ such that $f$ is equal to $[s_0:s_1]:S\dashrightarrow \PP^1$.  

\item 
Let $C\subset S$ be a irreducible curve. 
The intersection number $C\cdot F$ is greater than or equal to $0$. 
Furthermore, it is equal to $0$ if and only if $C$ is a component of a fiber of $f$ without base points. 

\end{enumerate}

\end{proposition}

\begin{proof}
(1) is merely a consequence of the adjunction formula, and we can prove (2) in the same way as that for \cite[Lemma 3.1]{HamadaHayano}. 
In what follows we will prove (3). 
Let $\tilde{S}$ be the complex surface obtained by blowing-up $S$ at all the base points of $f$ and $\tilde{C}\subset \tilde{S}$ (resp.~$\tilde{F}\subset \tilde{S}$) be the proper transform of $C$ (resp.~$F$).
The pencil $f$ induces a fibration $\tilde{f}:\tilde{S}\to \PP^1$.
Without loss of generality we can assume that $\tilde{F}$ does not contain any singular point of $\tilde{C}$ and any critical point of $\tilde{f}|_{\tilde{C}}$. 
Since $\tilde{F}$ is a fiber of $\tilde{f}$, the intersection number $\tilde{F}\cdot \tilde{C}$ is equal to $\sharp(\tilde{F} \cap \tilde{C})$. 
Hence we obtain:
\[
F\cdot C =\tilde{F}\cdot \tilde{C}+\sharp(C\cap B)=\sharp (\tilde{F}\cap \tilde{C})+\sharp(C\cap B)\geq 0. 
\]
Moreover, the equality holds only if $C\cap B=\emptyset$ and $\tilde{f}|_{\tilde{C}}$ is a constant map. 
The latter condition implies that $C$ is contained in a fiber of $f$. 
\end{proof}

\begin{remark}\label{R:rel LP proj}

For a line bundle $L$ over a complex surface $S$ with sections, we can define a meromorphic mapping $\varphi_L:S\dashrightarrow \PP^{l-1}$ as follows:
\[
\varphi_L(x) = [s_1(x):\cdots :s_l(x)],
\]
where $s_1,\ldots, s_l$ is a basis of $H^0(S;L)$. 
The statement (2) of \cref{T:properties holLP} implies that $f$ is the composition of $\varphi_{[F]}:S\dashrightarrow \PP^{m-1}$ (where $m=\dim H^0(S;[F])$) and a projection $\PP^{m-1}\dashrightarrow \PP^1$. 
Note that the composition of $\varphi_L$ and a projection $\PP^l\dashrightarrow \PP^1$ is not always a Lefschetz pencil. 
It is known, however, that the composition is a Lefschetz pencil provided that $L$ is very ample and the projection is generic. 
Moreover, the smooth isomorphism class of the Lefschetz pencil does not depend on the choice of this projection (see \cite{Voisin,HamadaHayano}). 

\end{remark}

\subsection{Monodromy factorizations}\label{S:monodromy factorization}

For a compact oriented connected surface $\Sigma$ (possibly with boundaries), we denote by $\Diff(\Sigma)$ the set of self-diffeomorphisms of $\Sigma$ preserving the boundary pointwise, endowed with the Whitney $C^\infty$-topology. 
Let $\MCG(\Sigma) = \pi_0(\Diff(\Sigma))$, which has the group structure defined by the composition of representatives.

Let $f:X\setminus B\to \PP^1$ be a genus-$g$ Lefschetz pencil with $k$ base points and $Q=\{q_1,\ldots, q_m\}\subset \PP^1$ be the set of critical values of $f$. 
We take a point $q_0\in \PP^1\setminus Q$ and a path $\alpha_i\subset \PP^1$ ($i=1,\ldots,m$) from $q_0$ to $q_i$ satisfying the following conditions: 

\begin{itemize}

\item 
$\alpha_1,\ldots,\alpha_m$ are mutually distinct except at the common initial point $q_0$, 

\item 
$\alpha_1,\ldots, \alpha_m$ appear in this order when we go around $q_0$ counterclockwise. 

\end{itemize}

\noindent
The system of paths $(\alpha_1,\ldots, \alpha_m)$ satisfying the conditions above is called a \emph{Hurwitz path system} of $f$. 
Let $\gamma_i\subset \PP^1$ be a based loop with the base point $q_0$ obtained by connecting $q_0$ with a small circle oriented counterclockwise by $\alpha_i$. 
It is known that the monodromy along $\gamma_i$ is the Dehn twist along some simple closed curve $c_i\subset \overline{f^{-1}(q_0)}$, called a \emph{vanishing cycle} of $f$ with respect to the path $\alpha_i$. 
Furthermore, we can obtain the following relation in $\MCG(\overline{f^{-1}(q_0)}\setminus \nu B)$: 
\begin{equation}\label{Eq:monodromy factorization}
t_{c_m}\cdots t_{c_1} = t_{\delta_1} \cdots t_{\delta_k}, 
\end{equation}
where $\nu B$ is a tubular neighborhood of $B\subset \overline{f^{-1}(q_0)}$ and $\delta_1,\ldots, \delta_k\subset \overline{f^{-1}(q_0)}\setminus \nu B$ are simple closed curves parallel to the boundary components.  
We call this relation a \emph{monodromy factorization} of $f$. 
Conversely, let $\Sigma_g^k$ be a genus-$g$ compact surface with $k$ boundary components, and $c_1,\ldots, c_m\subset \Sigma_g^k$ be simple closed curves satisfying the relation \eqref{Eq:monodromy factorization} in $\MCG(\Sigma_g^k)$. 
We can construct a genus-$g$ Lefschetz pencil $f:X\setminus B \to \PP^1$ with $k$ base points and vanishing cycles $c_1,\ldots, c_m$, under some identification of the complement $\overline{f^{-1}(q_0)}\setminus \nu B$ of the closure of a regular fiber with $\Sigma_g^k$.

\section{Complex surfaces admitting genus-$1$ Lefschetz pencils}\label{sec:LP on complex surf}

This section is devoted to proving \cref{T:genus-1 LP is blow-up of N_9 or S_8}, which one can easily deduce from the following theorem. 

\begin{theorem}\label{T:classification divisor fibers}

Let $S$ be a complex surface, $f:S\dashrightarrow \PP^1$ be a genus-$1$ holomorphic Lefschetz pencil and $F$ be the closure of a fiber of $f$. 
Suppose that no fibers of $f$ contain an embedded sphere. 
Then either of the following holds: 

\begin{itemize}

\item 
the complex surface $S$ can be obtained by blowing-up $\PP^2$ at $l \leq 8$ points and $F$ is linearly equivalent to $3 H - \sum_{i=1}^{l} E_i$, where $H$ is the total transform of a projective line $H'$ in $\PP^2$ and $E_1,\ldots, E_l$ are the exceptional spheres.

\item 
the complex surface $S$ is $\PP^1\times \PP^1$ and $F$ is linearly equivalent to $2F_1+2F_2$, where $F_i$ is a fiber of the projection $\pi_i:\PP^1\times \PP^1\to \PP^1$ onto the $i$-th component. 

\end{itemize}

\end{theorem}

\begin{remark} \label{rem:nosphere}

This remark concerns the assumption that no fibers of a pencil contain an embedded sphere, which is required in not only the theorem above but also the main theorems in the paper.  
Even if a Lefschetz pencil is relatively minimal, a fiber of it might contain an embedded sphere. 
For example, let us consider the Lefschetz pencil $f_{g,k}:X_{g,k}\dashrightarrow \PP^1$ with the following monodromy factorization:
\begin{equation}\label{Eq:monodromy fact trivial}
t_{\delta_1}\cdots t_{\delta_k} = t_{\delta_1}\cdots t_{\delta_k}\mbox{ in }\MCG(\Sigma_g^k). 
\end{equation}
The total space $X_{g,k}$ is a ruled surface and the pencil $f_{g,k}$ has $k$ critical (resp.~base) points corresponding to the twists in the left-hand (resp.~right-hand) side of \eqref{Eq:monodromy fact trivial}. 
This pencil is relatively minimal but each singular fiber of it contains a sphere. 
Furthermore, there exist other types of such Lefschetz pencils with genus-$0$: the pencils of degree $1$ and $2$ curves in $\PP^2$. 
The former (resp.~the latter) gives rise to the trivial relation $1=t_\delta$ in $\MCG(D^2)$ (resp.~the lantern relation) as the monodromy factorization. 
Such pencils, however, are not important in the context of classification; if a (not necessarily holomorphic) relatively minimal Lefschetz pencil has an embedded sphere in a fiber, it is isomorphic to one of the examples given here.
This follows from the observation in \cite[Remark 2.4]{PlamenevskayaVHMorris} for genus-$0$ and the lemma below for higher genera. 

\end{remark}

\begin{lemma}\label{T:blow-up of rel min is rel min}

Let $f:X\dashrightarrow \PP^1$ be a relatively minimal Lefschetz pencil with genus-$g\geq 1$.
Suppose that there exists an embedded sphere in a fiber of $f$.
Then a monodromy factorization of $f$ is $t_{\delta_1}\cdots t_{\delta_k} = t_{\delta_1}\cdots t_{\delta_k}$. 

\end{lemma}

\begin{proof}[Proof of \cref{T:blow-up of rel min is rel min}]
Let $m$ and $k$ be the numbers of critical points and base points of $f$, respectively, and $t_{c_m}\cdots t_{c_1} = t_{\delta_1}\cdots t_{\delta_k}$ be a monodromy factorization of $f$, where $c_1,\ldots, c_m\subset \Sigma_g^k$ be simple closed curves in $\Sigma_g^k$.  
By capping the boundary of $\Sigma_g^k$ by disks, we can regard $\Sigma_g^k$ as a subsurface of the closed surface $\Sigma_g$. 
By the assumption, one of the vanishing cycles of $f$, say $c_1$, is not essential in $\Sigma_g$. 
Let $S$ be the closure of the genus-$0$ component of the complement $\Sigma_g^k\setminus c_1$. 
Since $f$ is relatively minimal, $S$ contains a boundary component of $\Sigma_g^k$. 
By capping all the boundary components of $\Sigma_g^k$ except for one in $S$, we obtain the following relation in $\MCG(\Sigma_g^1)$: 
\[
t_{\delta}\cdot t_{c_2} \cdots t_{c_m} = t_{\delta} \Rightarrow t_{c_2}\cdots t_{c_m} = 1
\]
If one of the curves $c_2,\ldots, c_m$ is essential in $\Sigma_g^1$, the equality above implies that there exists a relatively minimal non-trivial genus-$g$ Lefschetz fibration with a square-zero section, contradicting \cite[Proposition 3.3]{SmithGeometricMonodromy} and \cite[Lemma 2.1]{StipsiczIndecomposability}. 
Thus, all the curves $c_1,\ldots, c_m$ bounds a genus-$0$ subsurface in $\Sigma_g^k$. 
We can also deduce from the observation above that the fundamental group of $X$ is isomorphic to that of $\Sigma_g$. 

Let $S_i$ be the genus-$0$ component of $\Sigma_g^k\setminus c_i$. 
Suppose that $S_i$ contains more than one components of $\Pa \Sigma_g^k$ for some $i$. 
Then $X$ has a symplectic structure such that there exists a embedded symplectic sphere $C$ with positive square. 
Since $X$ is not rational, one can verify in the same way as that in the proof of \cite[Theorem 1.4 (ii)]{McDuffImmersedSphere} that $X$ is a irrational ruled surface, $C$ is away from a maximal disjoint family of exceptional spheres, and after blow-down $C$ becomes a fiber of a $\PP^1$-bundle. 
However, this contradicts that $C$ has positive square. 
We can eventually show that $S_i$ contains only one component of $\Pa \Sigma_g^k$ for each $i=1,\ldots, m$, in particular each $c_i$ is isotopic to some $\delta_j$. 
The lemma then follows from the fact that the subgroup of $\MCG(\Sigma_g^k)$ generated by $t_{\delta_1},\ldots,t_{\delta_k}$ is isomorphic to the free abelian group $\Z^k$. 
\end{proof}

\begin{proof}[Proof of \cref{T:classification divisor fibers}]
Let $f:S\dashrightarrow \PP^1$ be a genus-$1$ holomorphic Lefschetz pencil and $\tilde{f}:\tilde{S}\to \PP^1$ be the Lefschetz fibration obtained by blowing-up all the base points of $f$. 
By \cref{T:blow-up of rel min is rel min} and the assumption, $\tilde{f}$ is relatively minimal. 
We can deduce from the classification of genus-$1$ Lefschetz fibrations in the smooth category (given in \cite{Moishezon_classification_genus1LF}) that $\tilde{S}$ is diffeomorphic to $\PP^2\sharp 9\overline{\PP^2}$. 
Thus, applying \cite[Corollary 2]{FriedmanQin_Kodairadim}, we can show that $S$ is a rational surface, in particular $S$ is either $\PP^2$ or a blow-up of the Hirzebruch surface $S_n$ for some $n\geq 0$ (for the definition of $S_n$, see \cite[Chap.~4, \S.3]{GH}). 
Suppose that $S$ is the projective plane.
Then the number of the base points of $f$ is equal to $9$, and thus the self-intersection of $F$ is also equal to $9$.
Since the line bundle $[F]$ has at least two linearly independent sections by (2) of \cref{T:properties holLP}, $F$ is linearly equivalent to $a H$ for some $a >0$ (note that the linear equivalence class of a divisor of a simply connected K\"ahler manifold is uniquely determined by the corresponding second cohomology class).
The self-intersection of $a H$ is equal to $a^2$. 
Hence we can conclude that $F$ is linearly equivalent to $3H$. 

In the rest of the proof we assume that $S$ can be obtained by blowing-up the Hirzebruch surface $S_n$ ($n\geq 0$) $l'$ times ($0\leq l' \leq 7$).
Let $E_\infty'\subset S_n$ be a section of $S_n$ (as a $\PP^1$-bundle) with self-intersection $-n$ (which is unique when $n>0$), and $C'\subset S_n$ be a fiber of the same $\PP^1$-bundle on $S_n$. 
We denote the total transforms of $E_\infty'$ and $C'$ by $E_\infty$ and $C$, respectively. 
Let $\hat{E}_i\subset S$ be the total transform of the exceptional sphere appearing in the $i$-th blow-up of $S_n$. 
Since $S$ is simply connected and K\"ahler, we can assume that $F$ is linearly equivalent to the following divisor: 
\begin{equation*}
a E_\infty +b C -\sum_{i=1}^{l'} c_i \hat{E}_i \hspace{1em} (a,b,c_i \in \Z). 
\end{equation*}
All the components of $E_\infty, C$ and $\hat{E}_i$ are spheres. 
Since no fiber of $f$ contains a sphere, we can deduce the following inequality from (3) of \cref{T:properties holLP}:
\begin{equation}\label{Eq:ineq from nef}
a >0,\hspace{.5em}b > na\mbox{, and }c_i>0 \mbox{ ($i=1,\ldots,l'$)}. 
\end{equation}
Since the number of base points of $f$ is equal to the self-intersection of $F$, we obtain the following equality: 
\begin{equation}\label{Eq:condition base}
8-l' = -na^2 +2ab -\sum_{i=1}^{l'} c_i^2. 
\end{equation}
The canonical class of $S_n$ is represented by the divisor $-2 E_\infty - (2+n)C$ (see \cite[Chap.~4, \S.3]{GH}). 
Thus we can deduce the following equality from (1) of \cref{T:properties holLP}:
\begin{equation}\label{Eq:condition genus}
8-l'+(n-2)a-2b+\sum_{i=1}^{l'}c_i = 0 \Leftrightarrow b =\frac{1}{2}\left(8-l'+(n-2)a+\sum_{i=1}^{l'}c_i\right).
\end{equation}
Combining the equalities \eqref{Eq:condition base} and \eqref{Eq:condition genus}, we obtain:
{\allowdisplaybreaks
\begin{align}
\nonumber&-na^2 +a \left((n-2)a+8-l'+\sum_{i=1}^{l'}c_i\right) -\sum_{i=1}^{l'}c_i^2-(8-l') =0\\
\label{Eq:quad equation on a}\Leftrightarrow& -2a^2+\left(8-l'+\sum_{i=1}^{l'}c_i\right)a-\sum_{i=1}^{l'}c_i^2-(8-l') =0.
\end{align}
}
We can regard \eqref{Eq:quad equation on a} as a quadratic equation on $a$, whose discriminant must be non-negative. 
Thus the following inequality holds: 
{\allowdisplaybreaks
\begin{align}
\nonumber&\left(8-l'+\sum_{i=1}^{l'}c_i\right)^2-8\left(\sum_{i=1}^{l'}c_i^2+(8-l')\right)\geq 0\\
\label{Eq:discriminant a}\Leftrightarrow& 
\left(\sum_{i=1}^{l'}c_i\right)^2+2(8-l')\sum_{i=1}^{l'}c_i-8\sum_{i=1}^{l'}c_i^2-l(8-l')\geq0.
\end{align}
}
Applying the Cauchy-Schwarz inequality to the vectors $\left(\sum_{i=1}^{l'}c_i,\ldots,\sum_{i=1}^{l'}c_i\right)$ and $(c_1,\ldots,c_{l'})$, we obtain the following inequality: 
\begin{equation}\label{Eq:norm vs inner product}
\left(\sum_{i=1}^{l'}c_i\right)^2 \leq \sqrt{l'}\left(\sum_{i=1}^{l'}c_i\right)\cdot \sqrt{\sum_{i=1}^{l'}c_i^2} \Rightarrow \sum_{i=1}^{l'}c_i \leq \sqrt{l'\sum_{i=1}^{l'}c_i^2}.
\end{equation}
Combining the inequalities \eqref{Eq:discriminant a} and \eqref{Eq:norm vs inner product}, we eventually obtain: 
{\allowdisplaybreaks
\begin{align*}
\nonumber& 
l'\sum_{i=1}^{l'}c_i^2+2\sqrt{l'}(8-l')\sqrt{\sum_{i=1}^{l'}c_i^2}-8\sum_{i=1}^{l'}c_i^2-l'(8-l')\geq0\\
\Rightarrow & -(8-l')\left(\sqrt{\sum_{i=1}^{l'}c_i^2}-\sqrt{l'}\right)^2\geq0. 
\end{align*}
}
Since $l'$ is less than $8$, we can deduce from this inequality that the sum $\sum_{i=1}^{l'}c_i^2$ is equal to $l'$. 
This equality together with the inequality in \eqref{Eq:ineq from nef} implies that $c_1,\ldots, c_{l'}$ are all equal to $1$. 
By substituting $1$ for all the $c_i$'s in \eqref{Eq:quad equation on a}, we obtain:
\[
-2a^2 +8a -8 =0 \Rightarrow a=2. 
\] 
We can further deduce from \eqref{Eq:condition genus} that $b$ is equal to $n+2$. 
Since $b$ is greater than $na$, $n$ is equal to $0$ or $1$. 

If $n$ is equal to $1$, the complex surface $S$ is a blow-up of $S_1$, which is a blow-up of $\PP^2$ at a single point. 
In other words, there is a sequence of blow-up from $\PP^2$ to $S$:
\begin{equation}\label{Eq:seq blow-up}
S^{(0)}:=\PP^2 \leftarrow S^{(1)} := S_1 \leftarrow S^{(2)} \leftarrow \cdots \leftarrow S^{(l)} =: S\mbox{, where $l=l'+1$}. 
\end{equation}
We denote the exceptional sphere in $\Sigma_1$ by $\hat{E}_0'$.
The divisors $E_\infty'$ and $C'$ are respectively linearly equivalent to $\hat{E}_0'$ and $H'-\hat{E}_0'$. 
Let $\hat{E}_0\subset S$ be the total transform of $\hat{E}_0'$. 
The closure of a fiber $F$ of $f$ is linearly equivalent to $a E_\infty + b C -\sum_{i=1}^{l'} \hat{E}_i = 3H- \sum_{i=1}^{l} E_i$, where $E_i= \hat{E}_{i-1}$.  
Suppose that the $j$-th blow-up in the sequence \eqref{Eq:seq blow-up} is applied at a point on the exceptional sphere appearing in the $i$-th blow-up for some $i<j$. 
Then the divisor $E_i-E_j$ would be linearly equivalent to a positive linear combination of spheres.  
By (3) of \cref{T:properties holLP} the self-intersection $(E_i-E_j)\cdot F$ would be positive, but this is not the case since $F$ is linearly equivalent to $3H- \sum_{i=1}^{l} E_i$.
We can eventually conclude that $S$ is obtained from $\PP^2$ by blowing-up $l$ points. 

Finally, suppose that $n$ is equal to $0$. 
The complex surface $S$ is $S_0=\PP^1\times \PP^1$, and $E_\infty'$ and $C'$ are respectively equal to $F_1$ and $F_2$. 
Since the blow-up of $\PP^1\times \PP^1$ at a single point is biholomorphic to the surface obtained by blowing-up $\PP^2$ at two points, we can assume that $l$ is equal to $0$ without loss of generality. 
The closure of a fiber $F$ is then linearly equivalent to $aF_1+bF_2 = 2F_2+2F_2$. 
This completes the proof of \cref{T:classification divisor fibers}. 
\end{proof}

\begin{proof}[Proof of \cref{T:genus-1 LP is blow-up of N_9 or S_8}]
We first observe that the Veronese embedding $v_3$ and the composition $v_2\circ \sigma$ are embeddings corresponding to the very ample line bundles $[3H]$ and $[2F_1+2F_2]$, respectively. 
Thus, according to \cref{R:rel LP proj}, the corollary holds if $S$ is either $\PP^2$ or $\PP^1\times \PP^1$. 
Suppose that $S$ is obtained by blowing up $\PP^2$ $l$ times. 
We can regard the Lefschetz pencil $f$ as a projective line in the complete linear system $\PP(H^0(S;[F]))$. 
Let $E = \sum_{i=1}^{l} E_i$ be the exceptional divisor, $\pi:S\to \PP^2$ be the blow-down mapping and $s \in H^0(S;[E])$ a non-trivial section. 
We can then define the following linear mapping:
\[
\xi: H^0(S;[F]) \to H^0(S;[F]\otimes [E])=H^0(S;\pi^\ast[3H]), \hspace{.5em} \xi(\tau) = \tau \otimes s. 
\]
Since the blow-down mapping $\pi$ is birational, we can identify $H^0(S;[F]\otimes [E])$ with $H^0(\PP^2;[3H])$ via $\pi$. 
Under this identification, the image $\xi(f)$ is a genus-$1$ Lefschetz pencil defined on $\PP^2$, which is smoothly isomorphic to $f_n$ by \cref{R:rel LP proj}, and $f$ can be obtained by blowing-up $\xi(f)$.  
\end{proof}

\section{Vanishing cycles of genus-$1$ Lefschetz pencils}\label{sec:van cyc LP}

As we have shown, any genus-$1$ Lefschetz pencil can be obtained by blowing-up either of the pencils $f_n$ or $f_s$ in \cref{T:genus-1 LP is blow-up of N_9 or S_8}. 
In this section we will determine vanishing cycles of these pencils relying on the theory of braid monodromies due to Moishezon and Teicher.
Throughout this section, we denote the projective varieties $v_3(\PP^2)$ and $v_2\circ \sigma(\PP^1\times \PP^1)$ by $U_n$ and $U_s$, respectively.

\subsection{Braid monodromy techniques}

In this subsection, we will give a brief review on the theory of braid monodromies. 
We will first explain how the theory is related with vanishing cycles of Lefschetz pencils appearing as generic pencils of very ample line bundles, and then recall several facts we need to obtain monodromies of $f_n$ and $f_s$. 
The reader can refer to \cite{MTI,MTII,MTIII,MTIV} for more details on this subject.

Let $V\subset \PP^n$ be a non-singular projective surface. 
Restricting a generic projection $\PP^n\dashrightarrow \PP^2$, we obtain a regular mapping $\pi: V\to \PP^2$ whose critical value set $C$ is a curve with nodes and cusps. 
We further take a generic projection $\pi':\PP^2\dashrightarrow \PP^1$ with base point $p_0\in \PP^2$ so that the composition $f:=\pi'\circ \pi:V\dashrightarrow \PP^1$ is a Lefschetz pencil. 
The critical point set of $f$ is equal to the set of critical points of $\pi$ whose image by $\pi$ is a branch point of the restriction $\pi'|_C$. 
We can obtain the monodromy (or equivalently, vanishing cycles) of $f$ from the \emph{braid monodromy} of $C$ (around branch points of $\pi'|_C$) explained below. 

Let $Q:=\{q_1,\ldots,q_m\}\subset \PP^1$ be the set of images (by $\pi'$) of branch points of $\pi'|_C$.  
Take a reference point $q_0\in \PP^1\setminus Q$. 
The closure of the preimage $\overline{\pi'^{-1}(q_0)}$ (which is equal to $\pi'^{-1}(q_0)\cup \{p_0\}$) is a line in $\PP^2$ intersecting $C$ at $d := \deg C$ points. 
We take $d+1$ points $v_0,v_1,\ldots,v_d \in S^2$ and fix an identification of the triple $(\overline{\pi'^{-1}(q_0)},\overline{\pi'^{-1}(q_0)}\cap C,\{p_0\})$ with $(S^2,\{v_1,\ldots,v_d\},\{v_0\})$. 
Note that we can also identify the restriction $\pi|_{\overline{f^{-1}(q_0)}}:\overline{f^{-1}(q_0)}\to \overline{\pi'^{-1}(q_0)}$ with a simple branched covering $\theta:\Sigma\to S^2$ branched at $v_1,\ldots, v_d$ (where a simple branched covering is a branched covering such that all the branched points have degree $2$). 
We next take a Hurwitz path system $(\alpha_1,\ldots,\alpha_m)$ of $f$ with the base point $q_0$, and the corresponding loops $\gamma_i$ for $i=1,\ldots,m$ as we took in \cref{S:monodromy factorization}. 
Taking the isotopy class of a parallel transport along $\gamma_i$ preserving $C$, we obtain a sequence of elements $\tau_1,\ldots, \tau_m$ of the braid group $B_d$ defined as follows: 
\[
B_d := \pi_0(\Diff(S^2,\{v_1,\ldots, v_d\},v_0)), 
\]
where we denote by $\Diff(S^2,\{v_1,\ldots, v_d\},v_0)$ the group of orientation-preserving self-diffeomorphisms of $S^2$ preserving $v_0$ and the set $\{v_1,\ldots,v_d\}$.  
It is easy to see that each element $\tau_i$ is a half twist along some path $\beta_i\subset S^2$ between two points in $\{v_1,\ldots, v_d\}$.
The path $\beta_i$ is called a \emph{Lefschetz vanishing cycle} of the corresponding branched point of $\pi'|_C$.   
The preimage $\theta^{-1}(\beta_i)$ has the unique circle component $c_i\subset \Sigma$, and this circle is a vanishing cycle of a Lefschetz singularity of $f$ in $f^{-1}(q_i)$ with respect to the path $\alpha_i$.

\begin{remark}

In the series of papers of Moishzon-Teicher, a Lefschetz vanishing cycle and a braid monodromy are defined not only for branched points of the restriction of a projection on the critical value set, but also for multiple points and cusps of a general curve in $\PP^2$. 
The reader can refer to \cite{MTII}, for example, for details of this subject. 
Note that we will deal with braid monodromies of multiple points (which is a Dehn twist along some simple closed curve in a punctured sphere) in order to determine vanishing cycles of $f_n$ and $f_s$. 

\end{remark}

\begin{remark}

Although the product $t_{c_m} \cdots t_{c_1}$ in $\MCG(\overline{f^{-1}(q_0)})$ is equal to the unit, the product $\tau_m \cdots \tau_1$ is \emph{not} equal to the unit since we do not consider braid monodromies of nodes and cusps of the critical value set $C$. 

\end{remark}

In summary, we can get vanishing cycles of the Lefschetz pencils $f_n$ and $f_s$ in \cref{T:genus-1 LP is blow-up of N_9 or S_8} once we obtain Lefschetz vanishing cycles of the branch points of the critical value sets of generic projections from $U_n$ and $U_s$ to $\PP^2$ (and the monodromies of simple branched coverings defined over a line in $\PP^2$, which can be obtained easily in our situations). 
Moishezon and Teicher \cite{MTIV} have obtained the Lefschetz vanishing cycles for $U_n$ by giving a \emph{projective degeneration} of $U_n$ to a union of planes, and then analyzing how Lefschetz vanishing cycles are changed in the regeneration (the opposite deformation of the degeneration).  
As we will observe below, the Lefschetz vanishing cycles for $U_s$ can also be obtained in the same way. 
In what follows, we will review the definition of a projective degeneration and those for $U_n$ and $U_s$ given in \cite{MTIII} and \cite{MRT}, respectively. 

An algebraic set $U_0\subset \PP^{n_0}$ is said to be \emph{equivalent} to another algebraic set $U_1\subset \PP^{n_1}$ if there exist an algebraic set $W\subset \PP^N$ and projections $\pi_0:\PP^N\dashrightarrow \PP^{n_0}$ and $\pi_1:\PP^N\dashrightarrow \PP^{n_1}$ such that the restriction $\pi_i|_W:W\to U_i$ is an isomorphism for $i=0,1$. 
An algebraic set $U'\subset \PP^m$ is a \emph{projective degeneration} of $U\subset \PP^n$ if there exists an algebraic set $W\subset \PP^N\times \C$ 
such that $W\cap (\PP^N\times \{0\})$ is equivalent to $U'$ and $W\cap (\PP^N\times \{\varepsilon\})$ is equivalent to $U$ for any $\varepsilon$ with sufficiently small $|\varepsilon|>0$. 
In this paper, we mean by $U\rightsquigarrow U'$ that $U'$ is the result of a projective degenerations from $U$.
Following the notations in \cite{MTIII}, we will describe components of algebraic sets as follows: 

\begin{itemize}

\item 
A surface equivalent to the image of the Veronese embedding of degree $d$ on $\PP^2$ is denoted by $V_d$, and described by a triangle in figures.  

\item 
A surface equivalent to the image of the embedding $\varphi_{[E_\infty + lC]}$ on $S_1$ ($l>1$) is denoted by $T_l$, and described by a trapezoid in figures. 

\item 
A surface equivalent to the image of the embedding $\varphi_{[a F_1+b F_2]}$ on $\PP^1\times \PP^1$ ($a,b >0$) is denoted by $U_{a,b}$, and described by a square in figures. 

\end{itemize}

\begin{theorem}[\cite{MTIII}. A projective degeneration of $U_n$.]\label{T:proj degen U_n}

There exists a sequence of projective degenerations $U_n =:Y^{(0)}\rightsquigarrow Y^{(1)} \rightsquigarrow \cdots \rightsquigarrow Y^{(5)}$ from $U_n$ to a union of $9$ planes $Y^{(5)}$. 
The intermediate algebraic sets are described in \cref{F:proj degen U_n}. 

\begin{figure}[htbp]
\centering
\subfigure[$Y^{(0)}$.]{
\includegraphics[width=25mm]{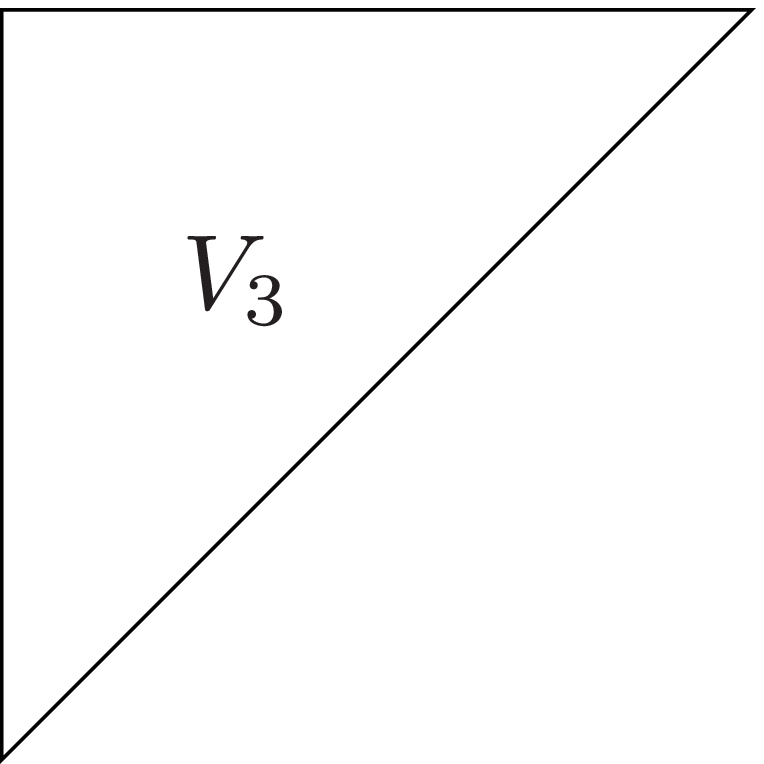}\label{F:variety Y0}}
\hspace{.4em}\subfigure[$Y^{(1)}$.]{
\includegraphics[width=25mm]{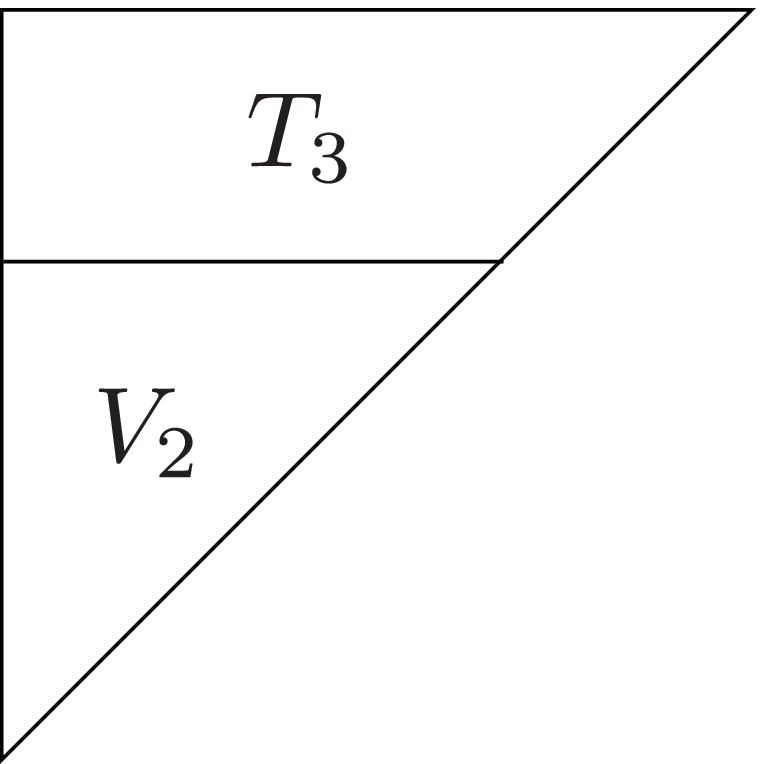}\label{F:variety Y1}}
\hspace{.4em}\subfigure[$Y^{(2)}$.]{
\includegraphics[width=25mm]{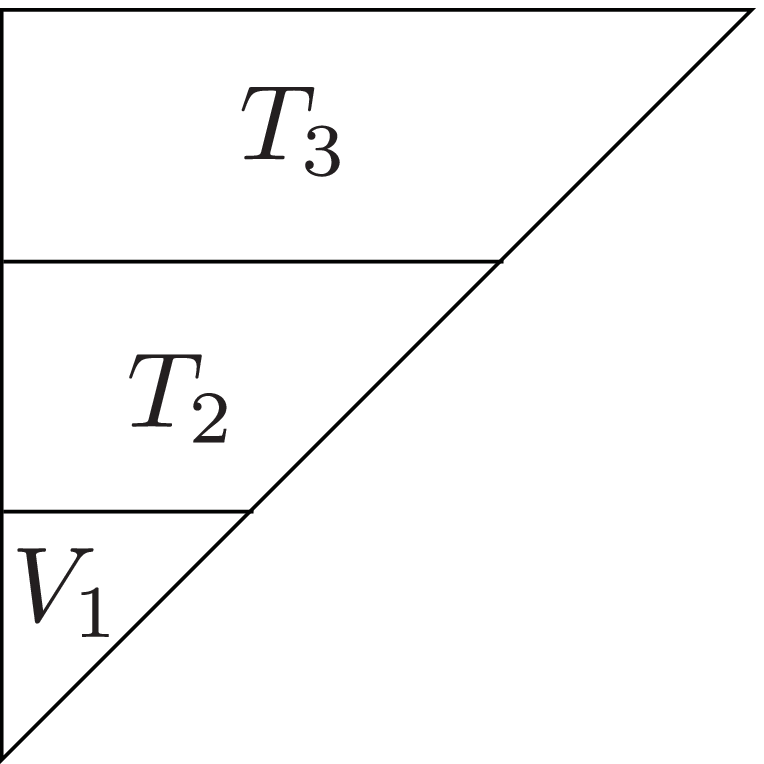}\label{F:variety Y2}}

\subfigure[$Y^{(3)}$.]{
\includegraphics[width=25mm]{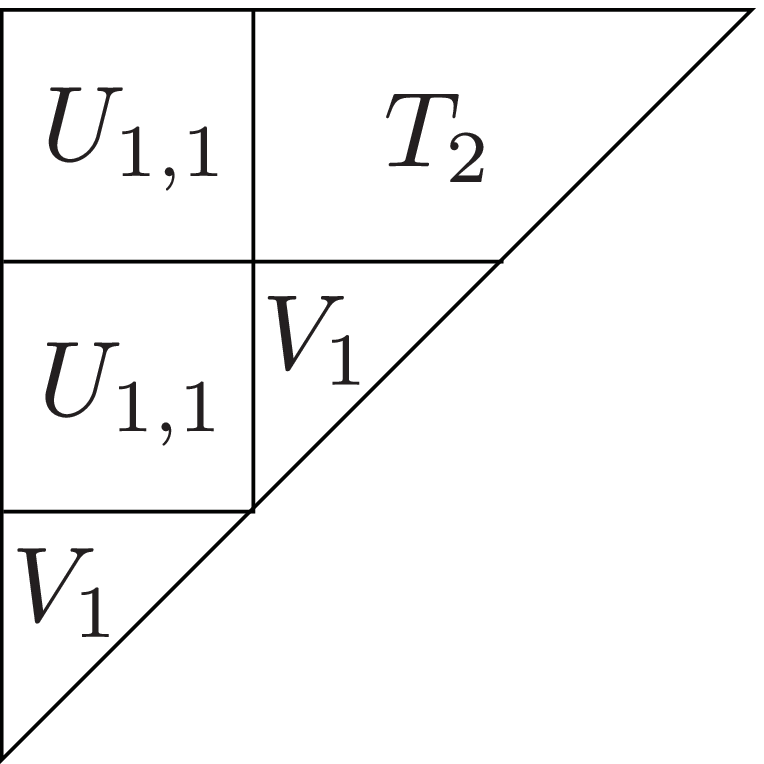}\label{F:variety Y3}}
\hspace{.4em}\subfigure[$Y^{(4)}$.]{
\includegraphics[width=25mm]{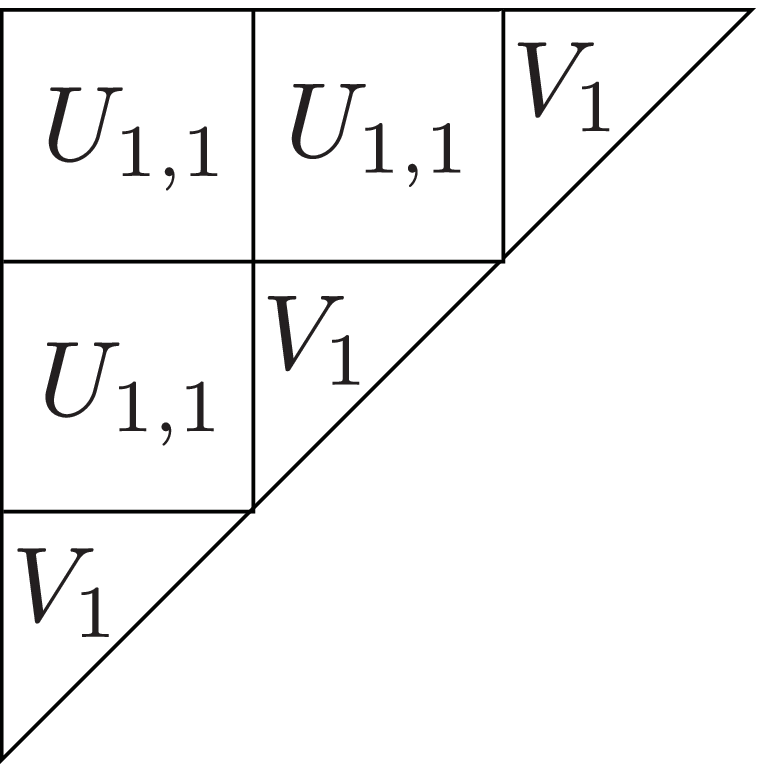}\label{F:variety Y4}}
\hspace{.4em}\subfigure[$Y^{(5)}$.]{
\includegraphics[width=25mm]{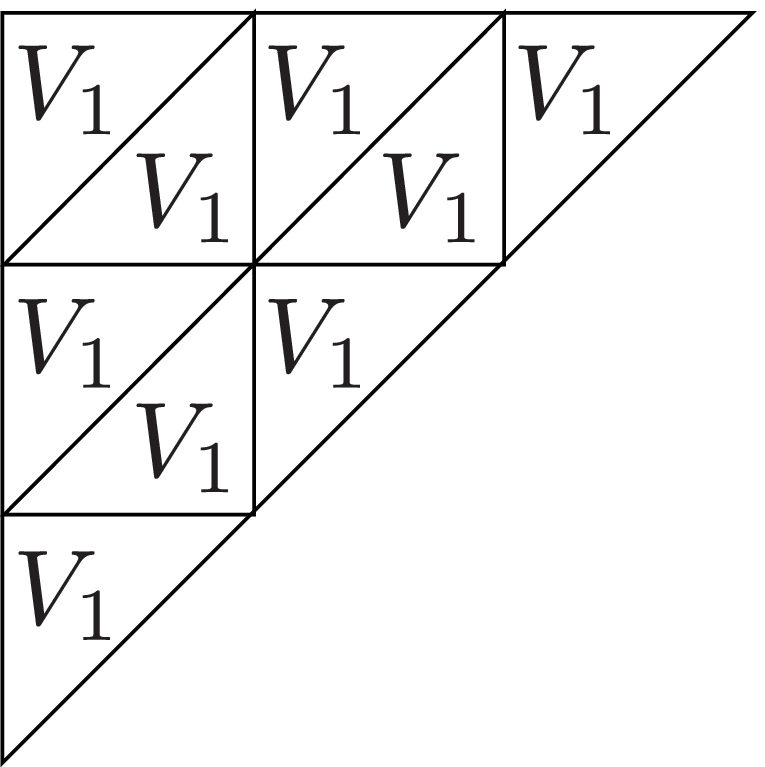}\label{F:variety Y5}}
\caption{A sequence of projective degenerations of $U_n$.}
\label{F:proj degen U_n}
\end{figure}

\end{theorem}

\begin{theorem}[\cite{MRT}. A projective degeneration of $U_s$.]\label{T:proj degen U_s}

There exists a sequence of projective degenerations $U_s =:Z^{(0)}\rightsquigarrow Z^{(1)} \rightsquigarrow Z^{(2)} \rightsquigarrow Z^{(3)}$ from $U_s$ to a union of $8$ planes $Z^{(3)}$. 
The intermediate algebraic sets are described in \cref{F:proj degen U_s}. 

\begin{figure}[htbp]
\centering
\subfigure[$Z^{(0)}$.]{
\includegraphics[width=25mm]{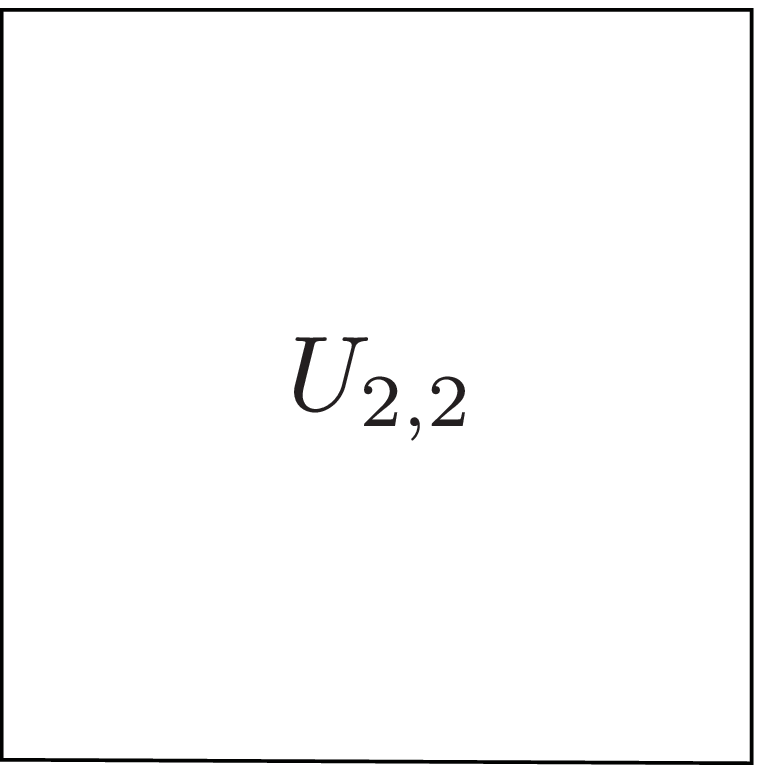}\label{F:variety Z0}}
\hspace{.4em}\subfigure[$Z^{(1)}$.]{
\includegraphics[width=25mm]{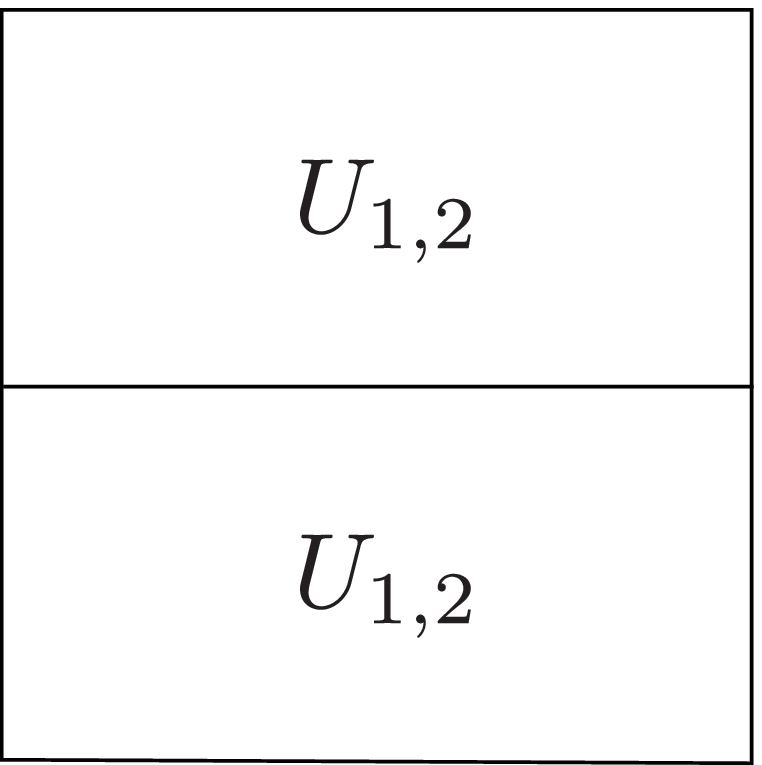}\label{F:variety Z1}}
\hspace{.4em}\subfigure[$Z^{(2)}$.]{
\includegraphics[width=25mm]{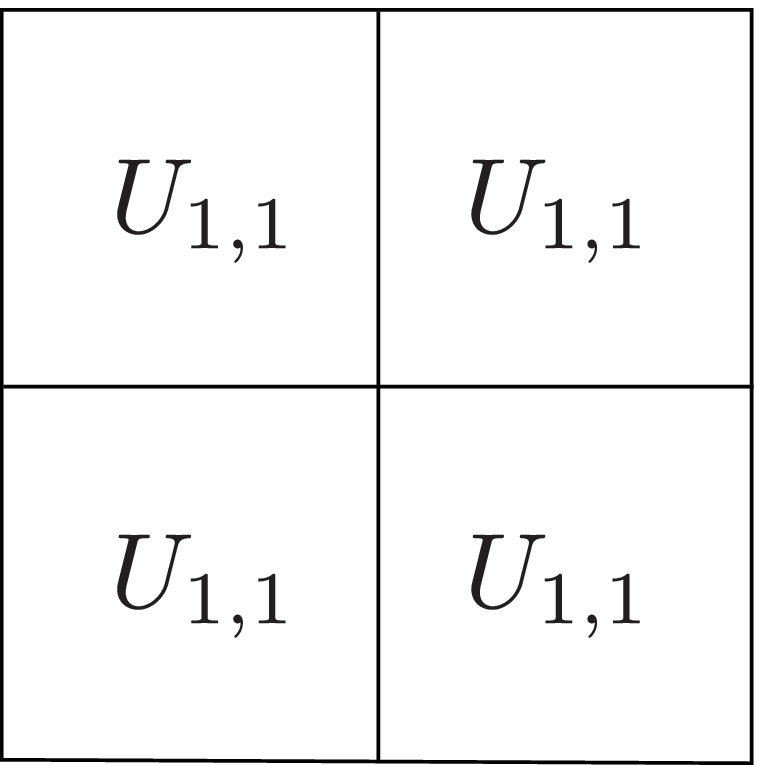}\label{F:variety Z2}}
\hspace{.4em}\subfigure[$Z^{(3)}$.]{
\includegraphics[width=25mm]{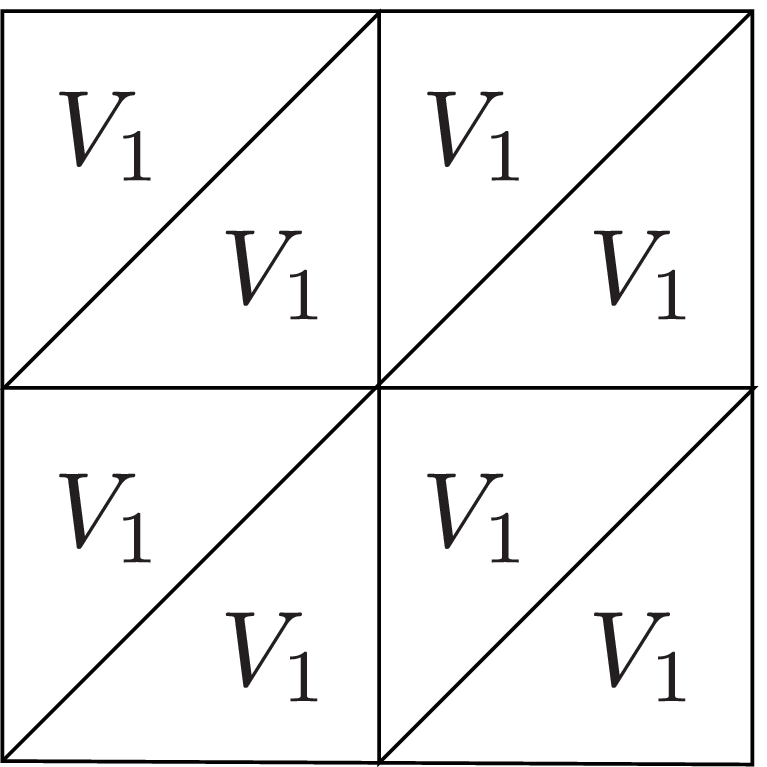}\label{F:variety Z3}}
\caption{A sequence of projective degenerations of $U_s$.}
\label{F:proj degen U_s}

\end{figure}

\end{theorem}

\begin{remark}\label{R:configuration intersections}

In each of the intermediate algebraic sets in \cref{F:proj degen U_n,F:proj degen U_s}, any two components adjacent to each other intersect on a rational curve, and the configuration of these curves are same as that of the segments between two regions in the figures. 
For example, in the algebraic set $Y^{(5)}$, there are $9$ lines appearing as intersections of two adjacent components, and $7$ multiple points in the line arrangement (see \cref{F:plane arrangement Un}).

\end{remark}

According to the observation in \cref{R:configuration intersections}, the sets of singular points of the algebraic sets $Y^{(5)}$ and $Z^{(3)}$ are unions of lines, and so are the images of them by projections to $\PP^2$. 
The braid monodromies of these line arrangements in $\PP^2$ are completely determined in \cite[Theorem IX.2.1]{MTI}.
We will next review the relation between these braid monodromies and those for the original varieties $U_n$ and $U_s$ discussed in \cite{MTII,MTIV}. 

In the sequences of regenerations given in \cref{T:proj degen U_n,T:proj degen U_s}, the line arrangement in $Y^{(5)}$ (resp.~$Z^{(3)}$) is also regenerated to the critical value set of the restriction of a generic projection on $U_n$ (resp.~$U_s$). 
In this regeneration process, each line in the arrangement is ``doubled'' in the following sense: if some small disk $D$ intersects a line in the arrangement at the center of $D$ transversely, this disk intersects the critical value set of the restriction of a generic projection at two points. (Note that, without loss of generality, we can assume that the critical value set is sufficiently close to the line arrangement. See \cite[\S.1]{MRT})
Furthermore, taking account of the plane arrangement and its regeneration, we can observe that each of the multiple points of the line arrangement has either of the following two properties:

\begin{itemize}

\item 
Three planes $P_1,P_2,P_3$ go through this point. 
Among the three planes, $P_i$ and $P_{i+1}$ ($i=1,2$) intersect on a line, while $P_1$ and $P_3$ intersect only at the point, in particular two lines $P_1\cap P_2$ and $P_2\cap P_3$ go through the point. 
In the regeneration process the line $P_1\cap P_2$ regenerates before the regeneration of $P_2\cap P_3$

\item 
Six planes $P_1,\ldots, P_6$ go through this point. 
Among the six planes, $P_i$ and $P_j$ intersect on a line if $|j-k|=1$ modulo $6$, or intersect only at the point otherwise. 
Among six lines $P_1\cap P_2,\ldots,P_6\cap P_1$, $P_1\cap P_2$ and $P_4\cap P_5$ regenerate first at the same time, $P_2\cap P_3$ and $P_5\cap P_6$ then regenerate at the same time, and lastly $P_3,\cap P_4$ and $P_6\cap P_1$ regenerate at the same time.  

\end{itemize} 

\noindent
In \cite{MRT}, the former multiple point is called a \emph{$2$-point}, while the latter one is called a \emph{type M $6$-point}. 
Following the rules below, we can determine the braid monodromies of branch points appearing around these points after the regeneration: 

\begin{theorem}[{\cite[Lemma 1]{MTIV}}]\label{T:regen rule 2point}

One branch point appears around a $2$-point after the regeneration. 
Suppose that the Lefschetz vanishing cycle of the $2$-point is a path $\beta$ shown in \cref{F:LVC 2point}, where $v_i$ is the intersection of the reference fiber and the line $P_i\cap P_{i+1}$ ($i=1,2$).
Then the Lefschetz vanishing cycle of the branch point appearing after the regeneration is the path $\beta'$ shown in \cref{F:LVC 2point after}. 

\begin{figure}[htbp]
\centering
\subfigure[]{\includegraphics[width=30mm]{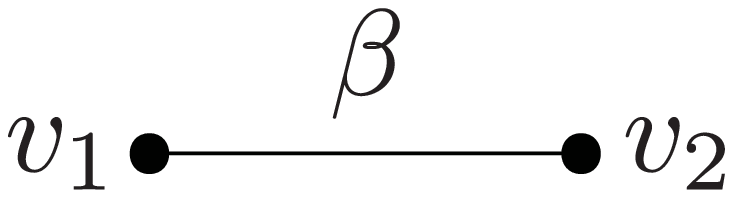}
\label{F:LVC 2point}}\hspace{.5em}
\subfigure[]{\includegraphics[width=30mm]{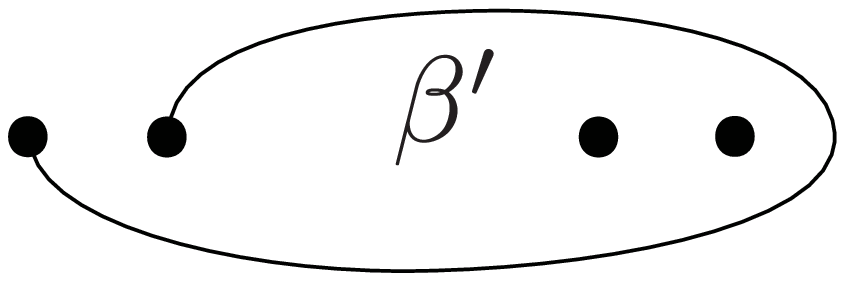}
\label{F:LVC 2point after}}
\caption{Lefschetz vanishing cycles of (a) a $2$-point and (b) the branch point around a $2$-point.}
\label{F:LVC for 2point rengeneration}
\end{figure}

\end{theorem}

\begin{theorem}[{\cite[Lemmas 5, 6, 7 and 8]{MTIV}}]\label{T:regen rule 6point}

Six branch points appear around a type M $6$-point after the regeneration. 
Suppose that the Lefschetz vanishing cycle of the type M $6$-point is a system of paths shown in \cref{F:LVC 6point}, where $v_i$ is the intersection of the reference fiber and the line $P_i\cap P_{i+1}$ (taking indeces modulo $6$). 
Then there exists a system of reference paths $(\alpha_1,\ldots, \alpha_6)$ for the six branch points, which appear in this order when we go around the reference point counterclockwise, such that the Lefschetz vanishing cycle associated with $\alpha_i$ is the path $\beta_i$, where $\beta_1$ and $\beta_6$ are shown in \cref{F:LVC beta1,F:LVC beta6}, while $\beta_2,\beta_3,\beta_4,\beta_5$ are defined as:
\[
\beta_2 = \beta,\hspace{.5em}\beta_3 = \tau_{\gamma_3}^{-1}\tau_{\gamma_4}^{-1}(\beta), \hspace{.5em}\beta_4 = \tau_{\gamma_1}^{-1}\tau_{\gamma_2}^{-1}(\beta), \hspace{.5em}\beta_5 = \tau_{\gamma_1}^{-1}\tau_{\gamma_2}^{-1}\tau_{\gamma_3}^{-1}\tau_{\gamma_4}^{-1}(\beta).  
\] 
(Here we denote the positive half twist along $\gamma_i$ by $\tau_{\gamma_i}$.)

\begin{figure}[htbp]
\centering
\subfigure[The Lefschetz vanishing cycle of a type M $6$-point.]{\includegraphics[width=60mm]{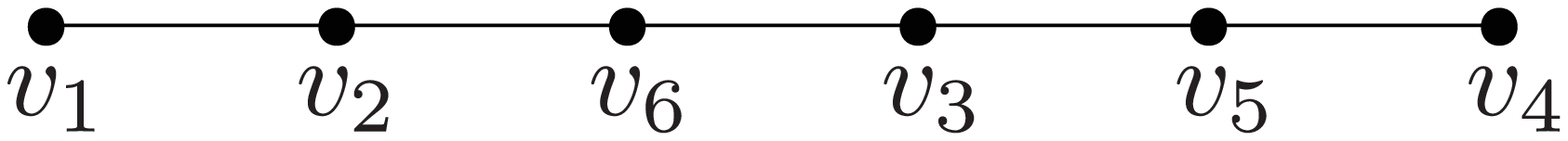}
\label{F:LVC 6point}}

\subfigure[The Lefschetz vanishing cycle $\beta_1$.]{\includegraphics[width=60mm]{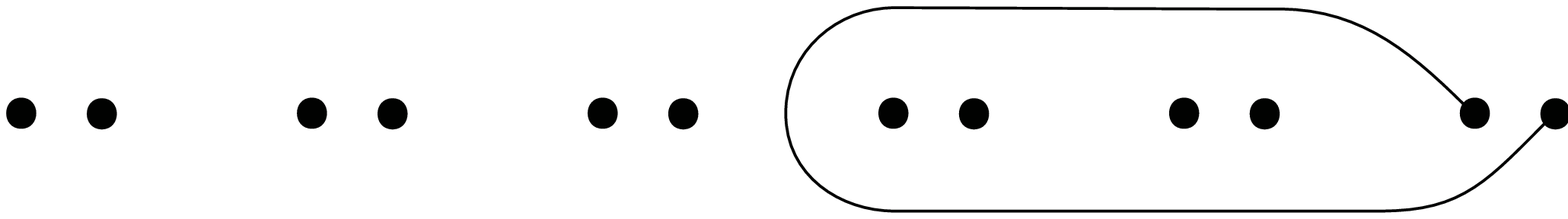}
\label{F:LVC beta1}}\hspace{.5em}
\subfigure[The Lefschetz vanishing cycle $\beta_6$.]{\includegraphics[width=60mm]{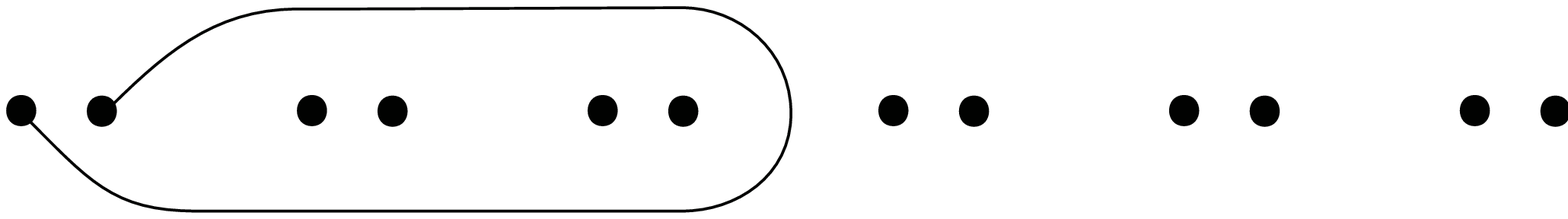}
\label{F:LVC beta6}}

\subfigure[The paths $\gamma_1,\ldots, \gamma_4$.]{\includegraphics[width=60mm]{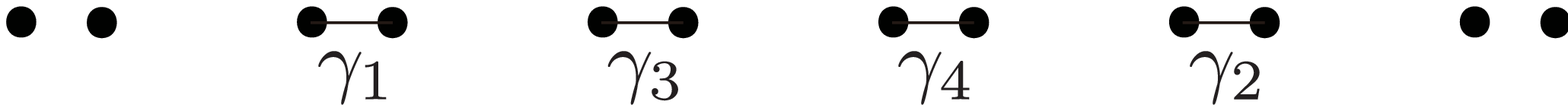}
\label{F:curve gammas}}\hspace{.5em}
\subfigure[The path $\beta$.]{\includegraphics[width=60mm]{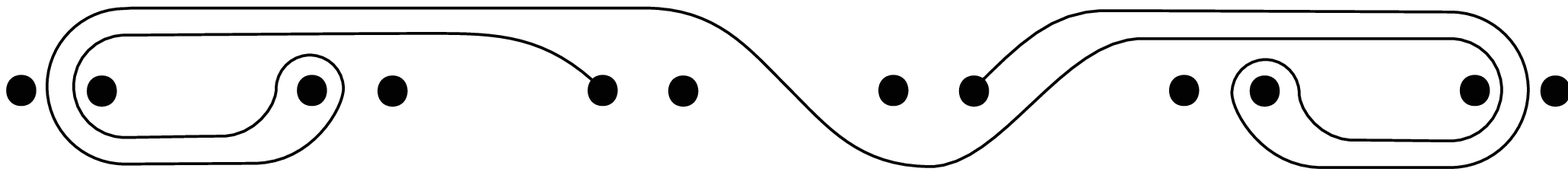}
\label{F:LVC beta}}
\caption{The paths around a type M $6$-point.}
\label{F:LVC for 6point rengeneration}
\end{figure}

\end{theorem}

\begin{remark}\label{R:extra branch point}

In general, a generic projection on a regenerated surface to $\PP^2$ might have branch points which do not appear around multiple points of the original line arrangement. 
Such branch points are called \emph{extra branch points} in \cite{RobbPhD,MRT}. 
According to Proposition 3.3.4 in \cite{RobbPhD}, there are no extra branch points in $U_n$, while there are two extra branch points in $U_s$ (cf.~\cite[Proposition 5.2.4]{MRT}). 
We will explain how to determine the braid monodromies of these branch points in \cref{S:VC pencil Us}. 

\end{remark}

\subsection{Vanishing cycles of a pencil of degree-$3$ curves in $\PP^2$}\label{S:VC pencil Un}

We will first calculate vanishing cycles of the Lefschetz pencil $f_n:U_n\dashrightarrow \PP^1$ given in \cref{T:genus-1 LP is blow-up of N_9 or S_8}. 
As shown in \cref{T:proj degen U_n}, we can take a sequence of projective degenerations from $U_n$ to a union of $9$ planes $Y^{(5)}$. 
Let $C_n$ be the union of all the lines in $Y^{(5)}$ appearing as intersections of two planes in $Y^{(5)}$. 
We denote the planes in $Y^{(5)}$, the lines in $C_n$ and the multiple points in $C_n$ by $\{P_i\}_{i=1}^9$, $\{l_j\}_{j=1}^9$ and $\{a_k\}_{k=1}^7$, respectively, as shown in \cref{F:plane arrangement Un}.
\begin{figure}[htbp]
\centering
\includegraphics[width=115mm]{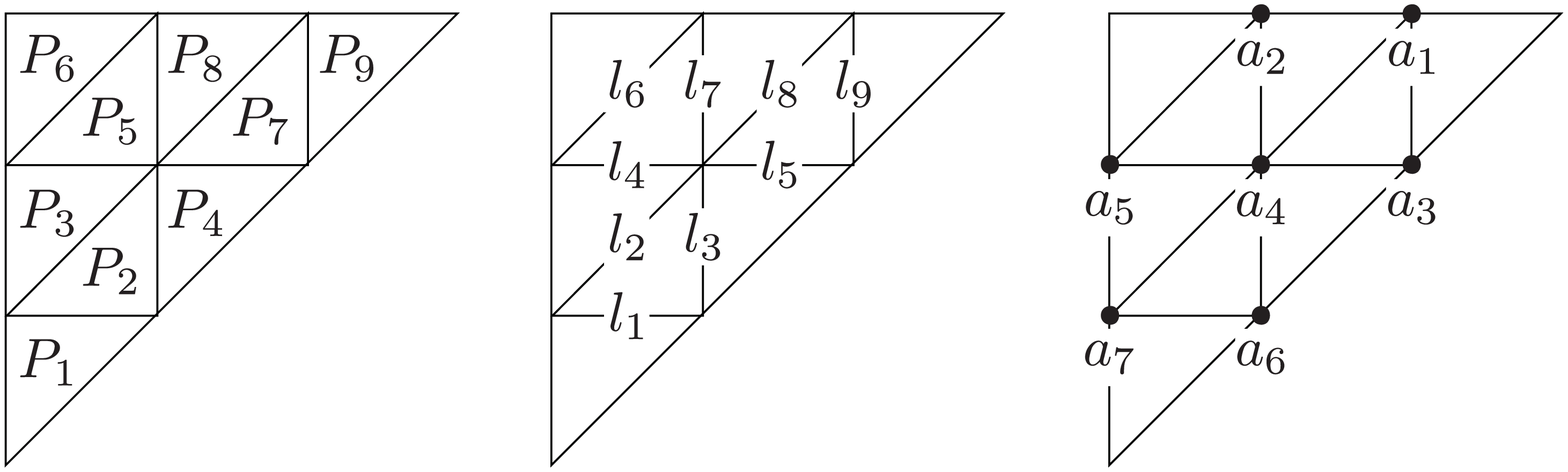}
\caption{Planes, lines and multiple points in $Y^{(5)}$.}
\label{F:plane arrangement Un}
\end{figure}
Note that all the multiple points in $C_n$ are $2$-points except for the unique type M $6$-point $a_4$. 
We can assume that $Y^{(5)}$ and $Y^{(0)}=U_n$ are both contained in $\PP^N$ and these are sufficiently close (cf.~\cite[\S.1]{MRT}). 
Take generic projections $\pi:\PP^N\dashrightarrow \PP^2$ and $\pi':\PP^2\dashrightarrow \PP^1$.
Let $\tilde{\pi}:U_n\to \PP^2$ be the restriction of $\pi$ on $U_n$ and $a_i' = \pi'\circ \pi(a_i)$. 
As observed in \cite{MTIV}, we can regard $C_n$ as a sub-arrangement of a line arrangement dual to generic introduced in \cite[Section IX]{MTI}. 
By \cite[Theorem IX.2.1]{MTI}, we can take a point $a_0'\in \PP^1$ away from $a_1',\ldots, a_7'$ and a simple path $\alpha_i'$ from $a_0'$ to $a_i'$ so that $\alpha_i'$'s are mutually disjoint except at the common initial point $a_0'$, the paths $\alpha_1',\ldots, \alpha_7'$ appear in this order when we go around $a_0'$ counterclockwise, and the Lefschetz vanishing cycles associated with the paths $\alpha_1',\ldots, \alpha_7'$ are as shown in \cref{F:LVCs for dual to generic ass with Un}, where the points labeled with $i$ is the intersection between the fiber $\overline{{\pi'}^{-1}(a_0')}\subset \PP^2$ and the line $\pi(l_i)$ ($i=1,\ldots, 9$). 
\begin{figure}[htbp]
\centering
\subfigure[The LVCs associated with $\alpha_1'$ and $\alpha_7'$.]{\includegraphics[width=60mm]{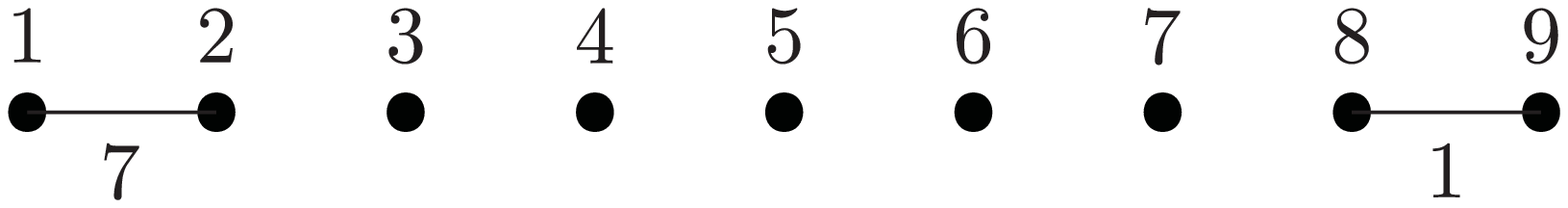}}\hspace{.6em}
\subfigure[The LVC associated with $\alpha_2'$.]{\includegraphics[width=60mm]{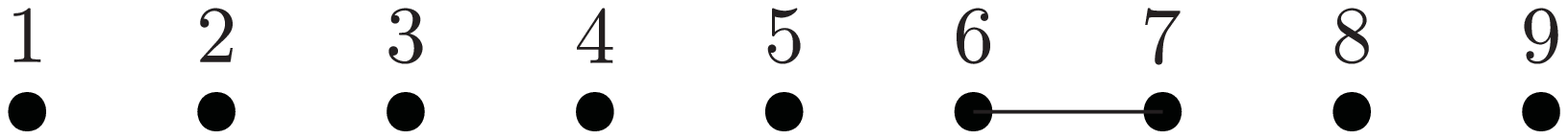}}

\subfigure[The LVC associated with $\alpha_3'$.]{\includegraphics[width=60mm]{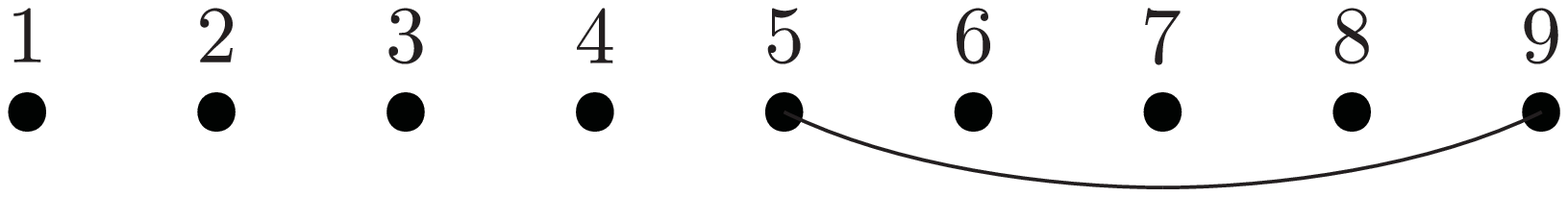}}\hspace{.6em}
\subfigure[The LVC associated with $\alpha_4'$.]{\includegraphics[width=60mm]{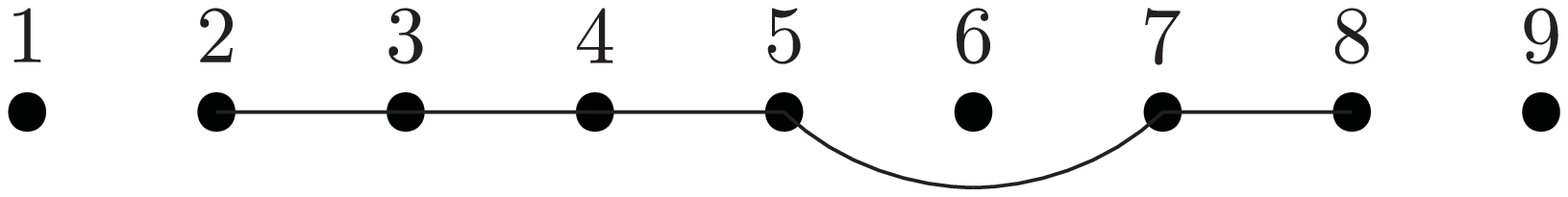}}

\subfigure[The LVC associated with $\alpha_5'$.]{\includegraphics[width=60mm]{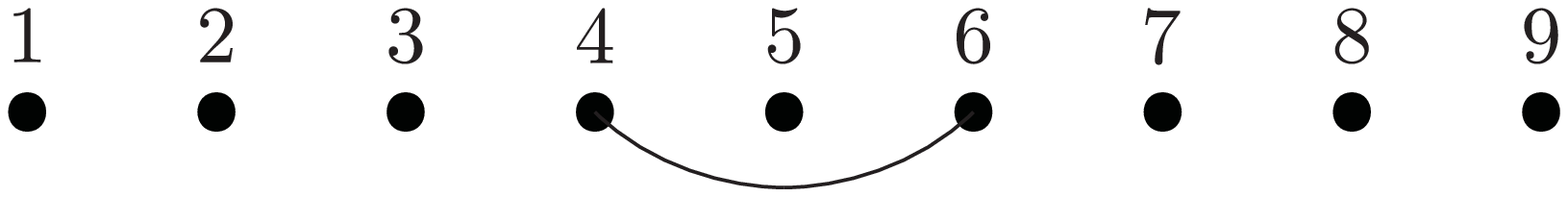}}\hspace{.6em}
\subfigure[The LVC associated with $\alpha_6'$.]{\includegraphics[width=60mm]{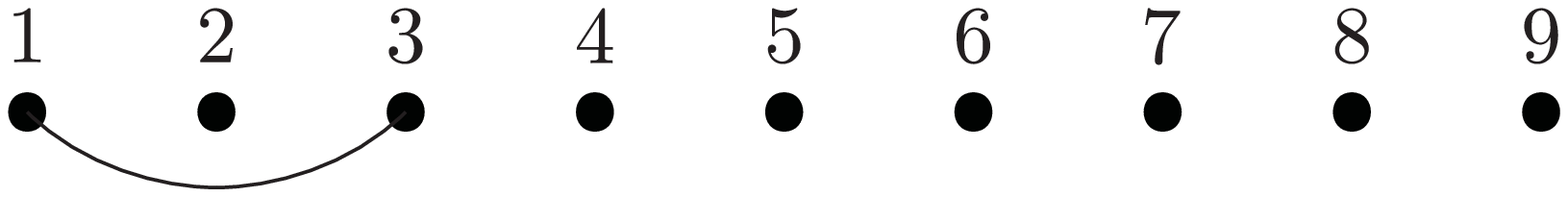}}

\caption{The Lefschetz vanishing cycles (LVC) of a line arrangement dual to points in general position.}
\label{F:LVCs for dual to generic ass with Un}
\end{figure}
We next apply \cref{T:regen rule 2point,T:regen rule 6point} in order to obtain the braid monodromies of the branch points of the restriction $\pi'|_{\widetilde{C_n}}\widetilde{C_n}\to \PP^1$, where $\widetilde{C_n}$ is the critical value set of $\tilde{\pi}:U_n\to \PP^2$. 
We eventually obtain a Hurwitz path system $(\alpha_1,\ldots, \alpha_{12})$ of $f_n$ ($=\pi'\circ \tilde{\pi}$) such that the Lefschetz vanishing cycles of the branch points associated with $\alpha_1,\ldots, \alpha_4,\alpha_9,\ldots,\alpha_{12}$ are as shown in \cref{F:LVC critical value set Un}, while those associated with $\alpha_5,\ldots,\alpha_8$ are respectively equal to
$\beta,\tau_{\gamma_3}^{-1}\tau_{\gamma_4}^{-1}(\beta),\tau_{\gamma_1}^{-1}\tau_{\gamma_2}^{-1}(\beta),\tau_{\gamma_1}^{-1}\tau_{\gamma_2}^{-1}\tau_{\gamma_3}^{-1}\tau_{\gamma_4}^{-1}(\beta)$, where the paths $\beta,\gamma_1,\ldots, \gamma_4$ are given in \cref{F:curves for LVC critv Un} and $\tau_{\gamma_i}$ is the half twist along $\gamma_i$.

In order to obtain vanishing cycles of $f_n$, we have to take the circle components of the preimages of the Lefschetz vanishing cycles under the branched covering
\begin{equation}\label{Eq:2-dim branched cover for Un}
\tilde{\pi}|_{\overline{f_n^{-1}(a_0')}}:\overline{f_n^{-1}(a_0')} \to \overline{{\pi'}^{-1}(a_0')}
\end{equation}
branched at $\overline{{\pi'}^{-1}(a_0')} \cap \widetilde{C_n}$. 
We denote the closure $\overline{{\pi'}^{-1}(a_0')}$ by $S$, the intersection $\overline{{\pi'}^{-1}(a_0')} \cap \widetilde{C_n}$ by $Q$. 
We take a point $q_0\in S\setminus Q$ and regard an element $\sigma$ in the symmetry group $\mathfrak{S}_9$ as a self-bijections of $\tilde{\pi}^{-1}(q_0)$ sending the point in $\tilde{\pi}^{-1}(q_0)$ close to $P_i$ to that close to $P_{\sigma(i)}$ for each $i=1,\ldots,9$ (note that we assumed that $U_n$ is sufficiently close to $Y^{(5)}$).
Let $\varrho:\pi_1(S\setminus Q, q_0) \to \mathfrak{S}_{9}$ be the monodromy representation of the branched covering \eqref{Eq:2-dim branched cover for Un}. 
As shown in \cref{F:path monodromy branch Un}, we take a system of oriented paths $\eta_1,\eta_1',\ldots, \eta_9,\eta_9'$ such that the common initial point of them is $q_0$ and the end point of $\eta_i$ (resp.~$\eta_i'$) is the points labeled with $i$ (resp.~$i'$). 
\begin{figure}[htbp]
\centering
\includegraphics[width=100mm]{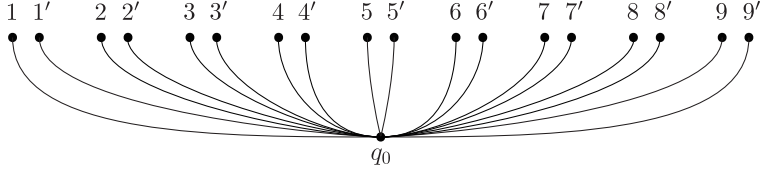}
\caption{The paths $\eta_1,\eta_1',\ldots, \eta_9,\eta_9'$ in $S$. 
In this figure these paths appear in this order when we go around $q_0$ clockwise.}
\label{F:path monodromy branch Un}
\end{figure}
Let $\widetilde{\eta_i}$ be a based loop in $S\setminus Q$ with the base point $q_0$ which can be obtained by connecting $q_0$ with a small clockwise circle around the point label with $i$ using $\eta_i$. 
We also take a based loop $\widetilde{\eta_i'}$ in a similar manner. 
The images $\varrho ([\widetilde{\eta_i}])$ and $\varrho ([\widetilde{\eta_i'}])$ are easily calculated as follows: 
\begin{align*}
& \varrho ([\widetilde{\eta_1}]) = \varrho ([\widetilde{\eta_1'}]) = (12), \hspace{.5em} \varrho ([\widetilde{\eta_2}]) = \varrho ([\widetilde{\eta_2'}]) = (23), \hspace{.5em}\varrho ([\widetilde{\eta_3}]) = \varrho ([\widetilde{\eta_3'}]) = (24), \\
& \varrho ([\widetilde{\eta_4}]) = \varrho ([\widetilde{\eta_4'}]) = (35), \hspace{.5em} \varrho ([\widetilde{\eta_5}]) = \varrho ([\widetilde{\eta_5'}]) = (47), \hspace{.5em}\varrho ([\widetilde{\eta_6}]) = \varrho ([\widetilde{\eta_6'}]) = (56), \\
& \varrho ([\widetilde{\eta_7}]) = \varrho ([\widetilde{\eta_7'}]) = (58), \hspace{.5em} \varrho ([\widetilde{\eta_8}]) = \varrho ([\widetilde{\eta_8'}]) = (78), \hspace{.5em}\varrho ([\widetilde{\eta_9}]) = \varrho ([\widetilde{\eta_9'}]) = (79).
\end{align*}
Note that all of these images are transpositions. 
We can thus describe the branched covering \eqref{Eq:2-dim branched cover for Un} as shown in \cref{F:2-dim branch cover for Un}. 
In this figure, the red circles in the upper surface $\overline{f_n^{-1}(a_0')}$ are the circle components of the preimages of the red paths between branch points in the lower sphere $\overline{{\pi'}^{-1}(a_0')}$. 
The point represented by $\times$ in the lower sphere is the base point of $\pi'$, while those in the upper surface are the preimages of it. 
As described in \cref{F:isotopy branch Un}, the complement of small disk neighborhoods of the base points of $f_n$ in the closure $\overline{f_n^{-1}(a_0')}$ is a nine-holed torus. 
The surface in \cref{F:isotopy branch Un shrinked} is obtained from \cref{F:isotopy branch Un original} by shrinking the subsurfaces labeled with $1,6$ and $9$. 
Those in \cref{F:isotopy branch Un shrinked,F:isotopy branch Un deformed} are homeomorphic to each other. 
Taking the preimages of the paths described in \cref{F:LVC critical value set Un,F:curves for LVC critv Un} under the branched covering described in \cref{F:2-dim branch cover for Un}, we can eventually obtain a monodromy factorization $t_{c_{12}}\circ \cdots \circ t_{c_1} = t_{\delta_1}\circ \cdots\circ t_{\delta_9}$ of $f_n$, where the simple closed curves $c_1,\ldots, c_5,c_9,\ldots, c_{12}$ are given in \cref{F:vanishing cycle Un}, while $c_6,c_7,c_8$ are respectively equal to $t_{d_3}^{-1}t_{d_4}^{-1}(c_5), t_{d_1}^{-1}t_{d_2}^{-1}(c_5), t_{d_1}^{-1}t_{d_2}^{-1}t_{d_3}^{-1}t_{d_4}^{-1}(c_5)$, where the simple closed curves $d_1,d_2,d_3$ and $d_4$ are given in \cref{F:vanishing cycle Un5}.

\subsection{Vanishing cycles of a pencil of curves with bi-degree-$(2,2)$ in $\PP^1\times \PP^1$}\label{S:VC pencil Us}

We will next calculate vanishing cycles of the Lefschetz pencil $f_s:U_s\dashrightarrow \PP^1$. 
Again, let $C_s$ be the union of all the lines in $Z^{(3)}$ appearing as intersections of two plane components, and denote the planes in $Z^{(3)}$, the lines in $C_s$ and the multiple points in $C_s$ by $\{P_s\}_{i=1}^8$, $\{l_j\}_{j=1}^8$ and $\{a_k\}_{k=2}^6$, respectively, as shown in \cref{F:plane arrangement Us}.
\begin{figure}[htbp]
\centering
\includegraphics[width=115mm]{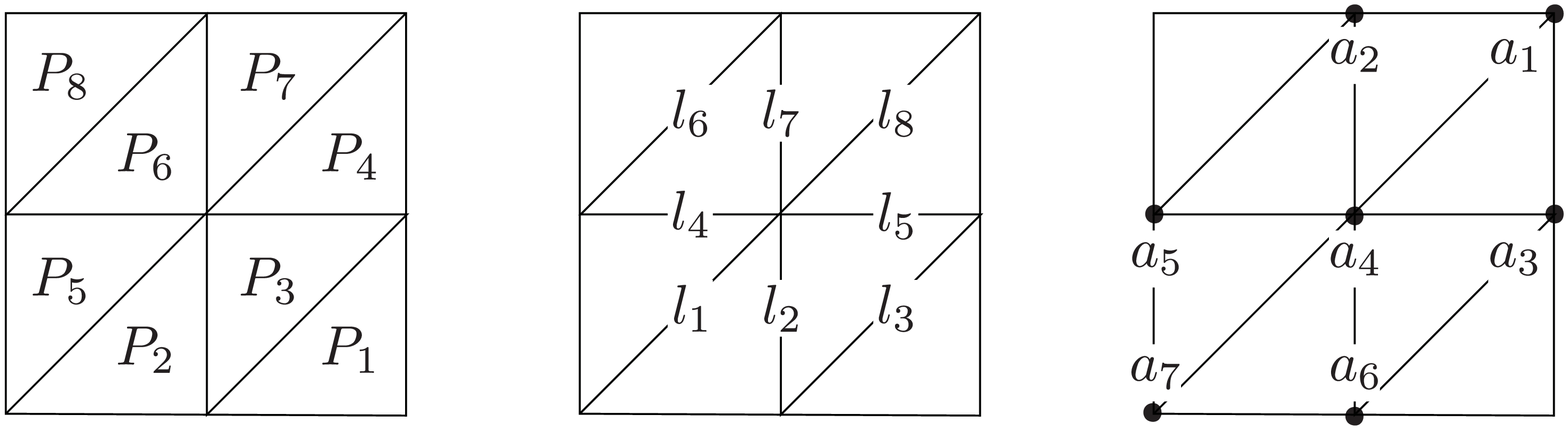}
\caption{Planes, lines and multiple points in $Z^{(3)}$.}
\label{F:plane arrangement Us}
\end{figure}
We further take points $a_1$ and $a_7$ on $l_8$ and $l_1$, respectively. 
Suppose that $Z^{(3)}$ and $U_s=Z^{(0)}$ are both contained in $\PP^N$ and these are sufficiently close. 
Moreover, without loss of generality, we can assume that the line arrangement $C_s$ is the same as that given in \cite[Theorem IX.2.1]{MTI} and the order of the lines in $C_s$ (given by indices) is the same as that in \cite[Theorem IX.2.1]{MTI} (meaning that the order of the vertices in $C_s$ is opposite to that in \cite[Theorem IX.2.1]{MTI}). 
As in the previous subsection, let $\pi:\PP^N\dashrightarrow \PP^2$ and $\pi':\PP^2\dashrightarrow \PP^1$ be generic projections, $\tilde{\pi}:U_s\to \PP^2$ be the restriction of $\pi$, $\widetilde{C_s}$ be the critical value set of $\tilde{\pi}$ and $a_i' = \pi'\circ \pi(a_i)$. 
Applying \cite[Theorem IX.2.1]{MTI}, we take a reference point $a_0'\in \PP^1$ and reference paths $\alpha_i'$ from $a_0'$ to $a_i'$ ($i=2,\ldots, 6$) so that the the corresponding Lefschetz vanishing cycles are as shown in \cref{F:LVCs for dual to generic ass with Us}. 
\begin{figure}[htbp]
\centering
\subfigure[The LVCs associated with $\alpha_2'$, $\alpha_3'$, $\alpha_5'$ and $\alpha_6'$.]{\includegraphics[width=60mm]{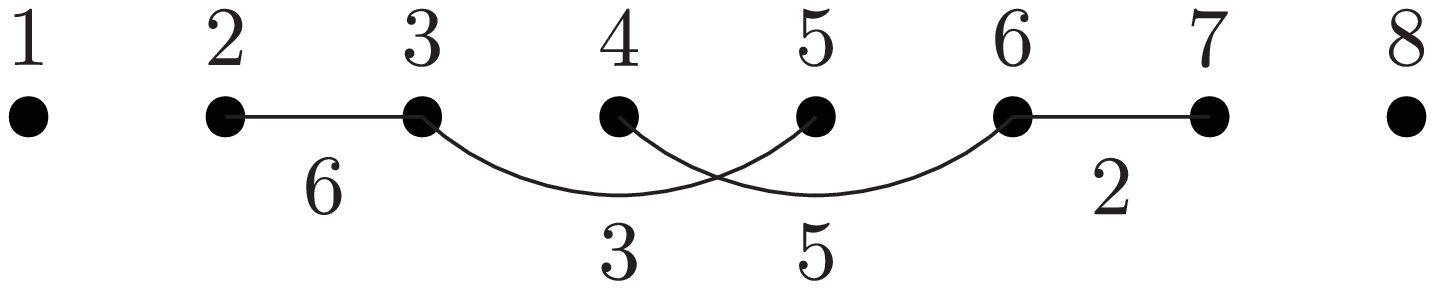}}\hspace{.6em}
\subfigure[The LVC associated with $\alpha_4'$.]{\includegraphics[width=60mm]{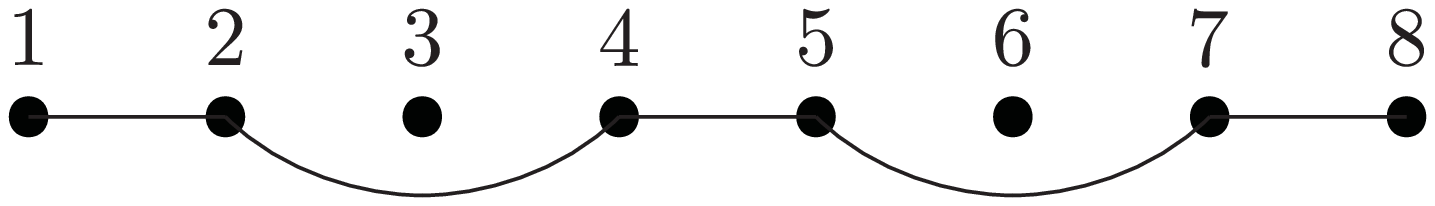}}
\caption{The Lefschetz vanishing cycles (LVC) of a line arrangement dual to points in general position.}
\label{F:LVCs for dual to generic ass with Us}
\end{figure}

As observed in \cref{R:extra branch point}, there are branch points of $\tilde{\pi}|_{\widetilde{C_s}}:\widetilde{C_s}\to \PP^1$ which are not close to multiple points of $C_s$.  
We take the regeneration from $Z^{(3)}$ to $Z^{(2)}$ so that the planes $P_4$ and $P_7$ (resp.~$P_2$ and $P_5$) are regenerated to $U_{1,1}$ going through the points $a_1,a_2,a_3,a_4$ (resp.~$a_4,a_5,a_6,a_7$). (See \cite[\S.3.5]{MTIII} for the detail of this regeneration.)
Analyzing the model of such a regeneration given in the proof of \cite[Proposition 14]{MTIII}, we can verify that the two extra branch points of $\tilde{\pi}|_{\widetilde{C_s}}$ appear around $a_1$ and $a_7$. 
We can further show that, for suitable reference paths $\alpha_1$ and $\alpha_{12}$ from $a_0'$ to the images of the branch points near $a_1$ and $a_7$, respectively, the Lefschetz vanishing cycles of the two extra branch points associated with $\alpha_1$ and $\alpha_{12}$ are as shown in \cref{F:LVC extrabranch Us} (cf.~\cite[\S.3.3]{RobbPhD}). 
By applying \cref{T:regen rule 2point,T:regen rule 6point}, we can take reference paths $\alpha_i$ ($i=2,\ldots,11$) so that $(\alpha_1,\ldots,\alpha_{12})$ is a Hurwitz path system of $f_s$, and the Lefschetz vanishing cycles of branch points of $\tilde{\pi}|_{\widetilde{C_s}}$ associated with $\alpha_2,\ldots,\alpha_4,\alpha_9,\alpha_{11}$ are as shown in , while those associated with $\alpha_5,\ldots,\alpha_8$ are respectively equal to $\beta,\tau_{\gamma_3}^{-1}\tau_{\gamma_4}^{-1}(\beta),\tau_{\gamma_1}^{-1}\tau_{\gamma_2}^{-1}(\beta),\tau_{\gamma_1}^{-1}\tau_{\gamma_2}^{-1}\tau_{\gamma_3}^{-1}\tau_{\gamma_4}^{-1}(\beta)$, where the paths $\beta,\gamma_1,\ldots, \gamma_4$ are given in \cref{F:LVC critv Us beta,F:LVC critv Us gamma}.
As in the previous subsection, we next consider the following branched covering: 
\begin{equation}\label{Eq:2-dim branched cover for Us}
\tilde{\pi}|_{\overline{f_s^{-1}(a_0')}}:\overline{f_n^{-1}(a_0')} \to \overline{{\pi'}^{-1}(a_0')}. 
\end{equation}
By calculating the monodromy representation of this covering, we can describe this branched covering as shown in \cref{F:2-dim branch cover for Us,F:isotopy branchcover Us}.
Taking the preimages of the paths described in \cref{F:LVC critical value set Us} under the branched covering described in \cref{F:2-dim branch cover for Us}, we can obtain a monodromy factorization $t_{c_{12}}\circ \cdots \circ t_{c_1} = t_{\delta_1}\circ \cdots\circ t_{\delta_8}$ of $f_s$, where the simple closed curves $c_1,\ldots, c_5,c_9,\ldots, c_{12}$ are given in \cref{F:vanishing cycle Us}, while $c_6,c_7,c_8$ are respectively equal to $t_{d_3}^{-1}t_{d_4}^{-1}(c_5), t_{d_1}^{-1}t_{d_2}^{-1}(c_5), t_{d_1}^{-1}t_{d_2}^{-1}t_{d_3}^{-1}t_{d_4}^{-1}(c_5)$, where the simple closed curves $d_1,d_2,d_3$ and $d_4$ are given in \cref{F:vanishing cycle Us5}.

\section{Combinatorial structures of genus-$1$ pencils}\label{sec:combinatorial structure}

In this section we study the combinatorial structures of the monodromy factorizations associated with the genus-$1$ holomorphic Lefschetz pencils. 
We will simplify those factorizations and show that they are Hurwitz equivalent to the known $k$-holed torus relations, which were combinatorially constructed by Korkmaz-Ozbagci~\cite{KorkmazOzbagci2008} and Tanaka~\cite{Tanaka2012}. In particular, we will see that a genus-$1$ holomorphic Lefschetz pencil obtained by blowing-up another holomorphic pencil is uniquely determined by the number of the blown-up points and independent of particular choices of such points. Thus, we complete the classification of genus-$1$ holomorphic Lefschetz pencils in the smooth category.

In the remainder of the paper, we simplify the notations regarding Dehn twists as follows.
We will denote the right-handed Dehn twist along a curve $\alpha$ also by $\alpha$, and its inverse, i.e., the left-handed Dehn twist along $\alpha$, by $\bar{\alpha}$.
We continue to use the functional notation for multiplication; $\beta \alpha$ means we first apply $\alpha$ and then $\beta$.
In addition, we denote the conjugation $\alpha \beta \bar{\alpha}$ by ${}_{\alpha}\!(\beta)$, which is the Dehn twist along the curve $t_{\alpha}(\beta)$.
Finally, we use the symbol $\partial_k$ to denote the boundary multi-twist $\delta_1 \delta_2 \cdots \delta_k$. 

\subsection{Monodromies of the minimal pencils}
We first deal with the minimal holomorphic Lefschetz pencils $f_n$ and $f_s$ as they are the base cases in the sense that the other holomorphic pencils are obtained by blowing-up those two pencils.
\subsubsection{Monodromy of $f_n$}
\begin{figure}[htbp]
	\centering
	\includegraphics[height=170pt]{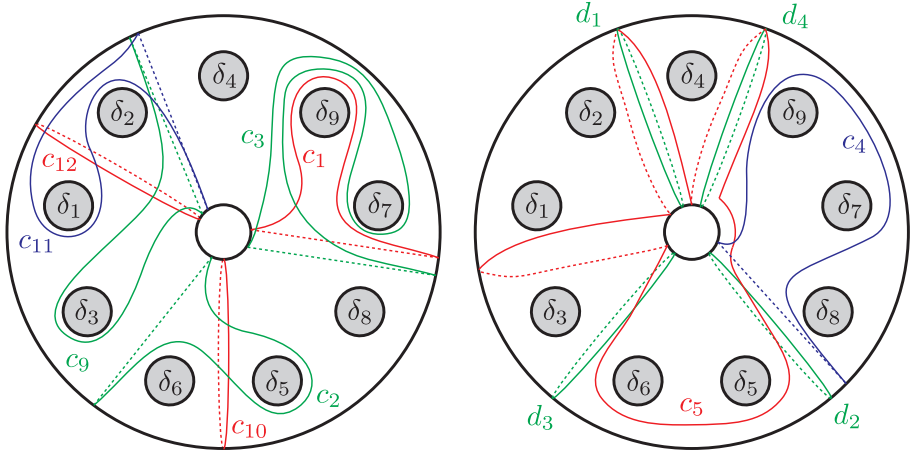}
	\caption{Redrawing of the vanishing cycles of $f_n$ on a standard $9$-holed torus $\Sigma_1^9$.}
	\label{F:k=9fn}
\end{figure}
\begin{figure}[htbp]
	\centering
	\includegraphics[height=170pt]{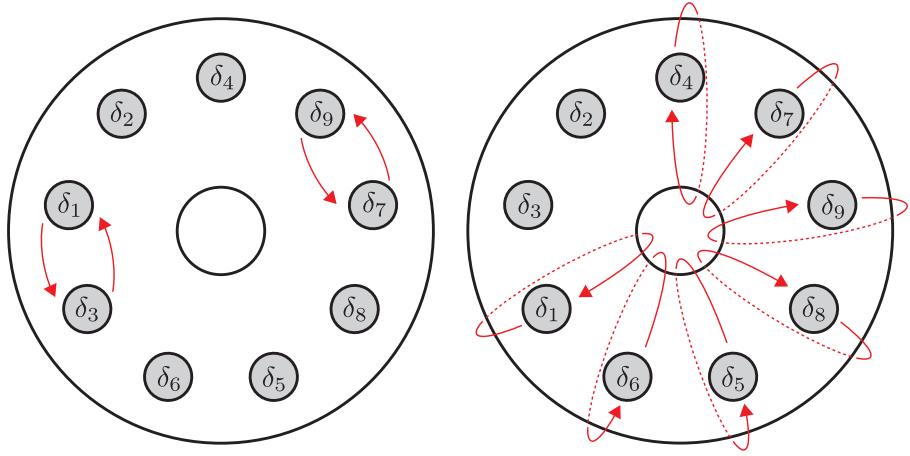}
	\caption{Pushing of boundary components.}
	\label{F:k=9transformation}
\end{figure}
\begin{figure}[htbp]
	\centering
	\includegraphics[height=170pt]{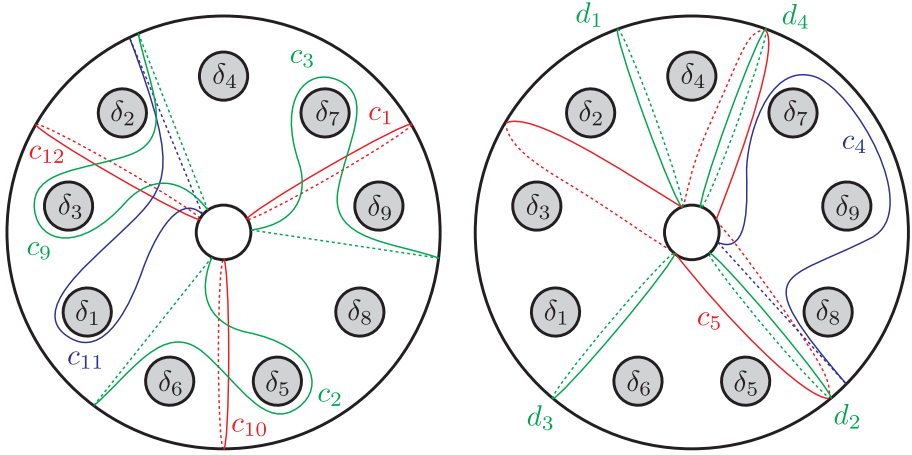}
	\caption{Vanishing cycles after repositioning the surface $\Sigma_1^9$.}
	\label{F:k=9fn_repositioned}
\end{figure}
In Section~\ref{S:VC pencil Un} we obtained a monodromy factorization of $f_n$,
\begin{align*}
c_{12} c_{11} c_{10} c_{9} c_{8} c_{7} c_{6} c_{5} c_{4} c_{3} c_{2} c_{1} = \delta_1 \delta_2\delta_3 \delta_4 \delta_5 \delta_6 \delta_7 \delta_8 \delta_9
\end{align*} 
with the vanishing cycles computed in Figure~\ref{F:vanishing cycle Un} where $c_6 = {}_{\bar{d}_3 \bar{d}_4}(c_5)$, $c_7 = {}_{\bar{d}_1 \bar{d}_2}(c_5)$, and $c_8 = {}_{\bar{d}_1 \bar{d}_2 \bar{d}_3 \bar{d}_4}(c_5)$. The curves are redrawn on a standardly positioned torus in Figure~\ref{F:k=9fn}.
We further reposition the surface by pushing the boundary components as indicated in Figure~\ref{F:k=9transformation}; we first swap $\delta_1$ and $\delta_3$, also $\delta_7$ and $\delta_9$, then push the boundary components except for $\delta_2$ and $\delta_3$ along the meridian in the indicated directions.
Accordingly, the vanishing cycles are now configured as in Figure~\ref{F:k=9fn_repositioned}.
We further modify the factorization by Hurwitz moves.
\begin{align*}
\delta_1 \delta_2\delta_3 \delta_4 \delta_5 \delta_6 \delta_7 \delta_8 \delta_9
&= c_{12} c_{11} c_{10} c_{9} c_{8} c_{7} c_{6} c_{5} c_{4} c_{3} c_{2} c_{1} \\
&\sim c_{11} \underline{c_9 c_8} \; \underline{c_7 c_6 c_5 c_{12}} c_4 c_3 c_1 c_2 c_{10} \\
&\sim c_{11} c_8^\prime c_9 c_{12} {}_{\bar{c}_{12}}\!(c_7) {}_{\bar{c}_{12}}\!(c_6) \underline{c_5^\prime c_4} c_3 c_1 c_2 c_{10} \\
&\sim c_{11} c_8^\prime c_9 c_{12} {}_{\bar{c}_{12}}\!(c_7) \underline{{}_{\bar{c}_{12}}\!(c_6) c_4^\prime} c_5^\prime c_3 c_1 c_2 c_{10} \\
&\sim c_{11} c_8^\prime c_9 c_{12} {}_{\bar{c}_{12}}\!(c_7) c_4^\prime \underline{{}_{\bar{c}_4^\prime \bar{c}_{12}}\!(c_6) c_5^\prime c_3 c_1} c_2 c_{10} \\
&\sim 
\underline{c_{11} c_8^\prime} c_9 c_{12} 
{}_{\bar{c}_{12}}\!(c_7) c_4^\prime c_3 c_1 
\underline{{}_{\bar{c}_1 \bar{c}_3 \bar{c}_4^\prime \bar{c}_{12}}\!(c_6) c_5^{\prime\prime}} c_2 c_{10} \\
&\sim 
c_8^\prime \underline{c_{11}^\prime c_9} c_{12} 
{}_{\bar{c}_{12}}\!(c_7) \underline{c_4^\prime c_3} c_1 
c_5^{\prime\prime} \underline{c_6^{\prime} c_2} c_{10} \\
&\sim 
c_8^\prime {}_{c_{11}^\prime}\!(c_9) \underline{c_{11}^\prime c_{12}}
{}_{\bar{c}_{12}}\!(c_7) {}_{c_4^\prime}\!(c_3) \underline{c_4^\prime c_1}
c_5^{\prime\prime} {}_{c_6^{\prime}}\!(c_2) \underline{c_6^{\prime} c_{10}} \\
&\sim 
c_8^\prime {}_{c_{11}^\prime}\!(c_9) c_{12} {}_{\bar{c}_{12}}\!(c_{11}^\prime) 
{}_{\bar{c}_{12}}\!(c_7) {}_{c_4^\prime}\!(c_3) c_1 {}_{\bar{c}_1}\!(c_4^\prime) 
c_5^{\prime\prime} {}_{c_6^{\prime}}\!(c_2) c_{10} {}_{\bar{c}_{10}}\!(c_6^{\prime})  
\end{align*}
where $c_8^\prime ={}_{c_9}\!(c_8)$, $c_5^\prime = {}_{\bar{c}_{12}}\!(c_5)$, $c_4^\prime = {}_{c_5^\prime}\!(c_4)$, $c_5^{\prime\prime} = {}_{\bar{c}_1 \bar{c}_3}\!(c_5^\prime)$, $c_{11}^\prime = {}_{\bar{c}_8^\prime}(c_{11})$, $c_6^\prime = {}_{\bar{c}_5^{\prime\prime} \bar{c}_1 \bar{c}_3 \bar{c}_4^\prime \bar{c}_{12}}\!(c_6)$. 
It is routine to observe that the resulting curves are as depicted in Figure~\ref{F:k=9simple} and the last expression is $a_1 b_1 b_2 b_3 a_4 b_4 b_5 b_6 a_7 b_7 b_8 b_9$ up to labeling and a permutation.
Thus, we obtain the simpler monodromy factorization of $f_n$:
\begin{align}
a_1 b_1 b_2 b_3 a_4 b_4 b_5 b_6 a_7 b_7 b_8 b_9 = \delta_1 \delta_2\delta_3 \delta_4 \delta_5 \delta_6 \delta_7 \delta_8 \delta_9, \label{eq:k=9simple}
\end{align}
We refer to this relation as $N_9 =\partial_9$.
\begin{figure}[htbp]
	\centering
	\subfigure[The simplified relation $N_9 = \partial_9$. 
	\label{F:k=9simple}]
	{\includegraphics[height=170pt]{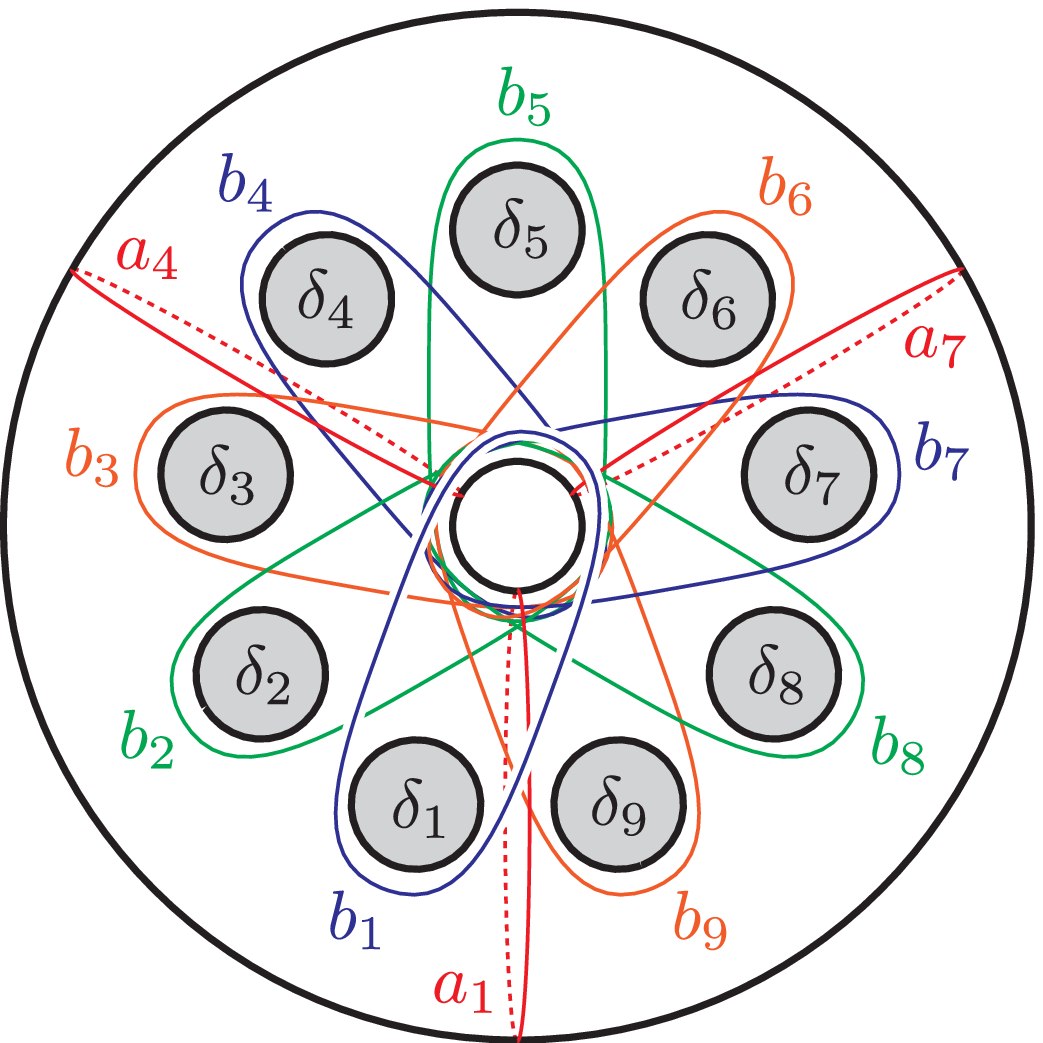}} 
	\hspace{.8em}	
	\subfigure[The original relation. 
	\label{F:k=9KO}]
	{\includegraphics[height=170pt]{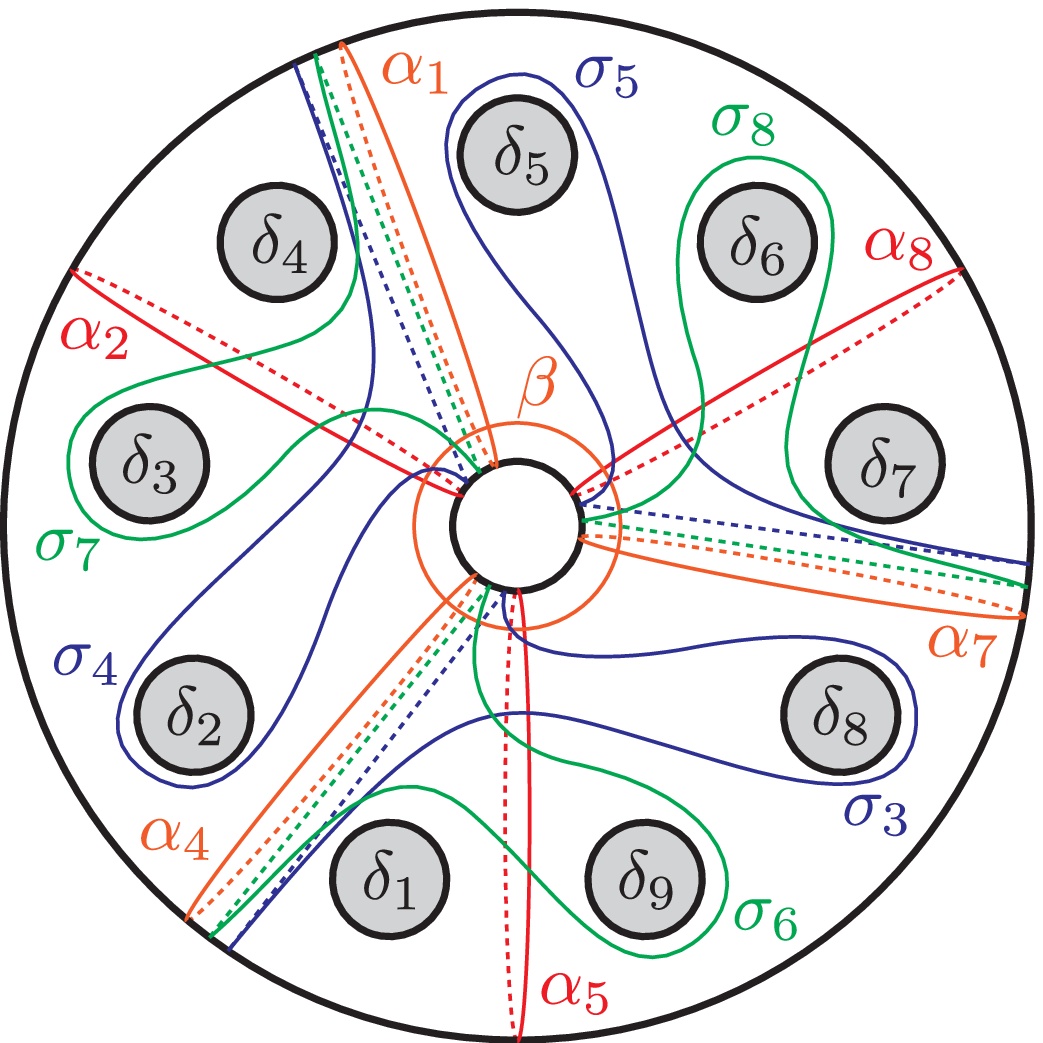}} 
	\caption{The curves for Korkmaz-Ozbagci's $9$-holed torus relation.} 	
\end{figure}

We are now ready for proving the following: 

\begin{theorem}\label{T:isom Korkmaz-Ozbagci}

The monodromy factorization $N_9 =\partial_9$ is Hurwitz equivalent to Korkmaz-Ozbagci's $9$-holed torus relation given in \cite{KorkmazOzbagci2008}.

\end{theorem}

\begin{proof}
With the curves shown in Figure~\ref{F:k=9KO}, Korkmaz and Ozbagci~\cite{KorkmazOzbagci2008} gave the $9$-holed torus relation
\begin{align}
\beta_4 \sigma_3 \sigma_6 \alpha_5 \beta_1 \sigma_4 \sigma_7 \alpha_2 \beta_7 \sigma_5 \sigma_8 \alpha_8
= \delta_1 \delta_2\delta_3 \delta_4 \delta_5 \delta_6 \delta_7 \delta_8 \delta_9,
\end{align}
where $\beta_4 = {}_{\alpha_4}\!(\beta)$, $\beta_1 = {}_{\alpha_1}\!(\beta)$ and $\beta_7 = {}_{\alpha_7}\!(\beta)$.
We modify this relation as follows:
\begin{align*}
\delta_1 \delta_2\delta_3 \delta_4 \delta_5 \delta_6 \delta_7 \delta_8 \delta_9
&= \underline{\beta_4 \sigma_3 \sigma_6} \alpha_5 \underline{\beta_1 \sigma_4 \sigma_7} \alpha_2 \underline{\beta_7 \sigma_5 \sigma_8} \alpha_8 \\
&\sim {}_{\beta_4}\!(\sigma_3) {}_{\beta_4}\!(\sigma_6) \underline{\beta_4 \alpha_5} {}_{\beta_1}\!(\sigma_4) {}_{\beta_1}\!(\sigma_7) \underline{\beta_1 \alpha_2} {}_{\beta_7}\!(\sigma_5) {}_{\beta_7}\!(\sigma_8) \underline{\beta_7 \alpha_8} \\
&\sim {}_{\beta_4}\!(\sigma_3) {}_{\beta_4}\!(\sigma_6) \alpha_5 {}_{\bar{\alpha}_5}\!(\beta_4) {}_{\beta_1}\!(\sigma_4) {}_{\beta_1}\!(\sigma_7) \alpha_2 {}_{\bar{\alpha}_2}\!(\beta_1) {}_{\beta_7}\!(\sigma_5) {}_{\beta_7}\!(\sigma_8) \alpha_8 {}_{\bar{\alpha}_8}\!(\beta_7) \\
&\sim \alpha_5 {}_{\bar{\alpha}_5}\!(\beta_4) {}_{\beta_1}\!(\sigma_4) {}_{\beta_1}\!(\sigma_7) \alpha_2 {}_{\bar{\alpha}_2}\!(\beta_1) {}_{\beta_7}\!(\sigma_5) {}_{\beta_7}\!(\sigma_8) \alpha_8 {}_{\bar{\alpha}_8}\!(\beta_7) {}_{\beta_4}\!(\sigma_3) {}_{\beta_4}\!(\sigma_6).
\end{align*}
It is straightforward to see that the last expression coincides with the factorization $N_9$.
\end{proof}

\subsubsection{Monodromy of $f_s$}
\begin{figure}[htbp]
	\centering
	\includegraphics[height=170pt]{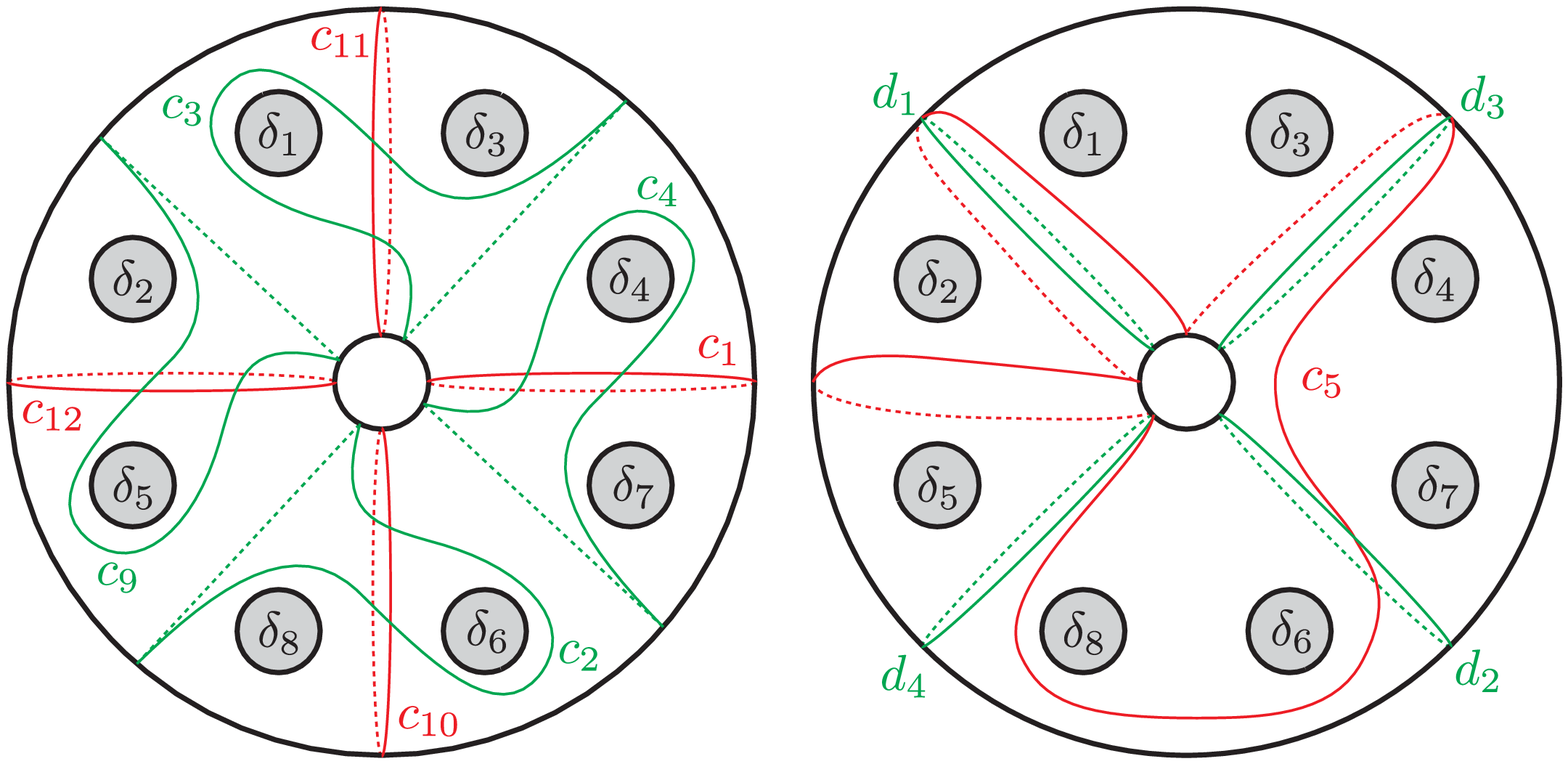}
	\caption{Redrawing of the vanishing cycles of $f_s$ on a standard $8$-holed torus $\Sigma_1^8$.}
	\label{F:k=8fs}
\end{figure}
\begin{figure}[htbp]
	\centering
	\includegraphics[height=170pt]{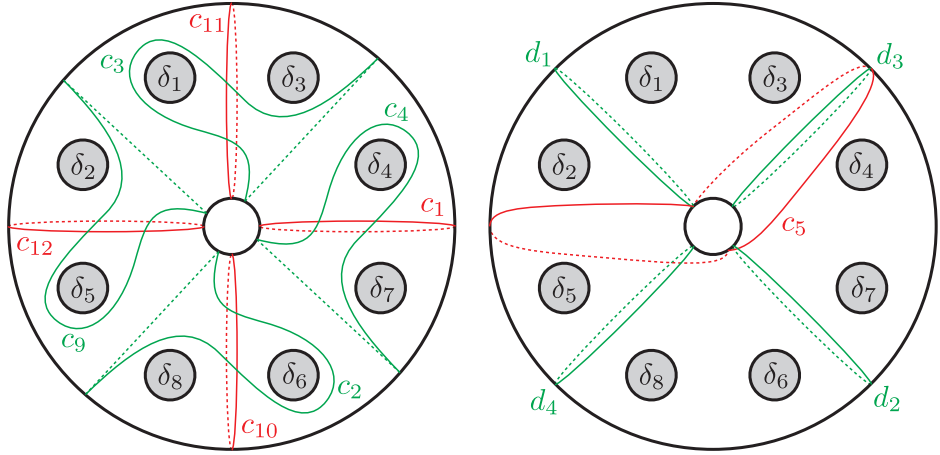}
	\caption{Vanishing cycles after the global conjugation by $\bar{d}_1 \bar{d}_4 d_2$.}
	\label{F:k=8fs_repositioned}
\end{figure}
A monodromy factorization of $f_s$ was computed in Section~\ref{S:VC pencil Us} as
\begin{align*}
c_{12} c_{11} c_{10} c_{9} c_{8} c_{7} c_{6} c_{5} c_{4} c_{3} c_{2} c_{1} = \delta_1 \delta_2\delta_3 \delta_4 \delta_5 \delta_6 \delta_7 \delta_8 
\end{align*} 
with the vanishing cycles found in Figure~\ref{F:vanishing cycle Us} where $c_6 = {}_{\bar{d}_3 \bar{d}_4}(c_5)$, $c_7 = {}_{\bar{d}_1 \bar{d}_2}(c_5)$, and $c_8 = {}_{\bar{d}_1 \bar{d}_2 \bar{d}_3 \bar{d}_4}(c_5)$. The curves are redrawn on a standardly positioned torus in Figure~\ref{F:k=8fs}.
We perform the global conjugation by $\bar{d}_1 \bar{d}_4 d_2$ to put the vanishing cycles as in Figure~\ref{F:k=8fs_repositioned}.
For simplicity, we keep using the same labeling $c_i$ for the resulting curves.
We then transform the factorization as follows.

\begin{align*}
\delta_1 \delta_2\delta_3 \delta_4 \delta_5 \delta_6 \delta_7 \delta_8 
&= c_{12} c_{11} c_{10} c_{9} c_{8} c_{7} c_{6} c_{5} c_{4} c_{3} c_{2} c_{1} \\
&\sim c_4 c_1 c_2 \underline{c_{10} c_{9} c_{8} c_{7}} c_{6} \underline{c_{5} c_{12} c_{3}} c_{11} \\
&\sim c_4 \underline{c_1 c_2 c_{8}^\prime {}_{c_{10} c_9}\!(c_7)} c_{10} c_{9} c_{6} c_{12} c_3 c_5^\prime c_{11} \\
&\sim 
\underline{c_4 c_7^{\prime} c_1} \cdot 
\underline{c_2 c_8^{\prime} c_{10}} \cdot 
\underline{c_{9} c_{6} c_{12}} \cdot
\underline{c_{3} c_5^\prime c_{11}} \\
&\sim 
c_7^{\prime} c_1 {}_{\bar{c}_1 \bar{c}_7^\prime}\!(c_4) \cdot 
c_8^{\prime} c_{10} {}_{\bar{c}_{10} \bar{c}_8^\prime}\!(c_2) \cdot 
c_{6} c_{12} {}_{\bar{c}_{12} \bar{c}_6}\!(c_{9}) \cdot
c_5^\prime c_{11} {}_{\bar{c}_{11} \bar{c}_5^\prime}\!(c_{3}) 
\end{align*}
where 
$c_5^\prime = {}_{\bar{c}_3 \bar{c}_{12}}\!(c_5)$, 
$c_8^\prime ={}_{c_{10}c_9}\!(c_8)$, 
$c_7^{\prime}={}_{c_1 c_2 c_8^\prime c_{10} c_9}\!(c_7)$.
The resulting curves are as depicted in Figure~\ref{F:k=8Tsimple} and the last expression is $a_1 b_1 b_2 a_3 b_3 b_4 a_5 b_5 b_6 a_7 b_7 b_8$ up to labeling and a permutation.
Thus, we obtain the simpler monodromy factorization of $f_s$:
\begin{align}
a_1 b_1 b_2 a_3 b_3 b_4 a_5 b_5 b_6 a_7 b_7 b_8 = \delta_1 \delta_2\delta_3 \delta_4 \delta_5 \delta_6 \delta_7 \delta_8, \label{eq:k=8Tsimple}
\end{align}
We refer to this relation as $S_8 = \partial_8$.
\begin{figure}[htbp]
	\centering
	\subfigure[The simplified relation $S_8 = \partial_8$. 
	\label{F:k=8Tsimple}]
	{\includegraphics[height=170pt]{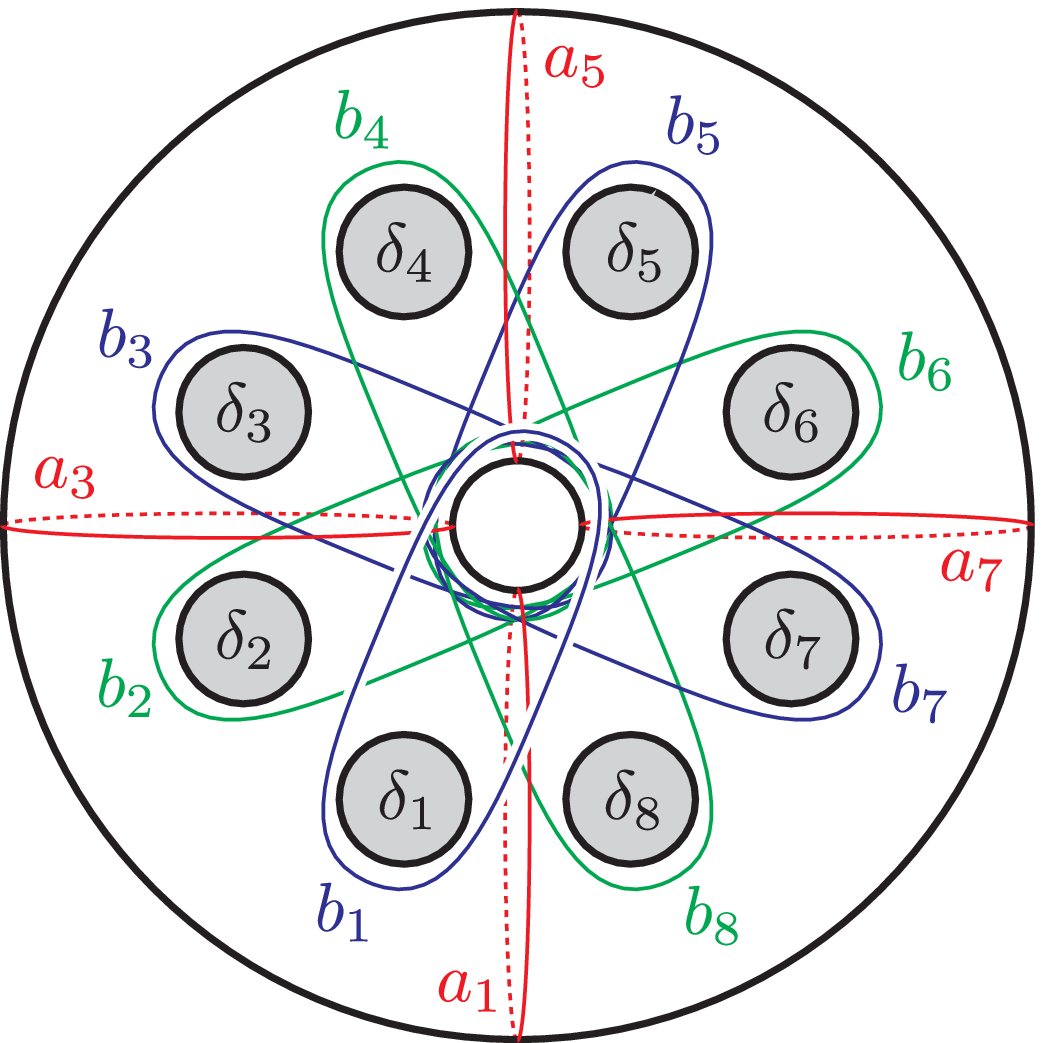}} 
	\hspace{.8em}	
	\subfigure[The original relation. 
	\label{F:k=8T}]
	{\includegraphics[height=170pt]{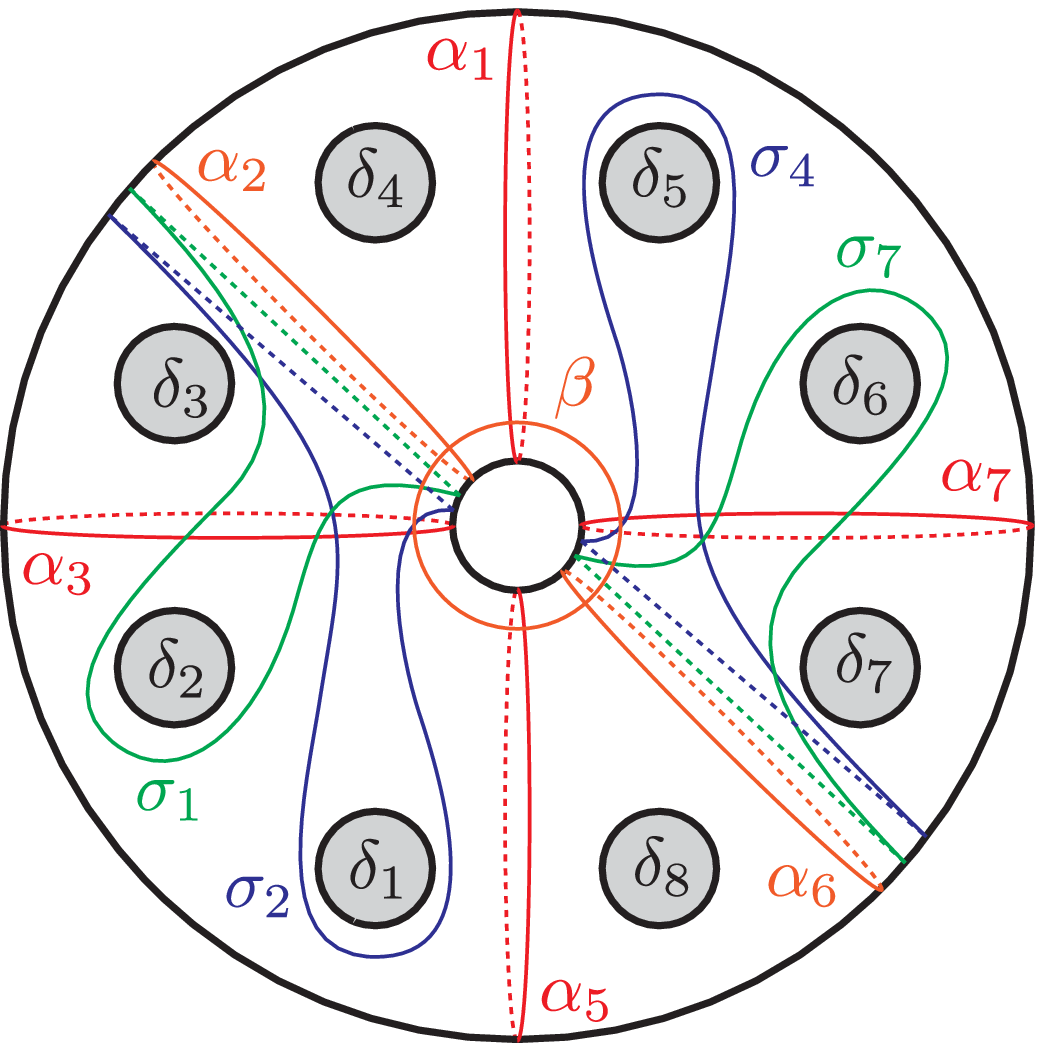}} 
	\caption{The curves for Tanaka's $8$-holed torus relation.} 	
\end{figure}

\begin{theorem}\label{T:isom Tanaka}

The monodromy factorization $S_8 = \partial_8$ is Hurwitz equivalent to Tanaka's $8$-holed torus relation given in \cite{Tanaka2012}.

\end{theorem}

\begin{proof}
With the curves shown in Figure~\ref{F:k=8T}, Tanaka~\cite{Tanaka2012} gave the $8$-holed torus relation
\begin{align}
\alpha_5 \alpha_7 \beta_{\bar{6}} \beta_2 \sigma_2 \sigma_1 \alpha_1 \alpha_3 \beta_{\bar{2}} \beta_6 \sigma_4 \sigma_7
= \delta_1 \delta_2\delta_3 \delta_4 \delta_5 \delta_6 \delta_7 \delta_8,
\end{align}
where $\beta_{\bar{6}} = {}_{\bar{\alpha}_6}\!(\beta)$, $\beta_2 = {}_{\alpha_2}\!(\beta)$, $\beta_{\bar{2}} = {}_{\bar{\alpha}_2}\!(\beta)$ and $\beta_6 = {}_{\alpha_6}\!(\beta)$.
We modify this relation as follows:
\begin{align*}
\delta_1 \delta_2\delta_3 \delta_4 \delta_5 \delta_6 \delta_7 \delta_8 
&= \underline{\alpha_5 \alpha_7} \beta_{\bar{6}} \beta_2 \sigma_2 \sigma_1 \underline{\alpha_1 \alpha_3} \beta_{\bar{2}} \beta_6 \sigma_4 \sigma_7 \\
&\sim \alpha_7 \underline{\alpha_5 \beta_{\bar{6}}} \beta_2 \sigma_2 \sigma_1 \alpha_3 \underline{\alpha_1 \beta_{\bar{2}}} \beta_6 \sigma_4 \sigma_7 \\
&\sim \alpha_7 {}_{\alpha_5}\!(\beta_{\bar{6}}) \alpha_5 \underline{\beta_2 \sigma_2 \sigma_1} \alpha_3  {}_{\alpha_1}\!(\beta_{\bar{2}}) \alpha_1 \underline{\beta_6 \sigma_4 \sigma_7} \\
&\sim \alpha_7 {}_{\alpha_5}\!(\beta_{\bar{6}}) \alpha_5  {}_{\beta_2}\!(\sigma_2) {}_{\beta_2}\!(\sigma_1) \beta_2 \alpha_3  {}_{\alpha_1}\!(\beta_{\bar{2}}) \alpha_1  {}_{\beta_6}\!(\sigma_4) {}_{\beta_6}\!(\sigma_7) \beta_6 \\
&\sim \alpha_5  {}_{\beta_2}\!(\sigma_2) {}_{\beta_2}\!(\sigma_1) \underline{\beta_2 \alpha_3}  {}_{\alpha_1}\!(\beta_{\bar{2}}) \alpha_1  {}_{\beta_6}\!(\sigma_4) {}_{\beta_6}\!(\sigma_7) \underline{\beta_6 \alpha_7} {}_{\alpha_5}\!(\beta_{\bar{6}}) \\
&\sim \alpha_5  {}_{\beta_2}\!(\sigma_2) {}_{\beta_2}\!(\sigma_1)  \alpha_3  {}_{\bar{\alpha}_3}\!(\beta_2)  {}_{\alpha_1}\!(\beta_{\bar{2}}) \alpha_1  {}_{\beta_6}\!(\sigma_4) {}_{\beta_6}\!(\sigma_7) \alpha_7 {}_{\bar{\alpha}_7}\!(\beta_6)  {}_{\alpha_5}\!(\beta_{\bar{6}}).
\end{align*}
The last expression coincides with $S_8$.
\end{proof}

\subsection{Monodromy and uniqueness of the non-minimal pencils}
By blowing-up some of the base points of $f_n$ or $f_s$ we obtain a non-minimal holomorphic Lefschetz pencil. In terms of monodromy factorization this corresponds to capping boundary components of $N_9 = \partial_9$ or $S_8 = \partial_8$ with disks and obtaining a $k$-holed torus relation with smaller $k$. The question is whether the resulting pencil is (smoothly) determined only by the number of blow-ups and independent of a particular set of base points that we blow-up. We prove that the answer is affirmative by providing a ``standard" $k$-holed torus relation $N_k = \partial_k$ for each $k \leq 8$ and showing the blow-up of any one base point of $N_k = \partial_k$, or additionally $S_8 = \partial_8$ when $k=8$, is Hurwitz equivalent to $N_{k-1}$.

The next lemma summarizes the techniques that we will repeatedly use in the Hurwitz equivalence computations. 
\begin{lemma} \label{lem:CommonTechniques}
	Consider the curves $a_i, b_i, b$ in the $k$-holed torus $\Sigma_1^k$ as in Figure~\ref{F:k=general}. Then the following relations between Dehn twists in $\MCG(\Sigma_1^k)$ are achieved by Hurwitz moves.
	\begin{enumerate}
		\item $b b_i \sim b_i b$, $a_i a_j \sim a_j a_i$.
		\item $a_i b a_i \sim b a_i b$.
		\item $b a_i b_i \sim a_i b_i a_{i+1} \sim b_i a_{i+1} b \sim a_{i+1} b a_i$.
	\end{enumerate}
	Here the indices are taken modulo $k$.
\end{lemma}
The verification is easy.

\subsubsection{One-time blow-up of $f_n$ and the $8$-holed torus relation $N_8 = \partial_8$}
\label{section:N_8}
We consider the $9$-holed torus relation $N_9 = \partial_9$ with the curves in Figure~\ref{F:k=9simple} (or Figure~\ref{F:k=general}) and cap one of the boundary components.

\noindent
{\bf Case 1}: Capping $\delta_9$.
This yields the relation 
\begin{align}
a_1 b_1 b_2 b_3 a_4 b_4 b_5 b_6 a_7 b_7 b_8 b = \delta_1 \delta_2\delta_3 \delta_4 \delta_5 \delta_6 \delta_7 \delta_8
\label{eq:N9WithDelta9CappedOff}
\end{align}
in $\MCG(\Sigma_1^8)$ where the curves are now understood to lie in $\Sigma_1^8$ as in Figure~\ref{F:k=general} with $k=8$ (see also Figure~\ref{F:k=8KOsimple}).
Notice that the curve $b_9$ becomes the central longitude $b$ as the boundary $\delta_9$ disappears.
We modify the relation as follows.
\begin{align*}
\delta_1 \delta_2\delta_3 \delta_4 \delta_5 \delta_6 \delta_7 \delta_8 
&= a_1 b_1 b_2 b_3 a_4 b_4 b_5 b_6 a_7 b_7 b_8 b \\
&\sim \underline{b a_1 b_1} b_2 b_3 a_4 b_4 b_5 b_6 a_7 b_7 b_8 \\
&\sim a_1 b_1 a_2 b_2 b_3 a_4 b_4 b_5 b_6 a_7 b_7 b_8.
\end{align*}
So we have the $8$-holed torus relation
\begin{align}
a_1 b_1 a_2 b_2 b_3 a_4 b_4 b_5 b_6 a_7 b_7 b_8 = \delta_1 \delta_2\delta_3 \delta_4 \delta_5 \delta_6 \delta_7 \delta_8, \label{eq:k=8simple}
\end{align}
to which we refer as $N_8 = \partial_8$. 
\begin{figure}[htbp]
	\centering
	\includegraphics[height=170pt]{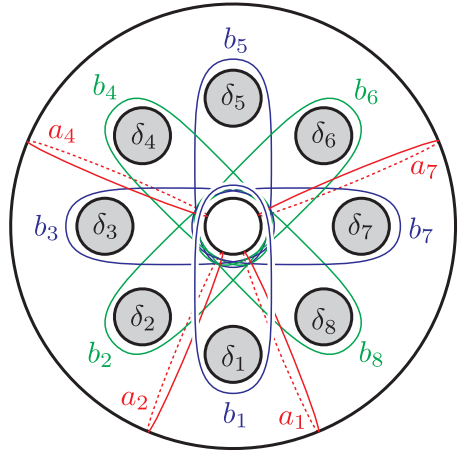}
	\caption{The relation $N_8 = \partial_8$.}
	\label{F:k=8KOsimple}
\end{figure}

\noindent
{\bf Case 2}: Capping $\delta_8$ or $\delta_7$.
Instead of $\delta_9$, we now cap $\delta_8$ or $\delta_7$ of $N_9 = \partial_9$. Then, after relabeling the curves so that they match the curves in Figure~\ref{F:k=general} with $k=8$, we have
\begin{align*}
a_1 b_1 b_2 b_3 a_4 b_4 b_5 b_6 a_7 b_7 \underline{b b_8} &= \delta_1 \delta_2\delta_3 \delta_4 \delta_5 \delta_6 \delta_7 \delta_8, \\
a_1 b_1 b_2 b_3 a_4 b_4 b_5 b_6 a_7 \underline{b b_7 b_8} &= \delta_1 \delta_2\delta_3 \delta_4 \delta_5 \delta_6 \delta_7 \delta_8. 
\end{align*}
Both of them are clearly equivalent to the relation~\eqref{eq:N9WithDelta9CappedOff}, and hence to $N_8 = \partial_8$, since $b$ commutes with $b_7$ and $b_8$.

\noindent
{\bf Case 3}: The other boundary components $\delta_1, \delta_2, \cdots, \delta_6$.
We can take advantage of the symmetry that the relation $N_9 = \partial_9$ possesses and reduce to the cases we have already discussed.
In Figure~\ref{F:k=9simple}, consider the clockwise rotation $r$ by $2 \pi/3$ about the axis perpendicular to the page and through the center of the figure. This diffeomorphism maps $(a_i, b_i, \delta_i)$ to $(a_{i+3}, b_{i+3}, \delta_{i+3})$, where the indices are taken modulo $9$. 
Then, via the rotation $r$ the relation $N_9 = \partial_9$ becomes
\begin{align*}
\delta_4 \delta_5 \delta_6 \delta_7 \delta_8 \delta_9 \delta_1 \delta_2\delta_3
&= a_4 b_4 b_5 b_6 a_7 b_7 b_8 b_9 a_1 b_1 b_2 b_3,
\end{align*}
which is just a permutation of $N_9 = \partial_9$.
Therefore, capping $\delta_6$ of $N_9 = \partial_9$ is the same as capping $\delta_9$ of $N_9 = \partial_9$ after applying $r$ and hence it results in the relation $N_8 = \partial_8$.
In the same way, capping $\delta_4$ or $\delta_5$ reduces to capping $\delta_7$ or $\delta_8$, respectively.
If we consider the counterclockwise rotation $r^{-1}$ we can reduce the cases of $\delta_1$, $\delta_2$, and $\delta_3$ to the cases of $\delta_7$, $\delta_8$, or $\delta_9$, respectively.

\subsubsection{Two-time blow-up of $f_n$ and the $7$-holed torus relation $N_7 = \partial_7$}
We take the $8$-holed torus relations $N_8 = \partial_8$ and $S_8 = \partial_8$ and cap each one of the boundary components.

\noindent
{\bf Case 1}: Capping $\delta_8$ or $\delta_7$ of $N_8 = \partial_8$.
They give 
\begin{align*}
\delta_1 \delta_2\delta_3 \delta_4 \delta_5 \delta_6 \delta_7 
&= a_1 b_1 a_2 b_2 b_3 a_4 b_4 b_5 b_6 a_7 \underline{b_7 b}, \\
\delta_1 \delta_2\delta_3 \delta_4 \delta_5 \delta_6 \delta_7 
&= a_1 b_1 a_2 b_2 b_3 a_4 b_4 b_5 b_6 a_7 \underline{b b_7},
\end{align*}
which are clearly equivalent as $b$ commutes with $b_7$. Then, from the second
\begin{align*}
\delta_1 \delta_2\delta_3 \delta_4 \delta_5 \delta_6 \delta_7 
&= a_1 b_1 a_2 b_2 b_3 a_4 b_4 b_5 \underline{b_6 a_7 b} b_7 \\
&\sim a_1 b_1 a_2 b_2 b_3 a_4 b_4 b_5 a_6 b_6 a_7 b_7 \\
&\sim a_7 b_7 a_1 b_1 a_2 b_2 b_3 a_4 b_4 b_5 a_6 b_6.
\end{align*}
Then perform the clockwise rotation by $2\pi/7$, which shifts all the indices by $1$. This results in the following $7$-holed torus relation $N_7 = \partial_7$:
\begin{align}
a_1 b_1 a_2 b_2 a_3 b_3 b_4 a_5 b_5 b_6 a_7 b_7 = \delta_1 \delta_2\delta_3 \delta_4 \delta_5 \delta_6 \delta_7.
\end{align}

\noindent
{\bf Case 2}: Capping $\delta_6$, $\delta_5$, or $\delta_4$ of $N_8 = \partial_8$.
They respectively give
\begin{align*}
\delta_1 \delta_2\delta_3 \delta_4 \delta_5 \delta_6 \delta_7 
&= a_1 b_1 a_2 b_2 b_3 a_4 \underline{b_4 b_5 b} a_6 b_6 b_7, \\
\delta_1 \delta_2\delta_3 \delta_4 \delta_5 \delta_6 \delta_7 
&= a_1 b_1 a_2 b_2 b_3 a_4 \underline{b_4 b b_5} a_6 b_6 b_7, \\
\delta_1 \delta_2\delta_3 \delta_4 \delta_5 \delta_6 \delta_7 
&= a_1 b_1 a_2 b_2 b_3 a_4 \underline{b b_4 b_5} a_6 b_6 b_7,
\end{align*}
which are equivalent to each other. From the first,
\begin{align*}
\delta_1 \delta_2\delta_3 \delta_4 \delta_5 \delta_6 \delta_7
&= a_1 b_1 a_2 b_2 b_3 a_4 b_4 b_5 \underline{b a_6 b_6} b_7 \\
&\sim a_1 b_1 a_2 b_2 b_3 a_4 b_4 b_5 a_6 b_6 a_7 b_7,
\end{align*}
which is the same as the one in Case 1 right before applying the rotation.

\noindent
{\bf Case 3}: Capping $\delta_2$ or $\delta_3$ of $N_8 = \partial_8$. 
They give the equivalent relations
\begin{align*}
\delta_1 \delta_2\delta_3 \delta_4 \delta_5 \delta_6 \delta_7 
&= a_1 b_1 a_2 \underline{b_2 b} a_3 b_3 b_4 b_5 a_6 b_6 b_7, \\
\delta_1 \delta_2\delta_3 \delta_4 \delta_5 \delta_6 \delta_7 
&= a_1 b_1 a_2 \underline{b b_2} a_3 b_3 b_4 b_5 a_6 b_6 b_7.
\end{align*}
From the first,
\begin{align*}
\delta_1 \delta_2\delta_3 \delta_4 \delta_5 \delta_6 \delta_7
&= a_1 b_1 a_2 b_2 \underline{b a_3 b_3} b_4 b_5 a_6 b_6 b_7 \\
&\sim a_1 b_1 a_2 b_2 a_3 b_3 a_4 b_4 b_5 a_6 b_6 b_7.
\end{align*}
By the counterclockwise rotation by $2\pi/7$ we can shift the indices by $-1$, which results in $N_7 = \partial_7$ up to a permutation.

\noindent
{\bf Case 4}: Capping $\delta_1$ of $N_8 = \partial_8$. This yields
\begin{align*}
\delta_1 \delta_2\delta_3 \delta_4 \delta_5 \delta_6 \delta_7 
&= \underline{a_1 b a_1} b_1 b_2 a_3 b_3 b_4 b_5 a_6 b_6 b_7 \\
&\sim \underline{b} a_1 \underline{b b_1 b_2} a_3 b_3 b_4 b_5 a_6 \underline{b_6 b_7} \\
&\sim a_1 b_1 b_2 \underline{b a_3 b_3} b_4 \underline{b_5 a_6 b} b_6 b_7 \\
&\sim a_1 b_1 b_2 a_3 b_3 a_4 b_4 a_5 b_5 a_6 b_6 b_7.
\end{align*}
Shifting the indices by $-3$ (or $+4$) by rotation we see that this is equivalent to $N_7 = \partial_7$.

\noindent
{\bf Case 5}: Capping $\delta_8$ or $\delta_7$ of $S_8 = \partial_8$. They yield the equivalent relations
\begin{align*}
\delta_1 \delta_2\delta_3 \delta_4 \delta_5 \delta_6 \delta_7 
&= a_1 b_1 b_2 a_3 b_3 b_4 a_5 b_5 b_6 a_7 \underline{b_7 b}, \\
\delta_1 \delta_2\delta_3 \delta_4 \delta_5 \delta_6 \delta_7 
&= a_1 b_1 b_2 a_3 b_3 b_4 a_5 b_5 b_6 a_7 \underline{b b_7}.
\end{align*}
From the first,
\begin{align*}
\delta_1 \delta_2\delta_3 \delta_4 \delta_5 \delta_6 \delta_7 
&= \underline{b a_1 b_1} b_2 a_3 b_3 b_4 a_5 b_5 b_6 a_7 b_7 \\
&\sim a_1 b_1 a_2 b_2 a_3 b_3 b_4 a_5 b_5 b_6 a_7 b_7,
\end{align*}
which is exactly the expression $N_7$.

\noindent
{\bf Case 6}: The other boundary components $\delta_1, \delta_2, \cdots, \delta_6$ of $S_8 = \partial_8$.
Observe that the relation $S_8 = \partial_8$ is symmetric with respect to the rotation by $2\pi/4$.
Therefore, in the similar way as Case 3 in~\ref{section:N_8}, we can reduce to the cases of capping $\delta_8$ or $\delta_7$.

Note that from the argument so far we deduce that the blow-up of any two base points of $f_n$ and the blow-up of any one base point of $f_s$ are isomorphic.

\subsubsection{Three-time blow-up of $f_n$ and the $6$-holed torus relation $N_6 = \partial_6$}
We cap each one of the boundary components of $N_7 = \partial_7$.

\noindent
{\bf Case 1}: Capping $\delta_7$.
We get
\begin{align*}
\delta_1 \delta_2\delta_3 \delta_4 \delta_5 \delta_6
&= a_1 b_1 a_2 b_2 a_3 b_3 b_4 a_5 b_5 \underline{b_6 a_1 b} \\
&\sim a_1 b_1 a_2 b_2 a_3 b_3 b_4 a_5 b_5 \underline{a_6 b_6 a_1} \\
&\sim a_1 b_1 a_2 b_2 a_3 b_3 b_4 a_5 \underline{b_5 b} a_6 b_6 \\
&\sim a_1 b_1 a_2 b_2 a_3 b_3 \underline{b_4 a_5 b} b_5 a_6 b_6 \\
&\sim a_1 b_1 a_2 b_2 a_3 b_3 a_4 b_4 a_5 b_5 a_6 b_6.
\end{align*}
Thus, we obtain the following $6$-holed torus relation $N_6 = \partial_6$:
\begin{align}
a_1 b_1 a_2 b_2 a_3 b_3 a_4 b_4 a_5 b_5 a_6 b_6
= \delta_1 \delta_2\delta_3 \delta_4 \delta_5 \delta_6.
\end{align}

\noindent
{\bf Case 2}: Capping $\delta_6$ or $\delta_5$.
They give the equivalent relations
\begin{align*}
\delta_1 \delta_2\delta_3 \delta_4 \delta_5 \delta_6  
&= a_1 b_1 a_2 b_2 a_3 b_3 b_4 a_5 \underline{b_5 b} a_6 b_6, \\
\delta_1 \delta_2\delta_3 \delta_4 \delta_5 \delta_6  
&= a_1 b_1 a_2 b_2 a_3 b_3 b_4 a_5 \underline{b b_5} a_6 b_6.
\end{align*}
From the second,
\begin{align*}
\delta_1 \delta_2\delta_3 \delta_4 \delta_5 \delta_6  
&= a_1 b_1 a_2 b_2 a_3 b_3 \underline{b_4 a_5 b} b_5 a_6 b_6 \\
&\sim a_1 b_1 a_2 b_2 a_3 b_3 a_4 b_4 a_5 b_5 a_6 b_6 =N_6.
\end{align*}

\noindent
{\bf Case 3}: Capping $\delta_4$ or $\delta_3$.
They give the equivalent relations
\begin{align*}
\delta_1 \delta_2\delta_3 \delta_4 \delta_5 \delta_6  
&= a_1 b_1 a_2 b_2 a_3 \underline{b_3 b} a_4 b_4 b_5 a_6 b_6, \\
\delta_1 \delta_2\delta_3 \delta_4 \delta_5 \delta_6  
&= a_1 b_1 a_2 b_2 a_3 \underline{b b_3} a_4 b_4 b_5 a_6 b_6.
\end{align*}
From the first,
\begin{align*}
\delta_1 \delta_2\delta_3 \delta_4 \delta_5 \delta_6  
&= a_1 b_1 a_2 b_2 a_3 b_3 \underline{b a_4 b_4} b_5 a_6 b_6, \\
&\sim a_1 b_1 a_2 b_2 a_3 b_3 a_4 b_4 a_5 b_5 a_6 b_6 =N_6.
\end{align*}

\noindent
{\bf Case 4}: Capping $\delta_2$.
We get
\begin{align*}
\delta_1 \delta_2\delta_3 \delta_4 \delta_5 \delta_6 
&= a_1 b_1 \underline{a_2 b a_2} b_2 b_3 a_4 b_4 b_5 a_6 b_6 \\
&\sim a_1 b_1 b a_2 \underline{b b_2 b_3 a_4 b_4} b_5 a_6 b_6 \\
&\sim a_1 b_1 \underline{b a_2 b_2} b_3 a_4 b_4 a_5 b_5 a_6 b_6 \\
&\sim a_1 b_1 a_2 b_2 a_3 b_3 a_4 b_4 a_5 b_5 a_6 b_6 =N_6.
\end{align*}

\noindent
{\bf Case 5}: Capping $\delta_1$.
We get
\begin{align*}
\delta_1 \delta_2\delta_3 \delta_4 \delta_5 \delta_6 
&= \underline{a_1 b a_1} b_1 a_2 b_2 b_3 a_4 b_4 b_5 a_6 b_6 \\
&\sim b a_1 b b_1 a_2 b_2 b_3 a_4 b_4 b_5 a_6 b_6 \\
&\sim a_1 b_1 \underline{b a_2 b_2} b_3 a_4 b_4 \underline{b_5 a_6 b} b_6 \\
&\sim a_1 b_1 a_2 b_2 a_3 b_3 a_4 b_4 a_5 b_5 a_6 b_6 = N_6.
\end{align*}

\subsubsection{Four-time blow-up of $f_n$ and the $5$-holed torus relation $N_5 = \partial_5$}
We cap each one of the boundary components of $N_6 = \partial_6$.

\noindent
{\bf Case 1}: Capping $\delta_6$.
We get
\begin{align*}
\delta_1 \delta_2\delta_3 \delta_4 \delta_5 
&= a_1 b_1 a_2 b_2 a_3 b_3 a_4 b_4 a_5 b_5 a_1 b \\
&\sim a_1 \underline{b a_1 b_1} a_2 b_2 a_3 b_3 a_4 b_4 a_5 b_5 \\
&\sim a_1 a_1 b_1 a_2 a_2 b_2 a_3 b_3 a_4 b_4 a_5 b_5.
\end{align*}
We denote the resulting $5$-holed torus relation by $N_5 = \partial_5$:
\begin{align}
a_1 a_1 b_1 a_2 a_2 b_2 a_3 b_3 a_4 b_4 a_5 b_5
= \delta_1 \delta_2\delta_3 \delta_4 \delta_5.
\end{align}

\noindent
{\bf Case 2}: The other boundary components $\delta_1, \delta_2, \cdots, \delta_5$.
Observe that the relation $N_6 = \partial_6$ is symmetric with respect to the rotation by $2\pi/6$.
Therefore, we can reduce all the other cases to Case 1.

\subsubsection{Five-time blow-up of $f_n$ and the $4$-holed torus relation $N_4 = \partial_4$}
We cap each one of the boundary components of $N_5 = \partial_5$.

\noindent
{\bf Case 1}: Capping $\delta_5$.
We get
\begin{align*}
\delta_1 \delta_2\delta_3 \delta_4  
&= a_1 a_1 b_1 a_2 a_2 b_2 a_3 b_3 a_4 b_4 a_1 b \\
&\sim a_1 b a_1 a_1 b_1 a_2 a_2 b_2 a_3 b_3 a_4 b_4 \\
&\sim \underline{b} a_1 b a_1 b_1 a_2 a_2 b_2 a_3 \underline{b_3 a_4 b_4} \\
&\sim \underline{a_1 b} a_1 b_1 a_2 a_2 b_2 a_3 a_3 b_3 a_4 \underline{b_4} \\
&\sim a_1 a_1 b_1 a_2 a_2 b_2 a_3 a_3 b_3 a_4 a_4 b_4.
\end{align*}
We take the last expression as the $4$-holed torus relation $N_4 = \partial_4$:
\begin{align}
a_1 a_1 b_1 a_2 a_2 b_2 a_3 a_3 b_3 a_4 a_4 b_4
= \delta_1 \delta_2\delta_3 \delta_4.
\end{align}

\noindent
{\bf Case 2}: Capping $\delta_4$.
We get
\begin{align*}
\delta_1 \delta_2\delta_3 \delta_4  
&= a_1 a_1 b_1 a_2 a_2 b_2 a_3 \underline{b_3 a_4 b} a_4 b_4 \\
&\sim a_1 a_1 b_1 a_2 a_2 b_2 a_3 a_3 b_3 a_4 a_4 b_4 =N_4.
\end{align*}

\noindent
{\bf Case 3}: Capping $\delta_3$.
We get
\begin{align*}
\delta_1 \delta_2\delta_3 \delta_4  
&= a_1 a_1 b_1 a_2 a_2 b_2 a_3 \underline{b a_3 b_3} a_4 b_4 \\
&\sim a_1 a_1 b_1 a_2 a_2 b_2 a_3 a_3 b_3 a_4 a_4 b_4 =N_4.
\end{align*}

\noindent
{\bf Case 4}: Capping $\delta_2$.
We get
\begin{align*}
\delta_1 \delta_2\delta_3 \delta_4  
&= a_1 a_1 b_1 a_2 \underline{a_2 b a_2} b_2 a_3 b_3 a_4 b_4 \\
&\sim a_1 a_1 b_1 a_2 b a_2 \underline{b b_2 a_3 b_3} a_4 b_4 \\
&\sim a_1 a_1 b_1 a_2 \underline{b a_2 b_2} a_3 b_3 a_4 a_4 b_4 \\
&\sim a_1 a_1 b_1 a_2 a_2 b_2 a_3 a_3 b_3 a_4 a_4 b_4 = N_4.
\end{align*}

\noindent
{\bf Case 5}: Capping $\delta_1$.
We get
\begin{align*}
\delta_1 \delta_2\delta_3 \delta_4  
&= a_1 \underline{a_1 b a_1} a_1 b_1 a_2 b_2 a_3 b_3 a_4 b_4 \\
&\sim a_1 b \underline{a_1 b a_1} b_1 a_2 b_2 a_3 b_3 a_4 b_4 \\
&\sim \underline{a_1 b} b a_1 \underline{b b_1 a_2 b_2} a_3 b_3 a_4 \underline{b_4} \\
&\sim a_1 \underline{b a_1 b_1} a_2 b_2 a_3 a_3 b_3 a_4 a_4 b_4 \\
&\sim a_1 a_1 b_1 a_2 a_2 b_2 a_3 a_3 b_3 a_4 a_4 b_4 = N_4.
\end{align*}

\begin{remark}
	The $4$-holed torus relation $N_4 = \partial_4$ has a different but equally symmetric expression, which is given in~\cite{KorkmazOzbagci2008}. We can relate the two relation as follows.
	\begin{align*}
	\delta_1 \delta_2\delta_3 \delta_4 = N_4 
	&= a_1 a_1 b_1 a_2 a_2 b_2 a_3 a_3 b_3 a_4 a_4 b_4 \\
	&\sim \underline{a_1 b_1 a_2} \; \underline{a_2 b_2 a_3} \; \underline{a_3 b_3 a_4} \; \underline{a_4 b_4 a_1} \\
	&\sim \underline{b a_1 b_1} \; \underline{b a_2 b_2} \; \underline{b a_3 b_3} \; \underline{b a_4 b_4} \\
	&\sim a_2 b a_1 a_3 b a_2 a_4 b a_3 a_1 b a_4 \\
	&\sim (a_1 a_3 b a_2 a_4 b)^2.
	\end{align*}
	The last expression is Korkmaz-Ozbagci's $4$-holed torus relation.
\end{remark}

\subsubsection{Six-time blow-up of $f_n$ and the $3$-holed torus relation $N_3 = \partial_3$}
We need to cap each one of the boundary components of $N_4 = \partial_4$.
However, noticing that $N_4 = \partial_4$ is symmetric with respect to the rotation by $2\pi/4$, it is clear that any capping gives an equivalent $3$-holed torus relation.

For reference, we give a symmetric expression.
By capping $\delta_4$, we get
\begin{align*}
\delta_1 \delta_2\delta_3   
&= a_1 a_1 b_1 a_2 a_2 b_2 a_3 a_3 b_3 a_1 a_1 b \\
&\sim a_1 \underline{a_1 b a_1} a_1 b_1 a_2 a_2 b_2 a_3 a_3 b_3 \\
&\sim \underline{a_1 b} a_1 \underline{b a_1 b_1} a_2 a_2 b_2 a_3 a_3 \underline{b_3}  \\
&\sim a_1 a_1 a_1 b_1 a_2 a_2 a_2 b_2 a_3 a_3 a_3 b_3.
\end{align*}
We take the last expression as the $3$-holed torus relation $N_3 = \partial_3$:
\begin{align}
a_1 a_1 a_1 b_1 a_2 a_2 a_2 b_2 a_3 a_3 a_3 b_3
= \delta_1 \delta_2\delta_3.
\end{align}

\begin{remark}
	The $3$-holed torus relation $N_3 = \partial_3$ also has an alternative nice expression, which is called \textit{the star relation}~\cite{Gervais2001}. Here we show the equivalence explicitly.
	\begin{align*}
	\delta_1 \delta_2\delta_3 = N_3
	&= a_1 a_1 a_1 b_1 a_2 a_2 a_2 b_2 a_3 a_3 a_3 b_3 \\
	&\sim a_1 \underline{a_1 b_1 a_2} a_2 \underline{a_2 b_2 a_3} a_3 \underline{a_3 b_3 a_1} \\
	&\sim a_1 a_2 b a_1 a_2 a_3 b a_2 a_3 a_1 b a_3 \\
	&\sim (a_1 a_2 a_3 b)^3.
	\end{align*}
	The last expression gives nothing but the star relation.
\end{remark}

\subsubsection{Seven-time blow-up of $f_n$ and the $2$-holed torus relation $N_2 = \partial_2$}
Since $N_3 = \partial_3$ is symmetric with respect to the rotation by $2\pi/3$ it is obvious that capping any one boundary component of $N_3 = \partial_3$ yields an equivalent $2$-holed torus relation.
 
Capping $\delta_3$ of $N_3 = \partial_3$ gives
\begin{align*}
\delta_1 \delta_2   
&= a_1 a_1 \underline{a_1 b_1 a_2} a_2 \underline{a_2 b_2 a_1} a_1 a_1 b \\
&\sim \underline{a_1} a_1 b_1 a_2 b a_2 a_1 b a_2 a_1 \underline{a_1 b} \\
&\sim \underline{b a_1 b_1} a_2 b a_2 a_1 b a_2 a_1 b a_1 \\
&\sim a_2 b a_1 a_2 b a_2 a_1 b a_2 a_1 b a_1 \\
&\sim (a_1 b a_2)^4.
\end{align*}
Thus, we get the $2$-holed torus relation $N_2 = \partial_2$:
\begin{align}
(a_1 b a_2)^4 = \delta_1 \delta_2,
\end{align}
which is also known as {\it the $3$-chain relation}.

\subsubsection{Eight-time blow-up of $f_n$ and the $1$-holed torus relation $N_1 = \partial_1$}
Capping either $\delta_2$ or $\delta_1$ of $N_2 = \partial_2$ gives
\begin{align*}
\delta_1
&= (a_1 b a_1)^4 \\
&= a_1 b a_1 \underline{a_1 b a_1} a_1 b a_1 \underline{a_1 b a_1} \\
&\sim a_1 b a_1 b a_1 b a_1 b a_1 b a_1 b = (a_1 b)^6.
\end{align*}
Writing $a=a_1$, we get the $1$-holed torus relation $N_1 = \partial_1$:
\begin{align}
(a b)^6 = \delta_1,
\end{align}
which is also known as {\it the $2$-chain relation}.

\begin{remark}
	Our non-spin $k$-holed torus relations $N_k = \partial_k$ are all Hurwitz equivalent to Korkmaz-Ozbagci's $k$-holed torus relations. The latter were constructed in the way that the $9$-holed torus relation is a lift of the smaller $k$-holed torus relations and hence conversely they can be obtained by capping boundary components of the $9$-holed torus relation.
\end{remark}

\begin{remark}
	As we have shown, the relations $N_9 = \partial_9$ and $S_8 = \partial_8$ correspond to the minimal holomorphic Lefschetz pencils $f_n$ on $\PP^2$ and $f_s$ on $\PP^1 \times \PP^1$, respectively (while the others $N_k = \partial_k$ ($k<9$) are just blow-ups of them).
	In Figure~\ref{F:handle}, we draw two handle diagrams of the elliptic Lefschetz fibration $E(1) = \PP^2 \sharp 9\overline{\PP}{}^{2} \to \PP^1$ and locate the $(-1)$-sections corresponding to $N_9 = \partial_9$ and $S_8 = \partial_8$.
	Blowing-down those sections must yield the $4$-manifolds $\PP^2$ and $\PP^1 \times \PP^1$, respectively, and the exceptional spheres become the base points of the Lefschetz pencils $f_n$ and $f_s$.
	\begin{figure}[htbp]
		\centering
		\subfigure[The nine $(-1)$-sections corresponding to $N_9 = \partial_9$.
		\label{F:k=9simpleHandle}]
		{\includegraphics[height=170pt]{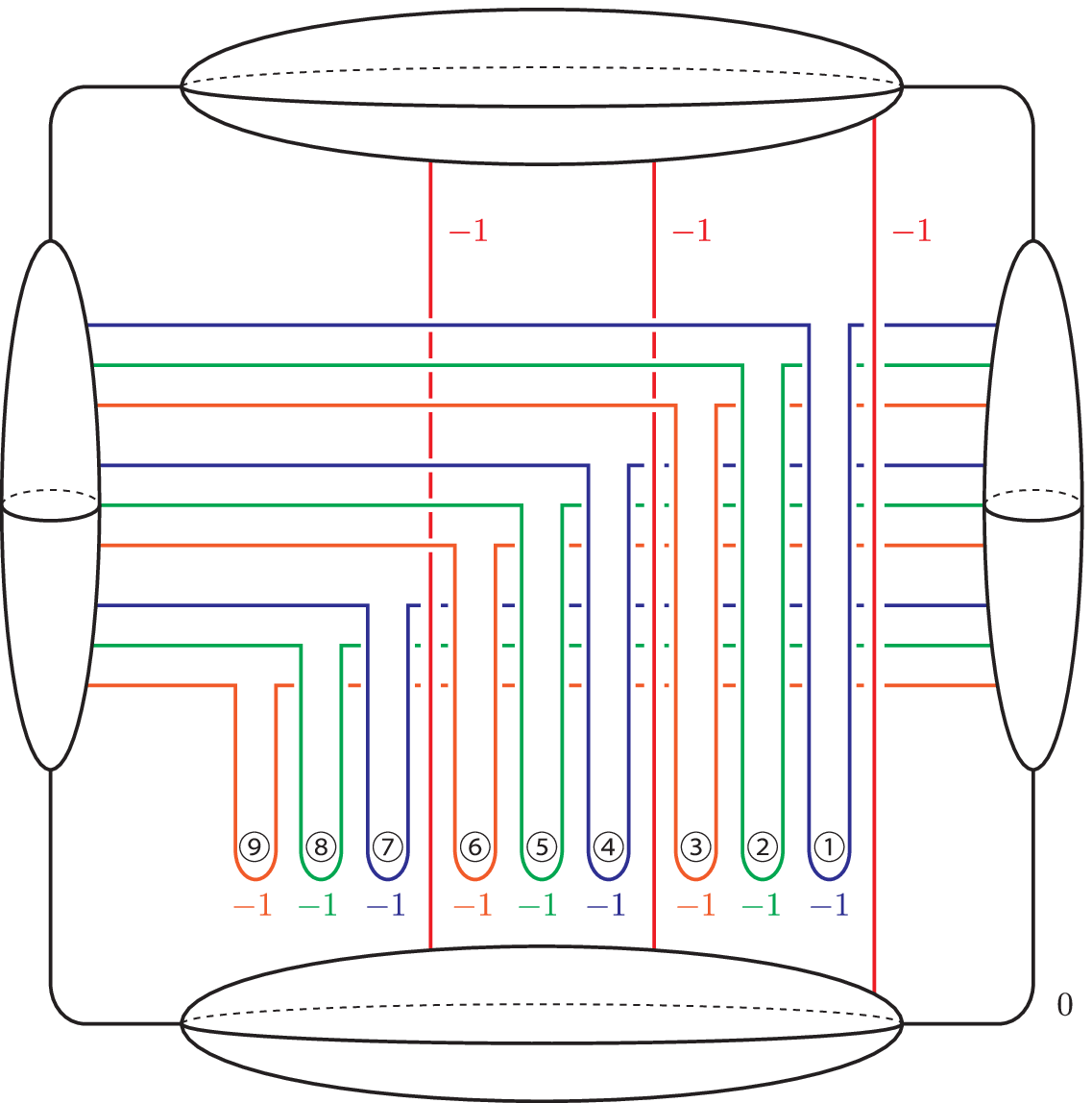}} 
		\hspace{.8em}	
		\subfigure[The eight $(-1)$-sections corresponding to $S_8 = \partial_8$. 
		\label{F:k=8TsimpleHandle}]
		{\includegraphics[height=170pt]{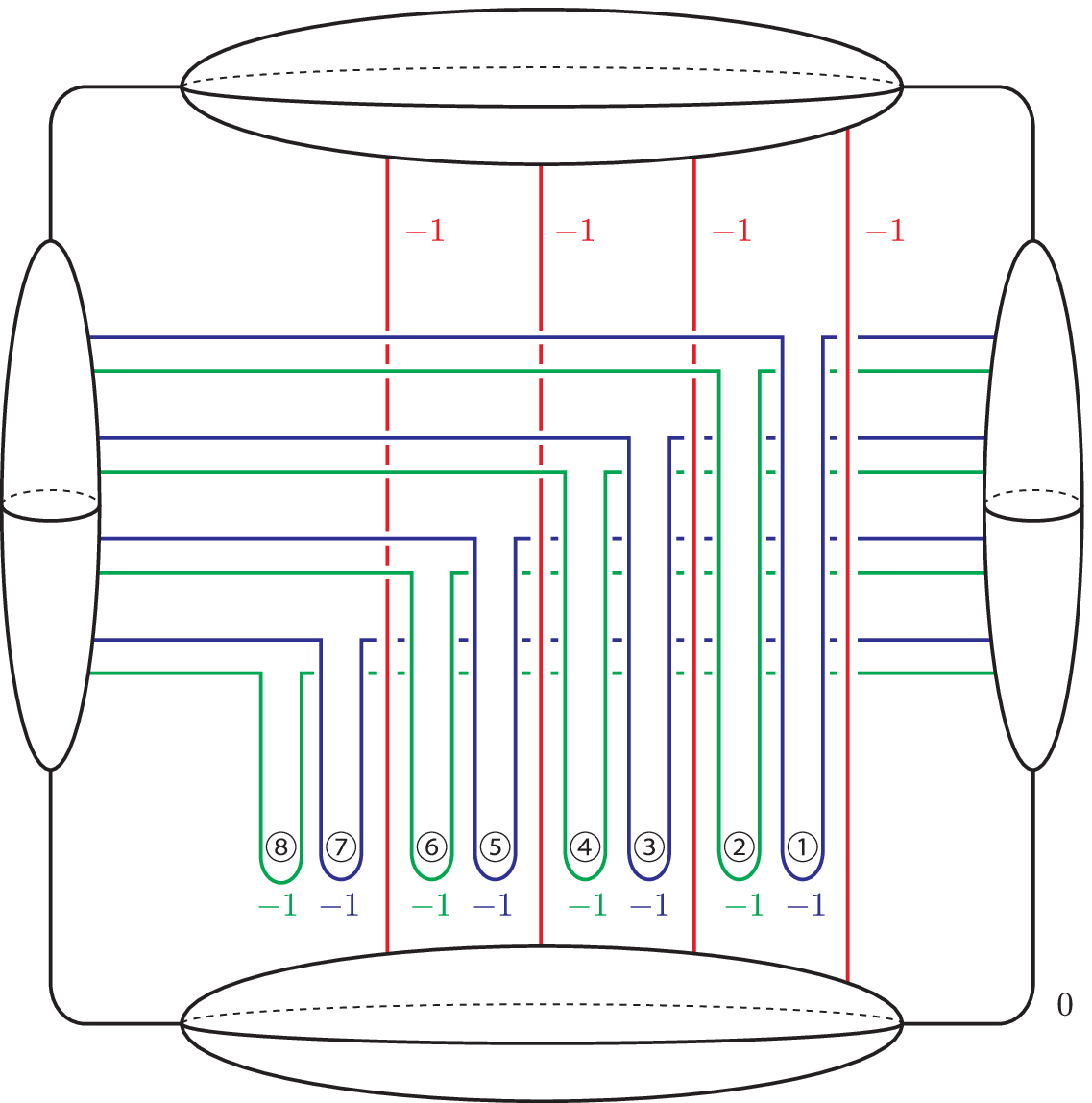}} 
		\caption{Handle diagrams of the elliptic Lefschetz fibration $E(1) = \PP^2 \sharp 9\overline{\PP}{}^{2} \to \PP^1$ with configurations of $(-1)$-sections.} \label{F:handle}
	\end{figure}
\end{remark}

\begin{remark}

	Positive Dehn twist factorizations (with homologically nontrivial curves) of elements in mapping class groups of holed surfaces also provide positive allowable Lefschetz fibrations over $D^2$, which in turn represent Stein fillings of contact $3$-manifolds.
	As summarized in~\cite{Ozbagci2015}, there is another elegant interpretation of the $k$-holed torus relations in this point of view.
	As the monodromy of an open book, the boundary multi-twist $t_{\delta_{1}} \cdots t_{\delta_{k}}$ in $\MCG(\Sigma_1^k)$ yields the contact $3$-manifold $(Y_k,\xi_k)$ that is given as the boundary of the symplectic $D^2$-bundle over $T^2$ with Euler number $-k$.
	While the symplectic $D^2$-bundle naturally gives a Stein filling of $(Y_k,\xi_k)$, the positive allowable Lefschetz fibration over $D^2$ associated with the obvious Dehn twist factorization $t_{\delta_{1}} \cdots t_{\delta_{k}}$ also gives rise to the same Stein filling.
	If the boundary multi-twist $t_{\delta_{1}} \cdots t_{\delta_{k}}$ has another factorization (i.e. a $k$-holed torus relation) it also gives a Stein filling of $(Y_k,\xi_k)$.
	Indeed, the Stein fillings of $(Y_k,\xi_k)$ are classified by Ohta and Ono~\cite{OhtaOno2003}; besides the symplectic $D^2$-bundle there is (i) no more Stein filling when $k\geq10$, (ii) one more Stein filling when $k\leq9$ and $k\neq8$, and (iii) two more Stein fillings when $k=8$.
	The relationship between those Stein fillings and the positive allowable Lefschetz fibrations associated with the $k$-holed torus relations is summarized as follows;
	the Stein fillings in (ii) correspond to $N_k$, the two Stein fillings in (iii) are associated with $N_8$ and $S_8$.
	The fact that $(Y_k,\xi_k)$ has a unique Stein filling for $k\geq10$ reflects that there is no $k(\geq10)$-holed torus relation, which can be also seen from the fact that $E(1)= \PP^2 \sharp 9\overline{\PP}{}^{2}$ can admit no more than nine $(-1)$-sections.

\end{remark}

\bibliography{bibfile}
\bibliographystyle{plain}

\begin{figure}[htbp]
\centering
\subfigure[The Lefschetz vanishing cycles associated with $\alpha_1$ and $\alpha_2$.]{\includegraphics[width=100mm]{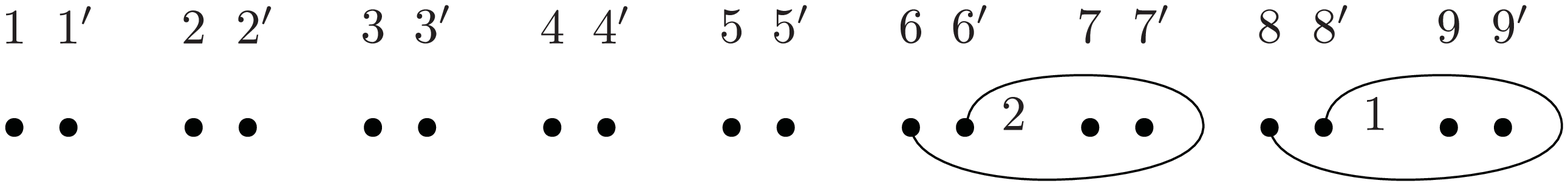}}
\subfigure[The Lefschetz vanishing cycle associated with $\alpha_3$.]{\includegraphics[width=100mm]{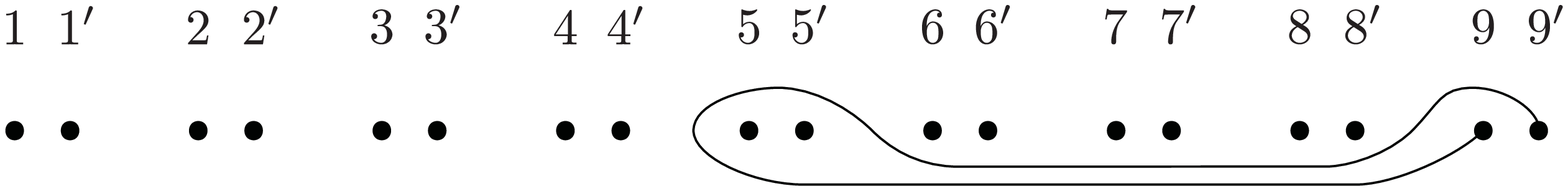}}
\subfigure[The Lefschetz vanishing cycle associated with $\alpha_4$.]{\includegraphics[width=100mm]{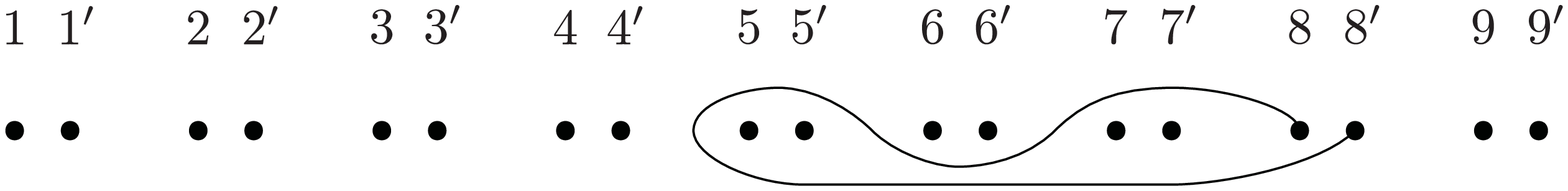}}
\subfigure[The Lefschetz vanishing cycle associated with $\alpha_9$.]{\includegraphics[width=100mm]{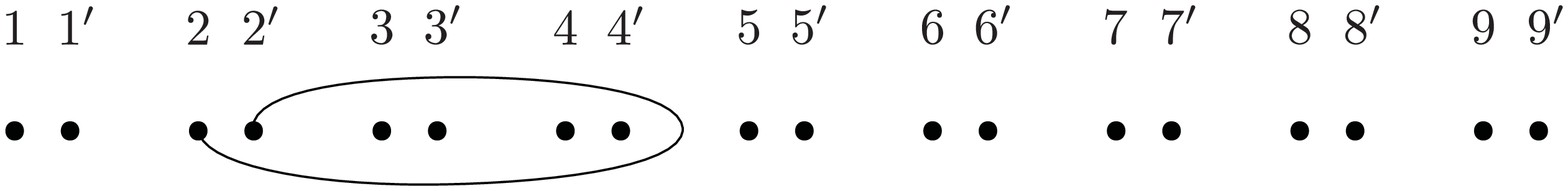}}
\subfigure[The Lefschetz vanishing cycles associated with $\alpha_{10}$ and $\alpha_{11}$.]{\includegraphics[width=100mm]{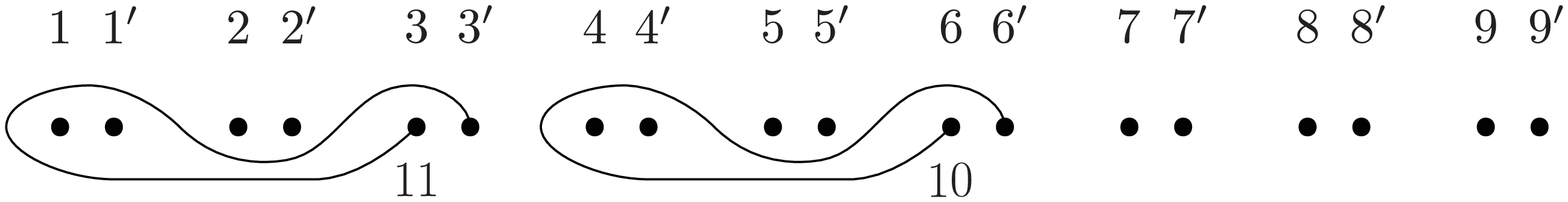}}
\subfigure[The Lefschetz vanishing cycle associated with $\alpha_{12}$.]{\includegraphics[width=100mm]{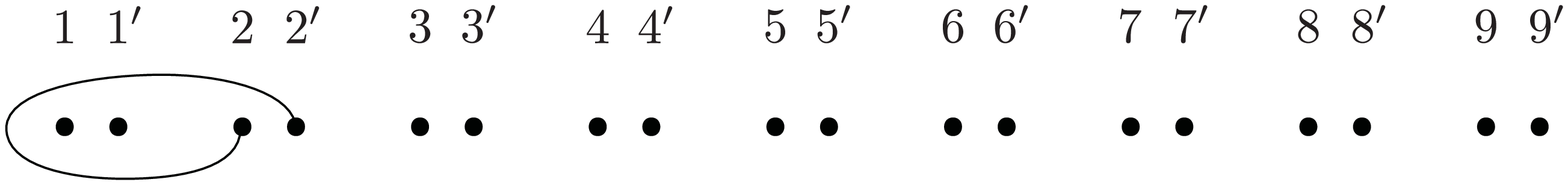}}
\caption{The Lefschetz vanishing cycles of branch points of the restriction $\pi'|_{\widetilde{C_n}}:\widetilde{C_n}\to \PP^1$.}
\label{F:LVC critical value set Un}
\end{figure}
\begin{figure}[htbp]
\centering
\subfigure[The path $\beta$.]{\includegraphics[width=100mm]{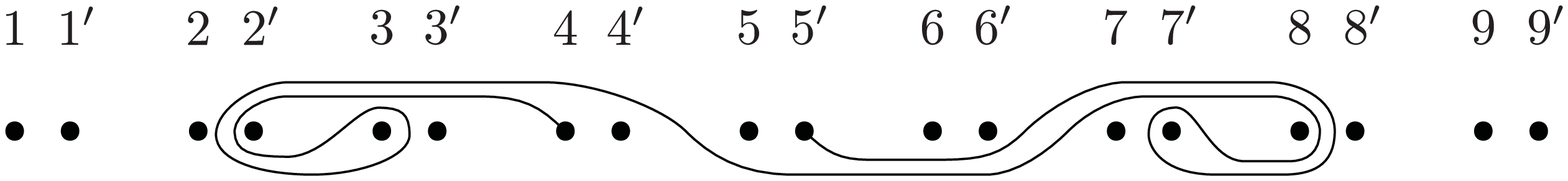}}
\subfigure[The paths $\gamma_1,\gamma_2,\gamma_3,\gamma_4$.]{\includegraphics[width=100mm]{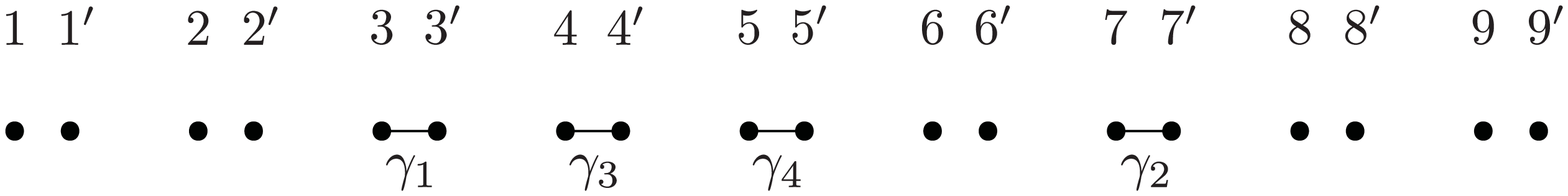}}
\caption{Paths in $\overline{{\pi'}^{-1}(a_0')}$.}
\label{F:curves for LVC critv Un}
\end{figure}
\begin{figure}[htbp]
\centering
\includegraphics[width=40mm]{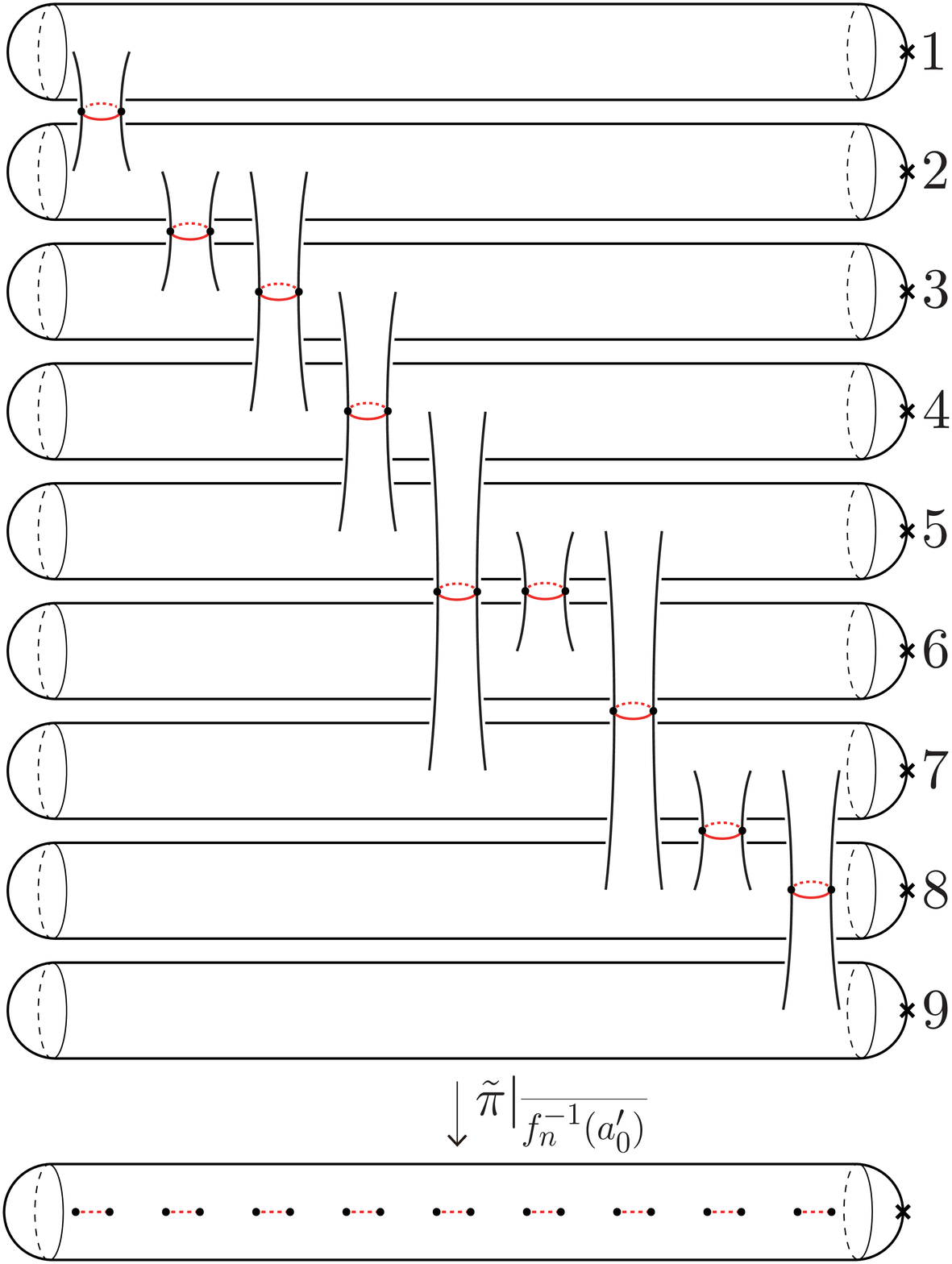}
\caption{The branched covering given in \cref{Eq:2-dim branched cover for Un}. }
\label{F:2-dim branch cover for Un}
\end{figure}
\begin{figure}[htbp]
\centering
\subfigure[]{\includegraphics[width=41mm]{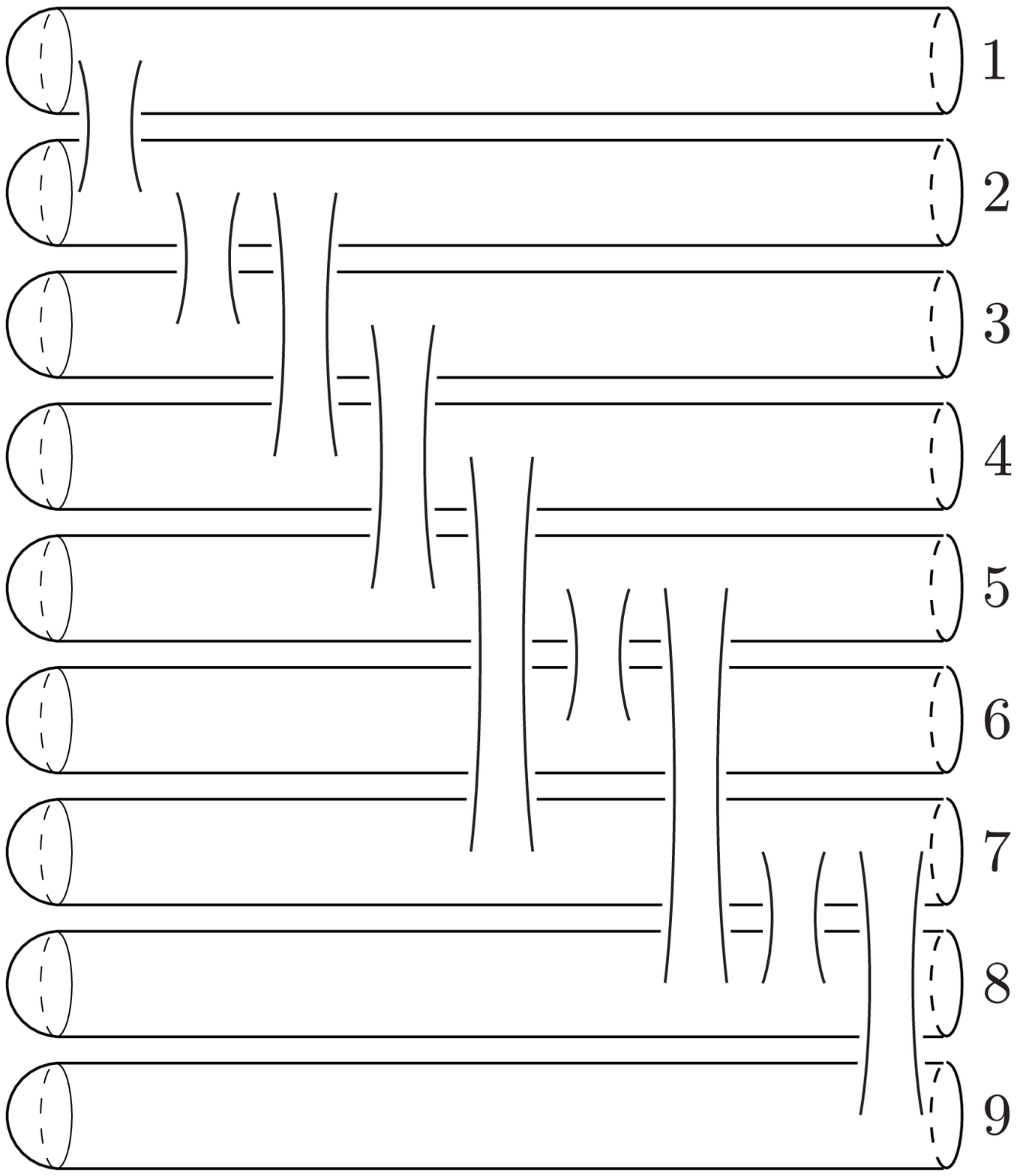}\label{F:isotopy branch Un original}}
\subfigure[]{\includegraphics[width=41mm]{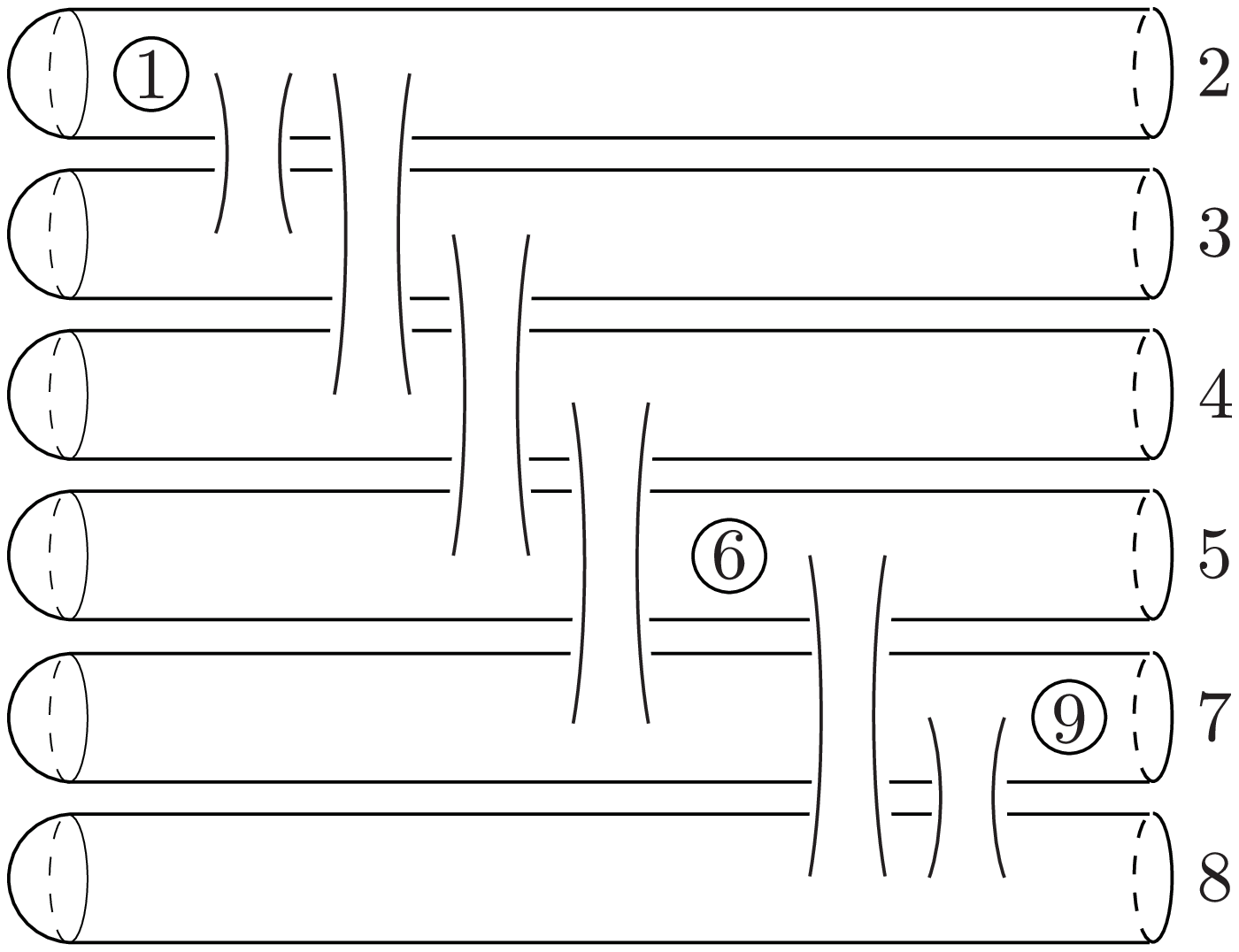}\label{F:isotopy branch Un shrinked}}
\subfigure[]{\includegraphics[width=41mm]{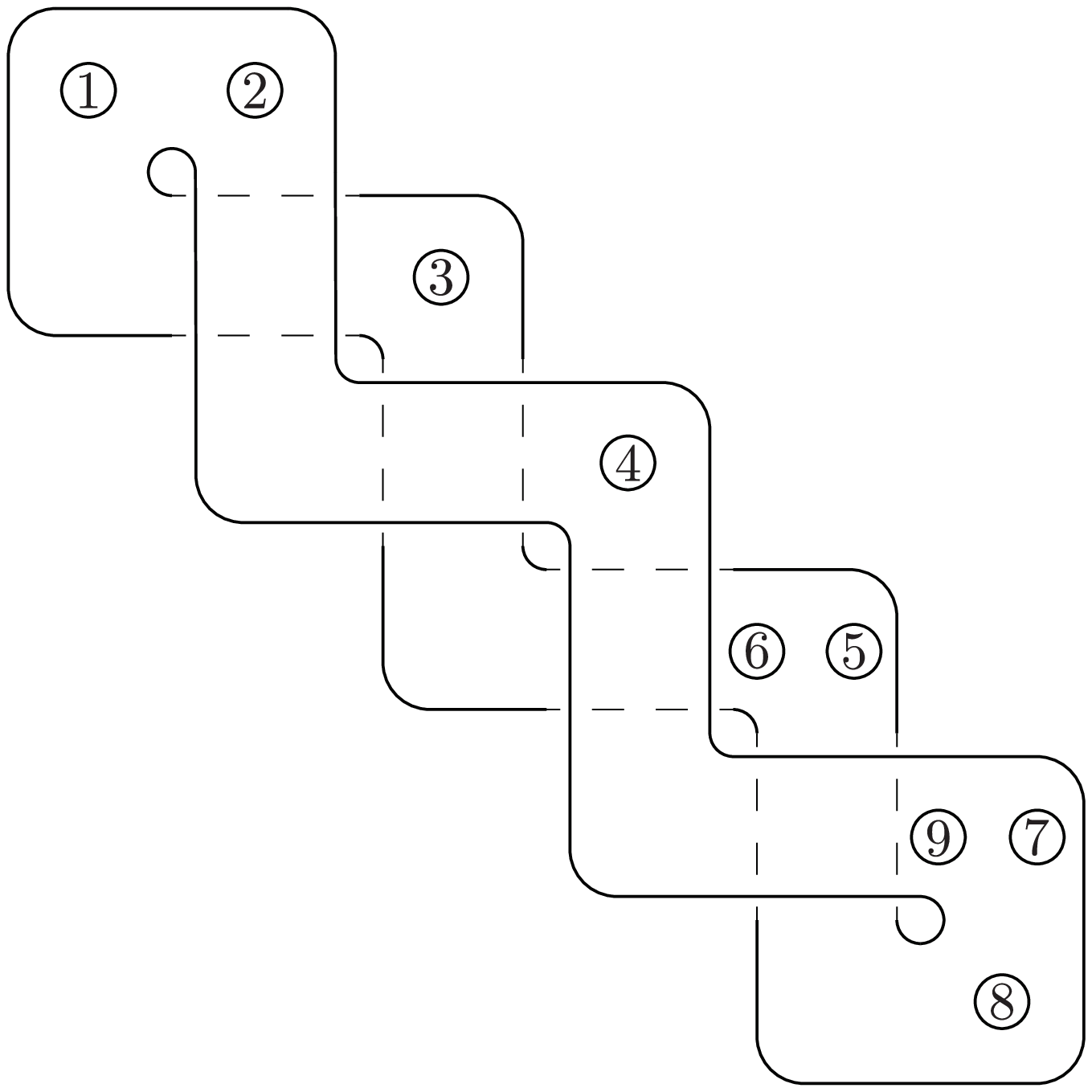}\label{F:isotopy branch Un deformed}}
\caption{The complement of neighborhoods of the base points in $\overline{f_n^{-1}(a_0')}$.  
}
\label{F:isotopy branch Un}
\end{figure}
\begin{figure}[htbp]
\centering
\subfigure[The vanishing cycles $c_1$ and $c_2$.]{\includegraphics[width=90mm]{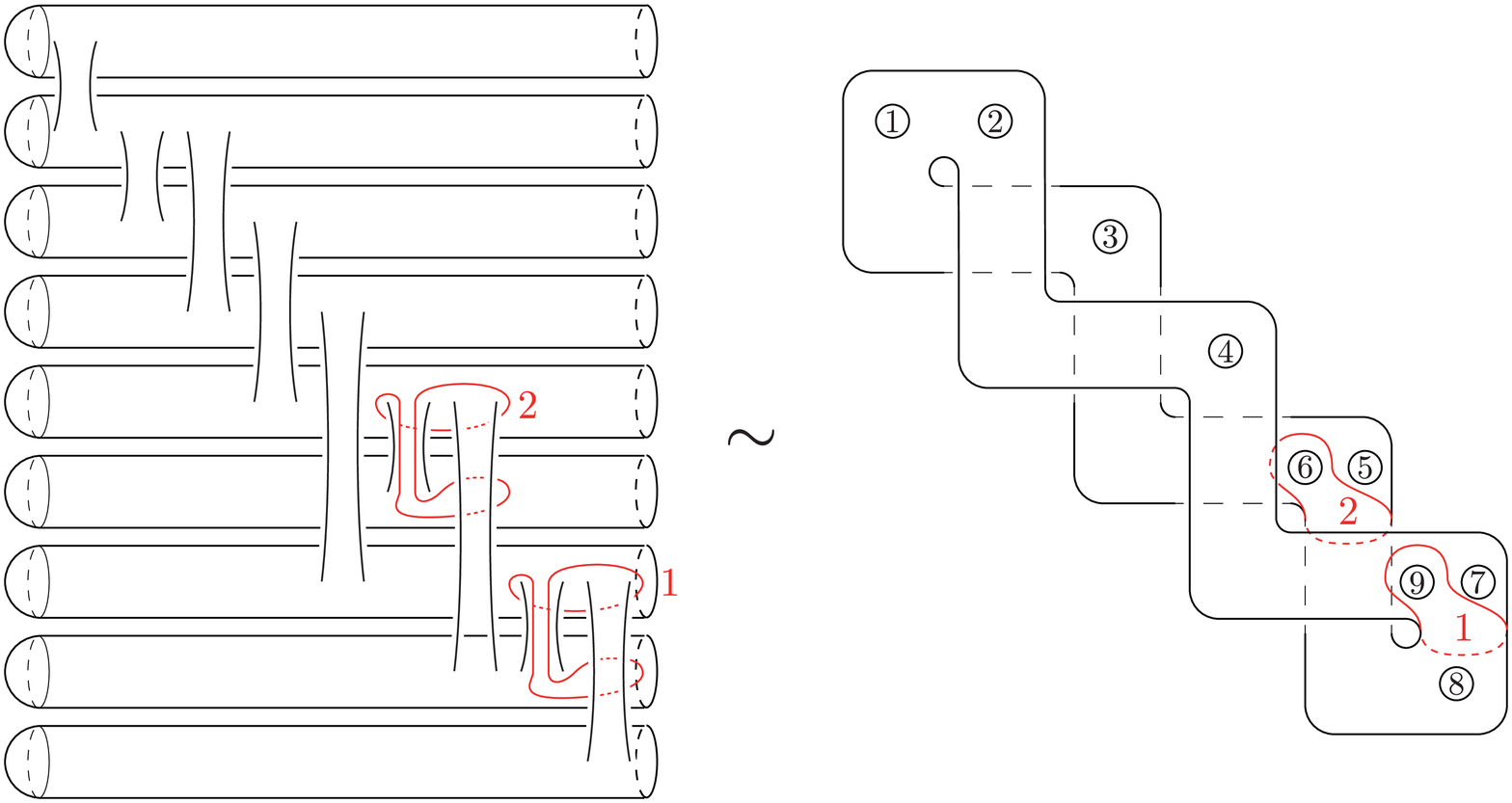}}
\caption{Vanishing cycles of the Lefschetz pencil $f_n$.}
\end{figure}
\addtocounter{figure}{-1}
\begin{figure}[htbp]
\centering
\subfigure{}
\setcounter{subfigure}{1}
\subfigure[The vanishing cycle $c_3$.]{\includegraphics[width=90mm]{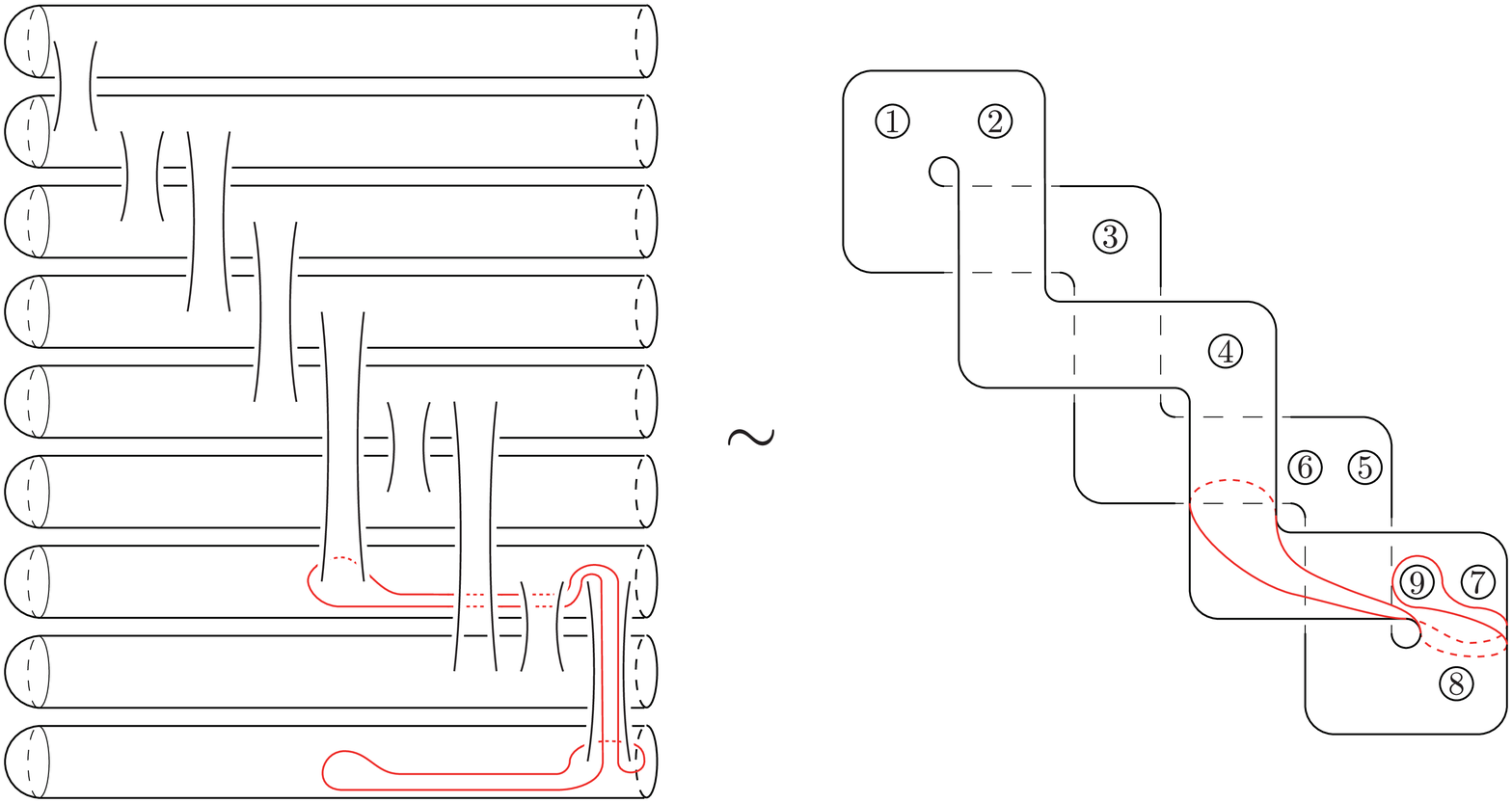}}
\subfigure[The vanishing cycle $c_4$.]{\includegraphics[width=90mm]{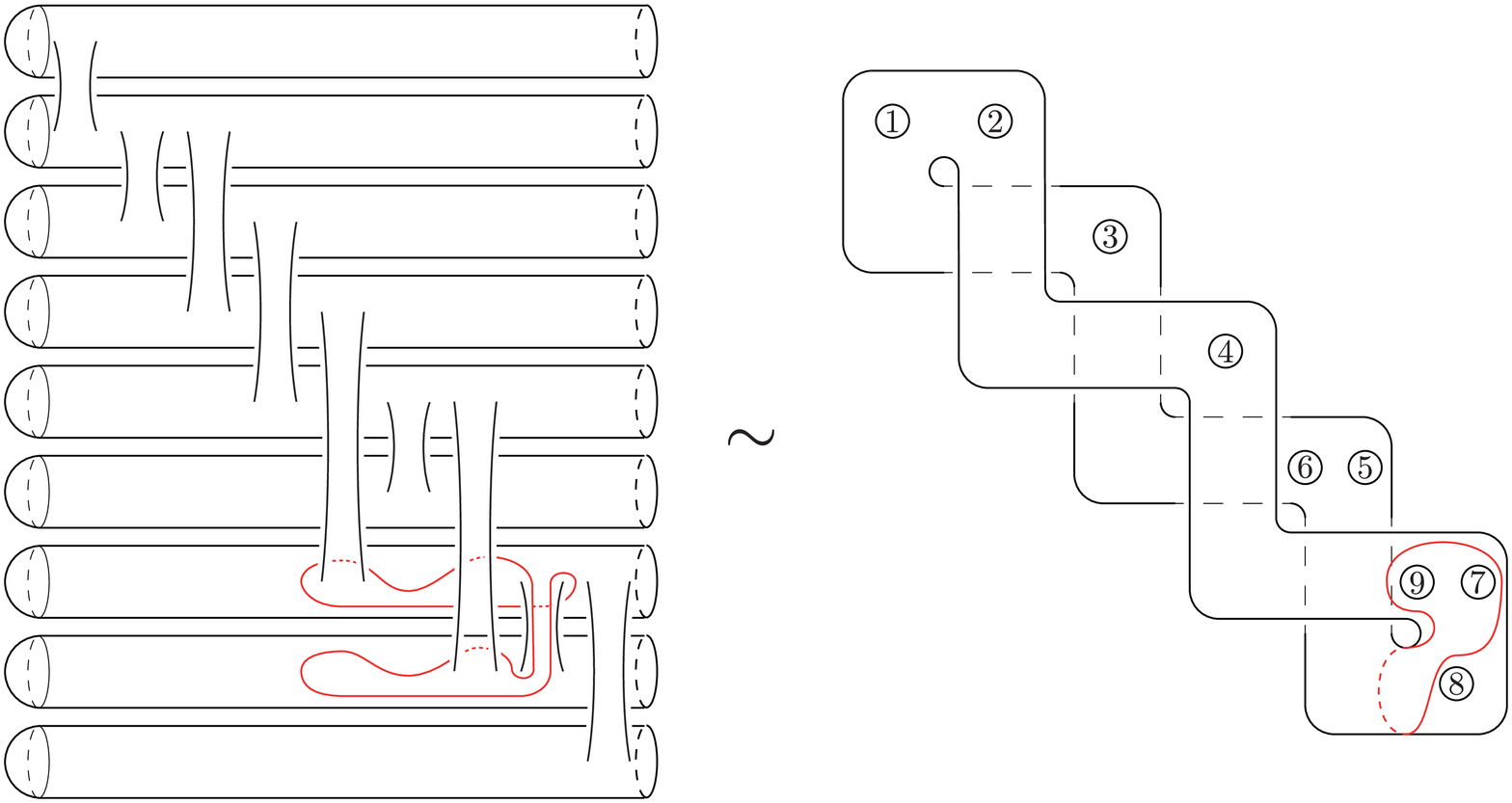}}
\subfigure[The vanishing cycle $c_5$.]{\includegraphics[width=90mm]{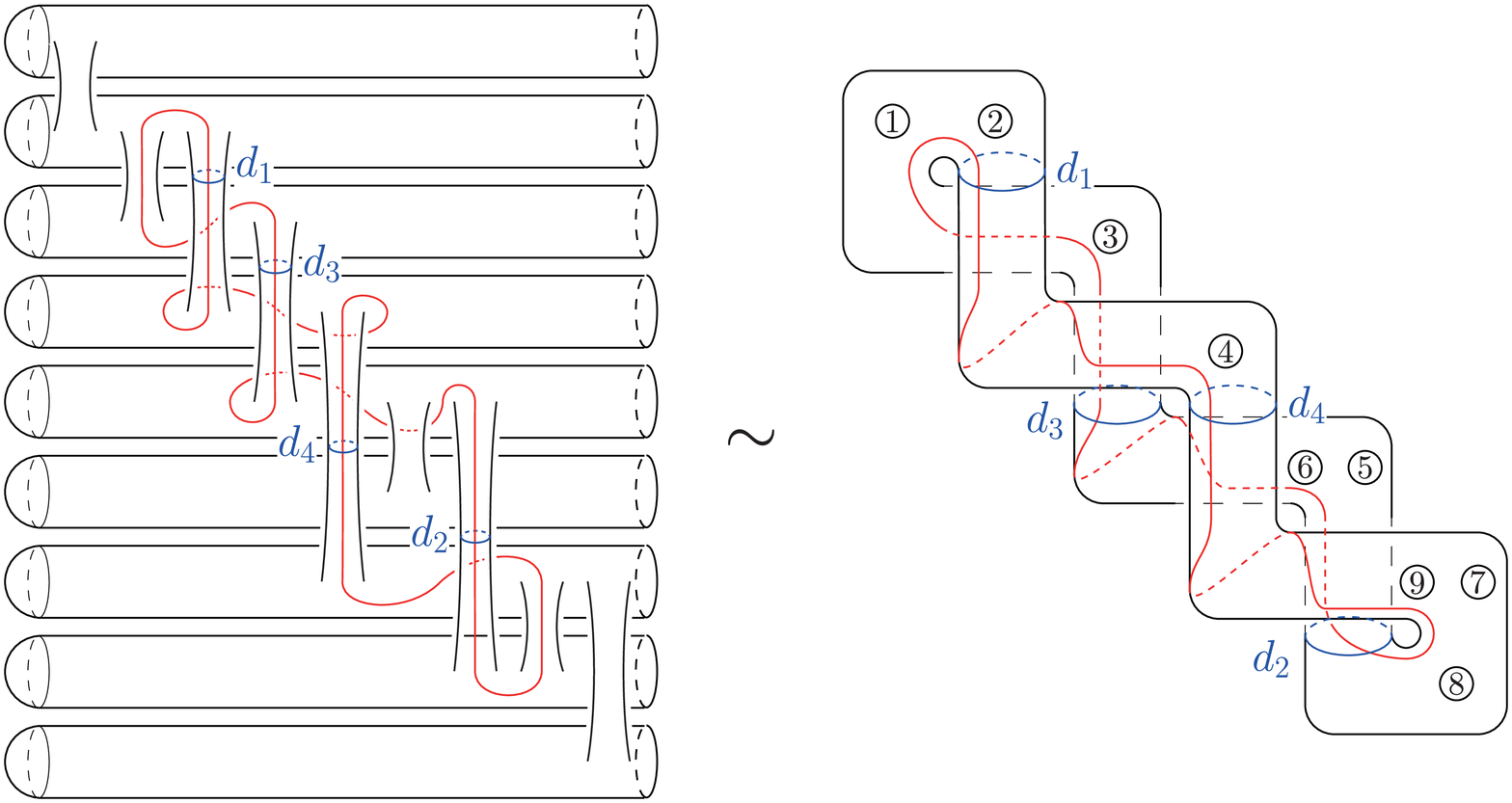}\label{F:vanishing cycle Un5}}
\caption{Vanishing cycles of the Lefschetz pencil $f_n$.}
\label{F:vanishing cycle Un}
\end{figure}
\addtocounter{figure}{-1}
\begin{figure}[htbp]
\centering
\subfigure{}
\setcounter{subfigure}{4}
\subfigure[The vanishing cycle $c_9$.]{\includegraphics[width=90mm]{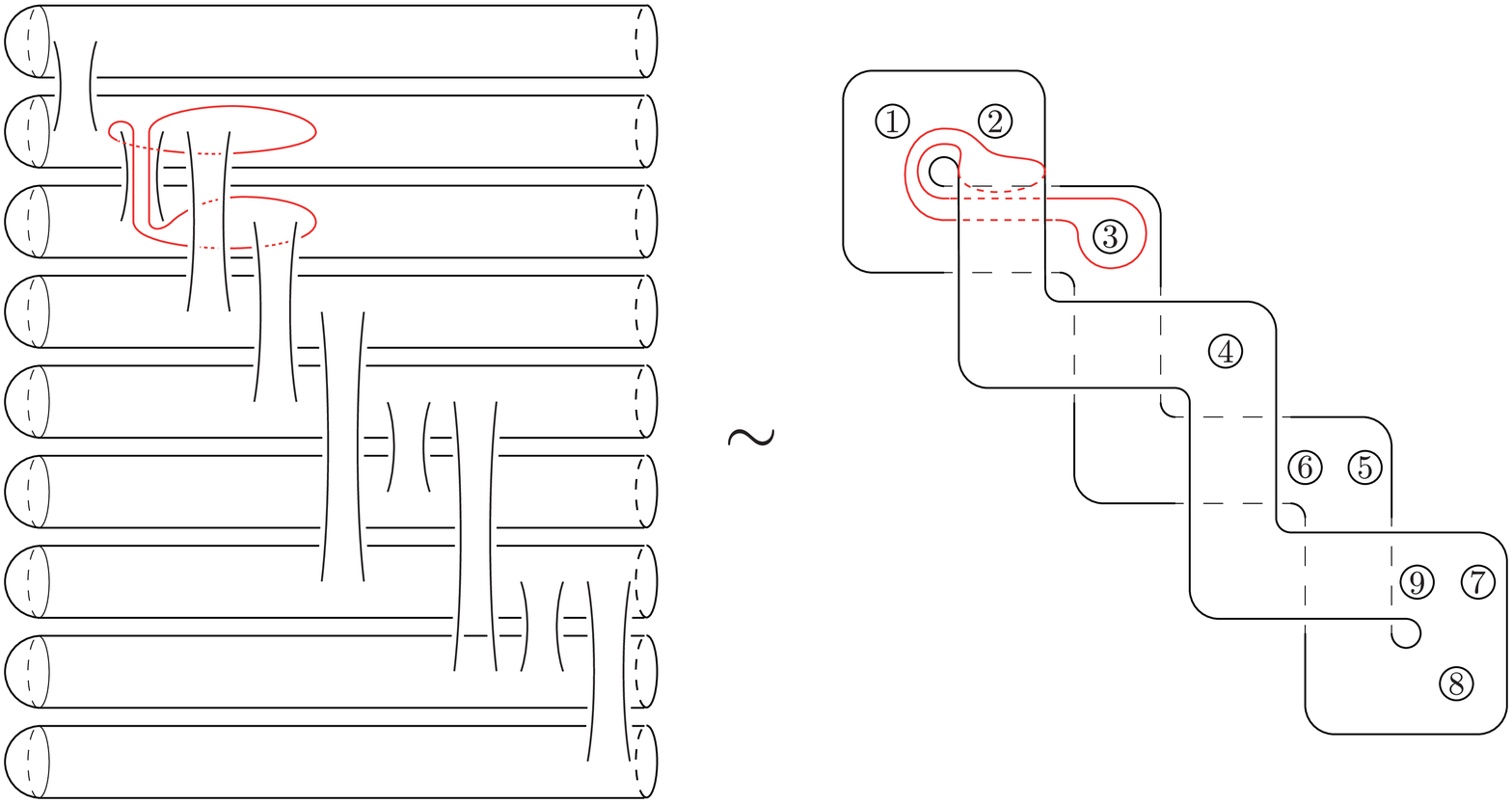}}
\subfigure[The vanishing cycles $c_{10}$ and $c_{11}$.]{\includegraphics[width=90mm]{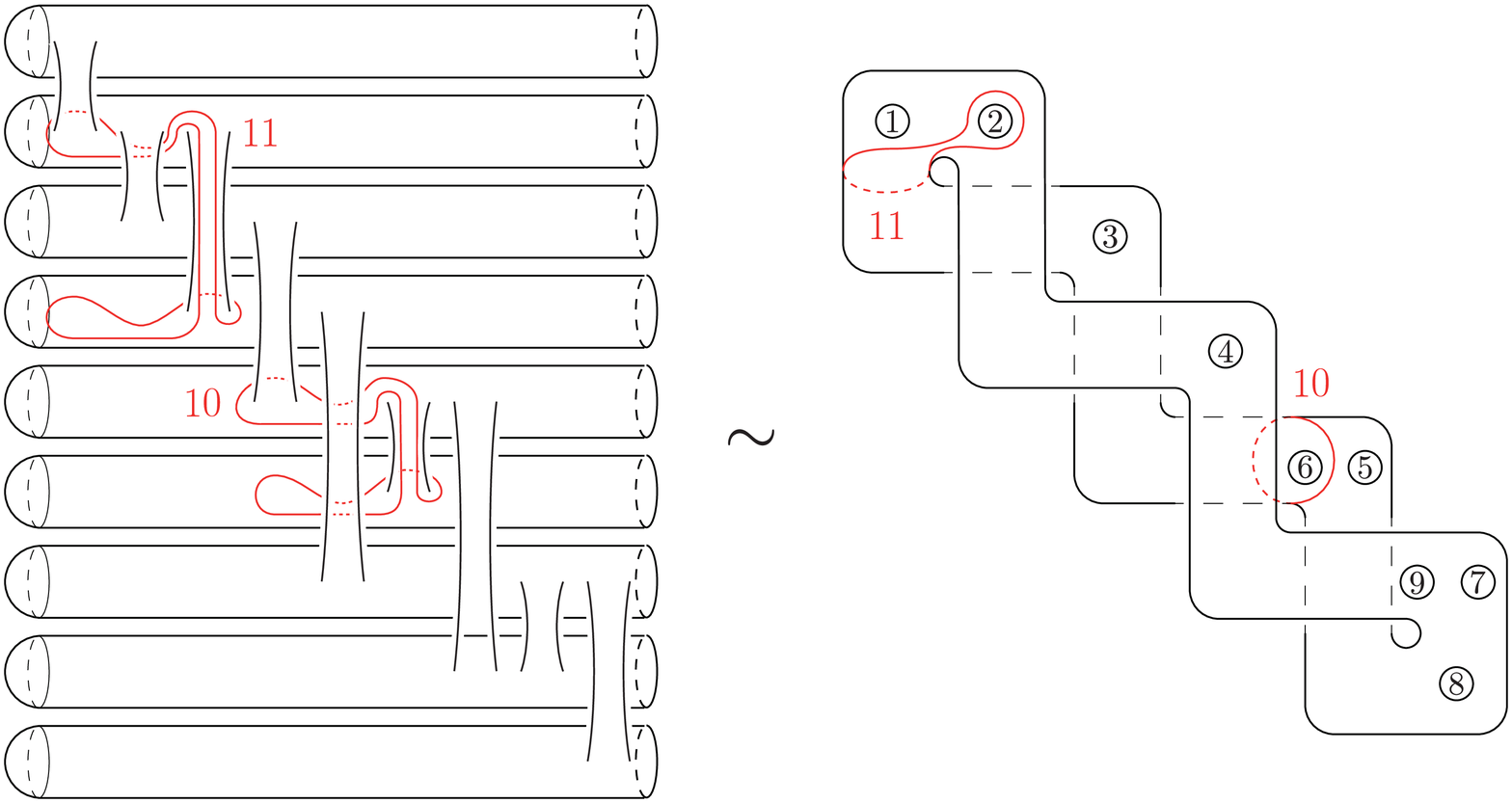}}
\subfigure[The vanishing cycle $c_{12}$.]{\includegraphics[width=90mm]{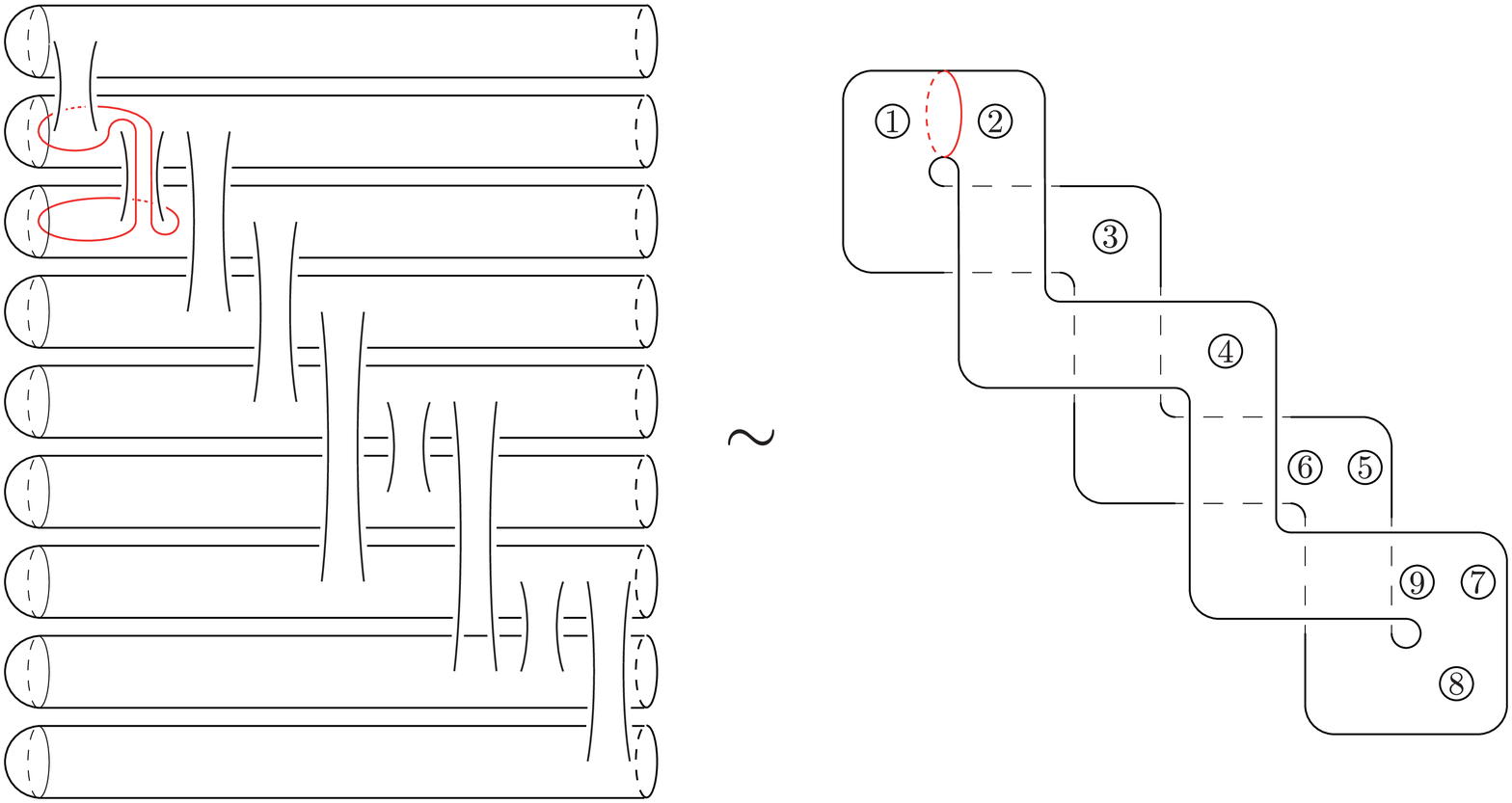}}
\caption{Vanishing cycles of the Lefschetz pencil $f_n$.}
\end{figure}
\begin{figure}[htbp]
\centering
\includegraphics[width=100mm]{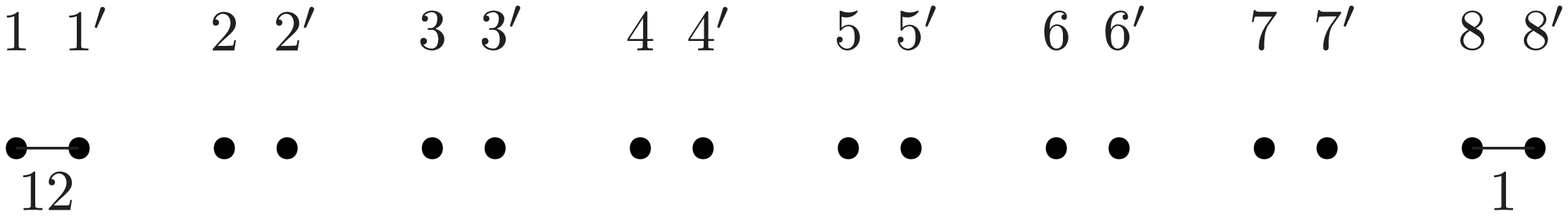}
\caption{The Lefschetz vanishing cycles associated with $\alpha_1$ and $\alpha_{12}$.}
\label{F:LVC extrabranch Us}
\end{figure}
\begin{figure}[htbp]
\centering
\subfigure[The Lefschetz vanishing cycles associated with $\alpha_2$ and $\alpha_{11}$.]{\includegraphics[width=100mm]{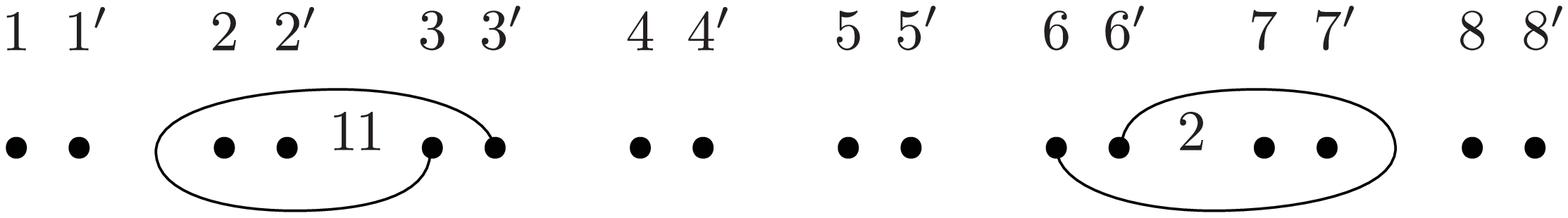}}
\subfigure[The Lefschetz vanishing cycle associated with $\alpha_{3}$.]{\includegraphics[width=100mm]{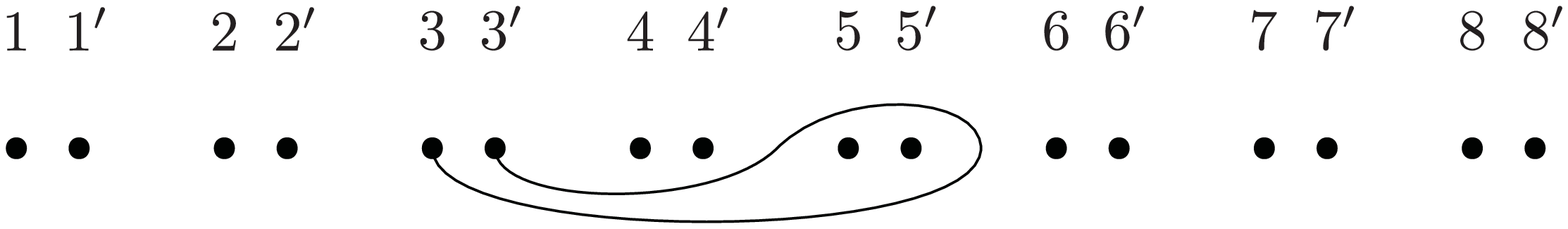}}
\subfigure[The Lefschetz vanishing cycle associated with $\alpha_{4}$.]{\includegraphics[width=100mm]{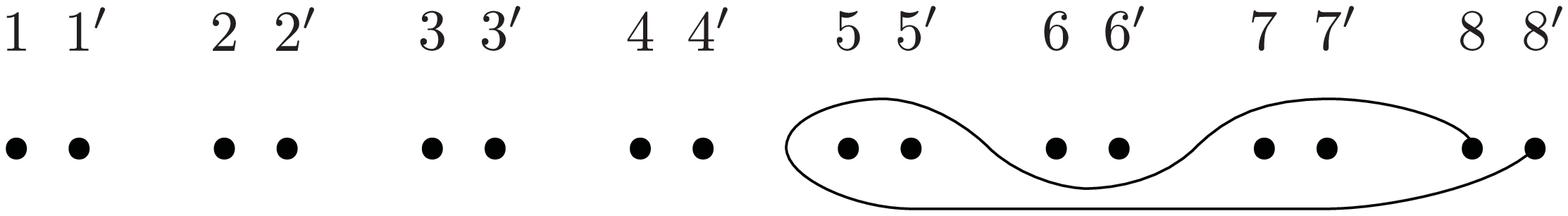}}
\subfigure[The Lefschetz vanishing cycle associated with $\alpha_{9}$.]{\includegraphics[width=100mm]{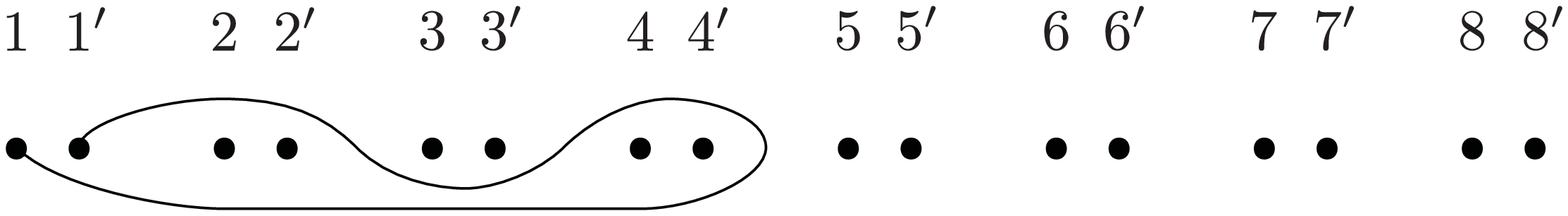}}
\subfigure[The Lefschetz vanishing cycle associated with $\alpha_{10}$.]{\includegraphics[width=100mm]{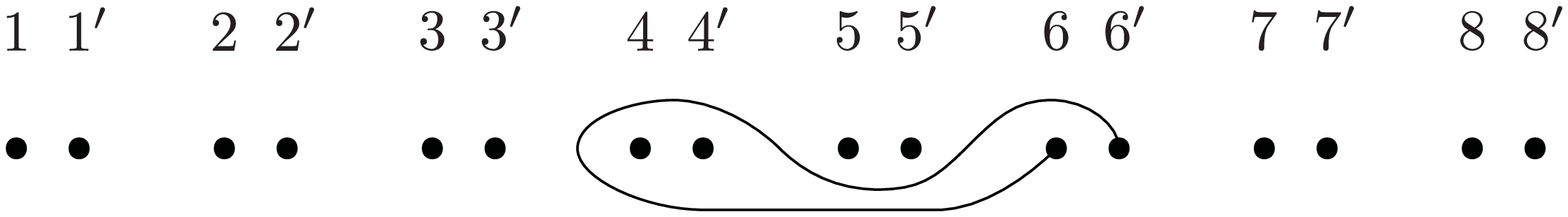}}
\subfigure[The path $\beta$.]{\includegraphics[width=100mm]{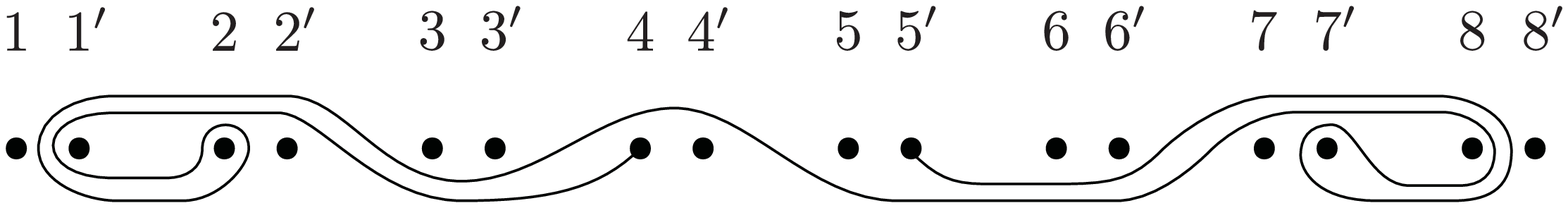}\label{F:LVC critv Us beta}}
\subfigure[The paths $\gamma_1,\gamma_2,\gamma_3,\gamma_4$.]{\includegraphics[width=100mm]{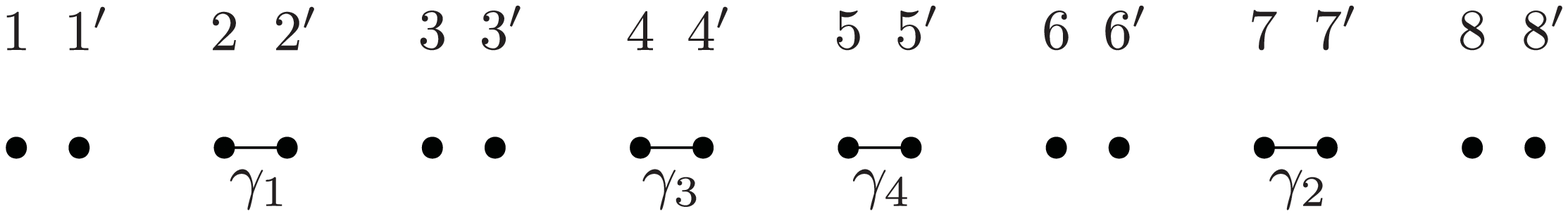}\label{F:LVC critv Us gamma}}
\caption{The Lefschetz vanishing cycles of branch points of the restriction $\pi'|_{\widetilde{C_s}}:\widetilde{C_s}\to \PP^1$.}
\label{F:LVC critical value set Us}
\end{figure}
\begin{figure}[htbp]
\centering
\includegraphics[width=43mm]{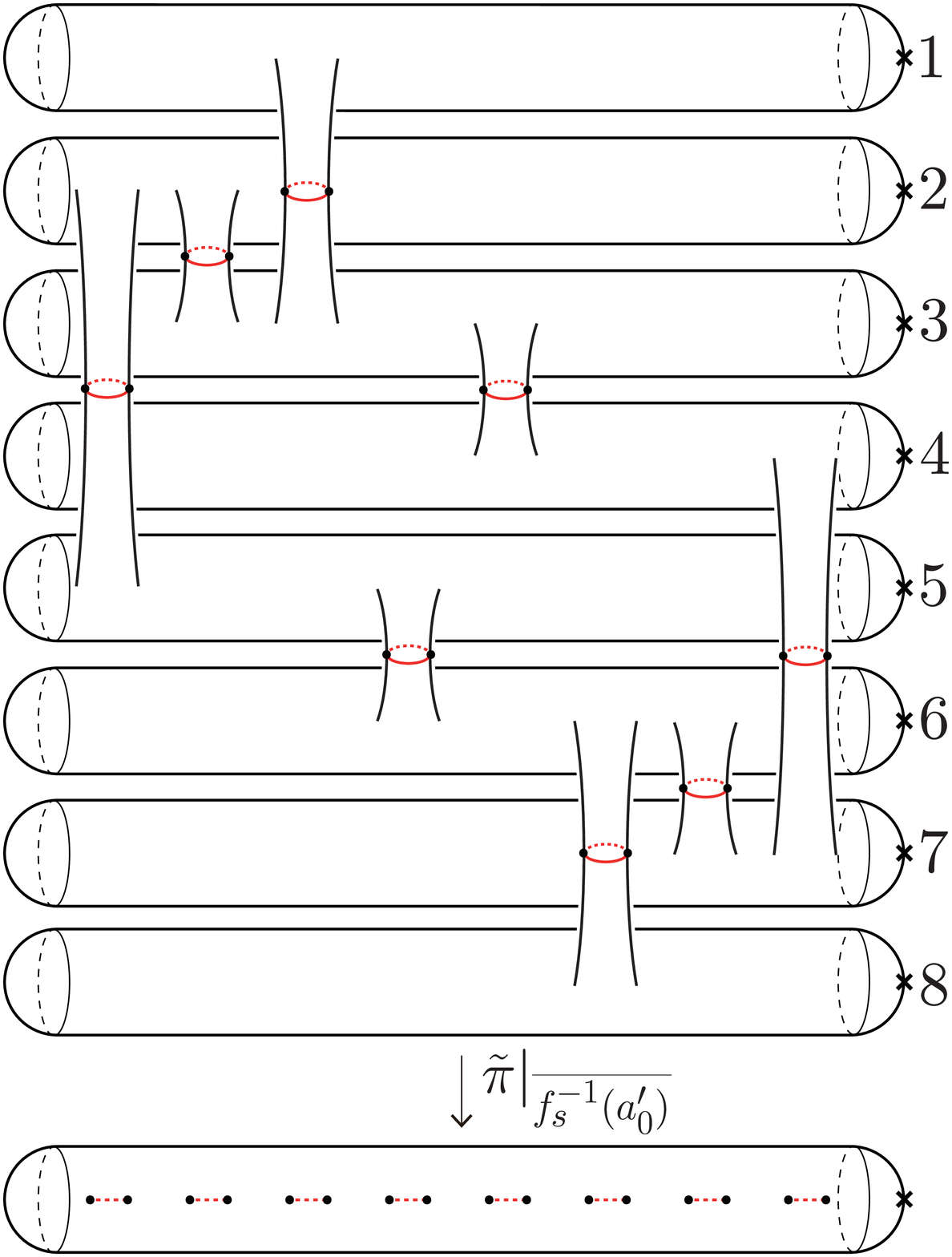}
\caption{The branched covering given in \cref{Eq:2-dim branched cover for Us}. }
\label{F:2-dim branch cover for Us}
\end{figure}
\begin{figure}[htbp]
\centering
\includegraphics[width=100mm]{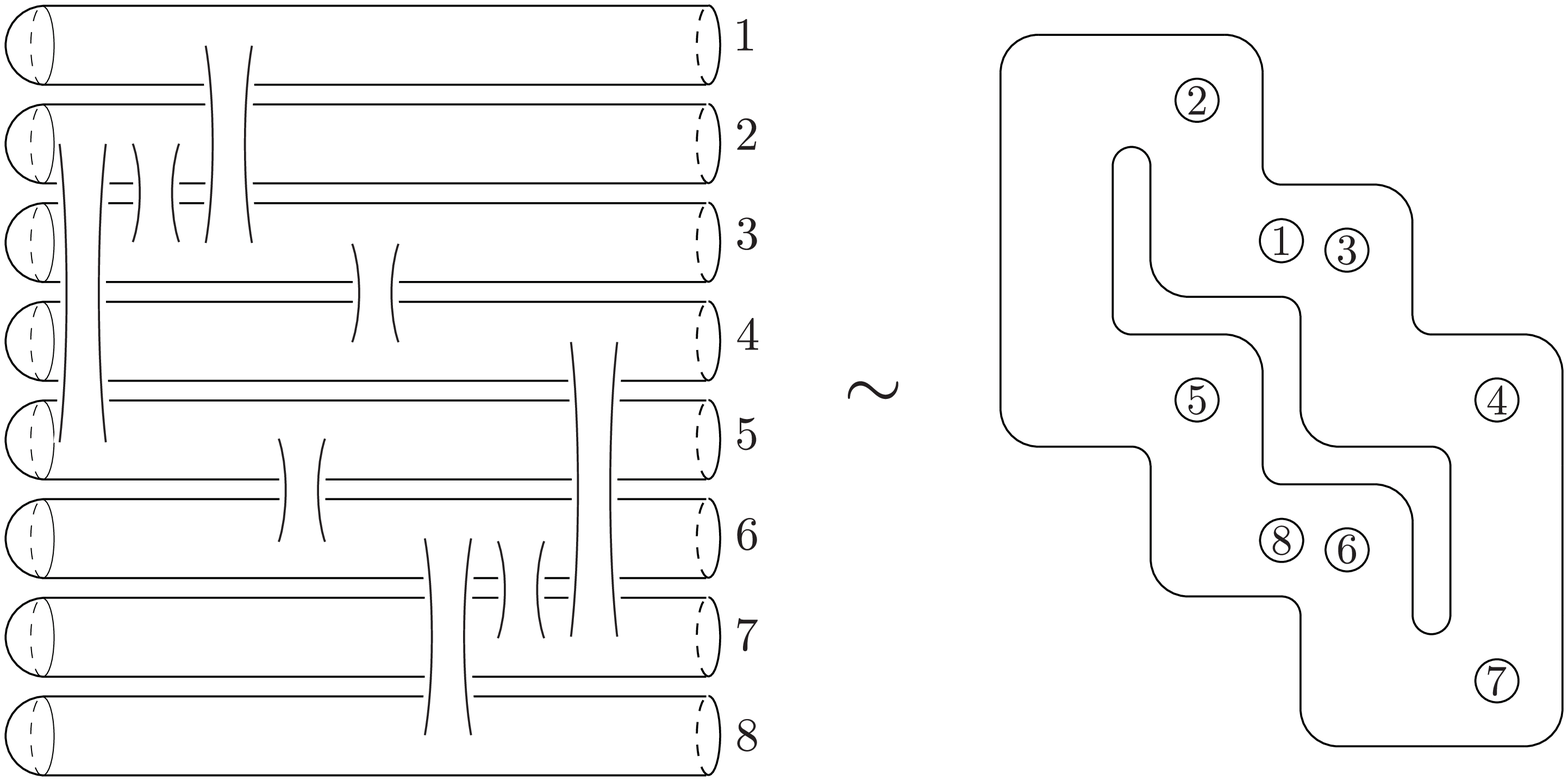}
\caption{The complement of neighborhoods of the base points in $\overline{f_s^{-1}(a_0')}$.  
}
\label{F:isotopy branchcover Us}
\end{figure}
\begin{figure}[htbp]
\centering
\subfigure[The vanishing cycles $c_1$ and $c_{12}$.]{\includegraphics[width=90mm]{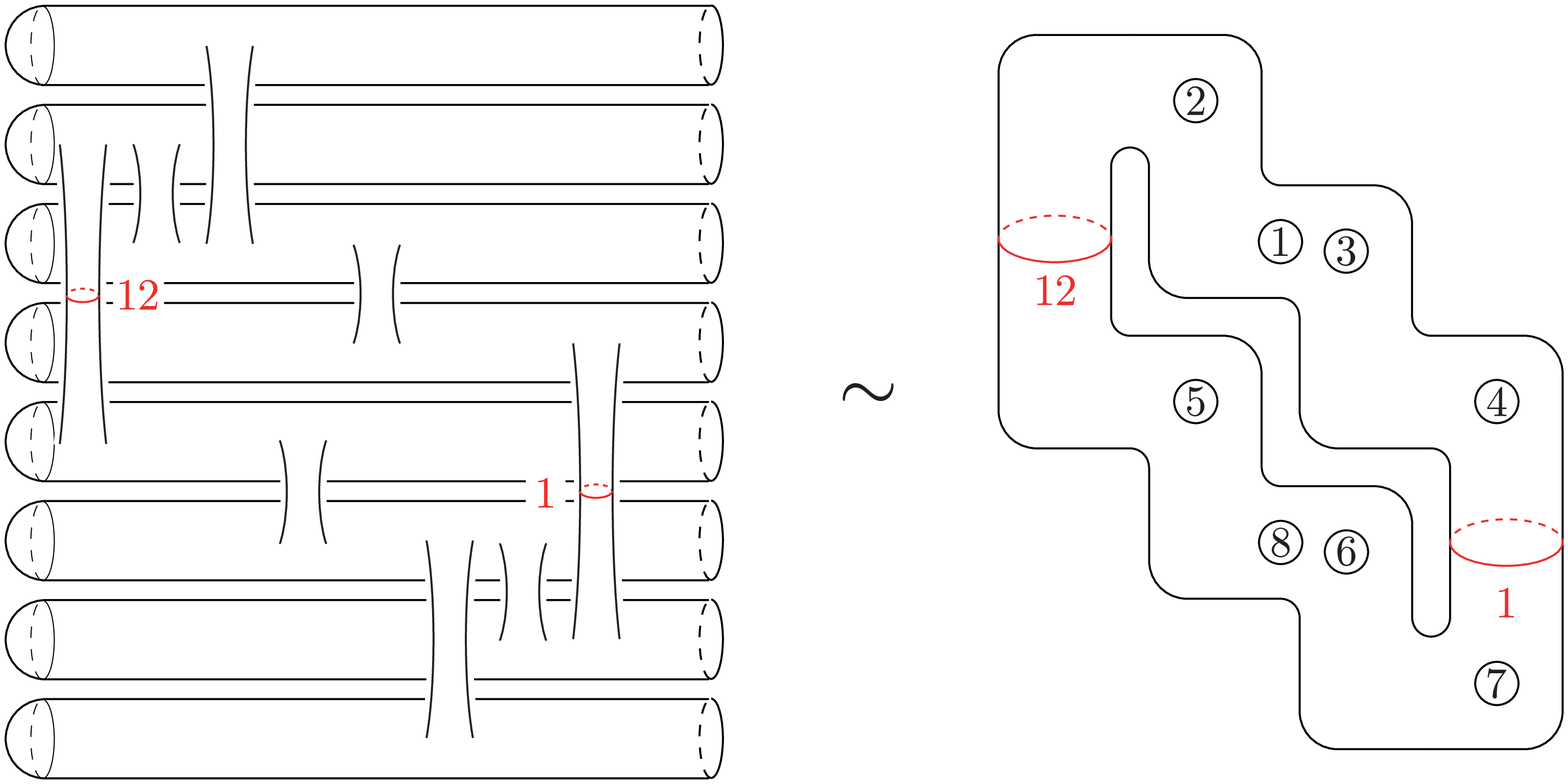}}
\caption{Vanishing cycles of the Lefschetz pencil $f_s$.}
\label{F:vanishing cycle Us}
\end{figure}
\addtocounter{figure}{-1}
\begin{figure}[htbp]
\centering
\subfigure{}
\setcounter{subfigure}{1}
\subfigure[The vanishing cycles $c_2$ and $c_{11}$.]{\includegraphics[width=90mm]{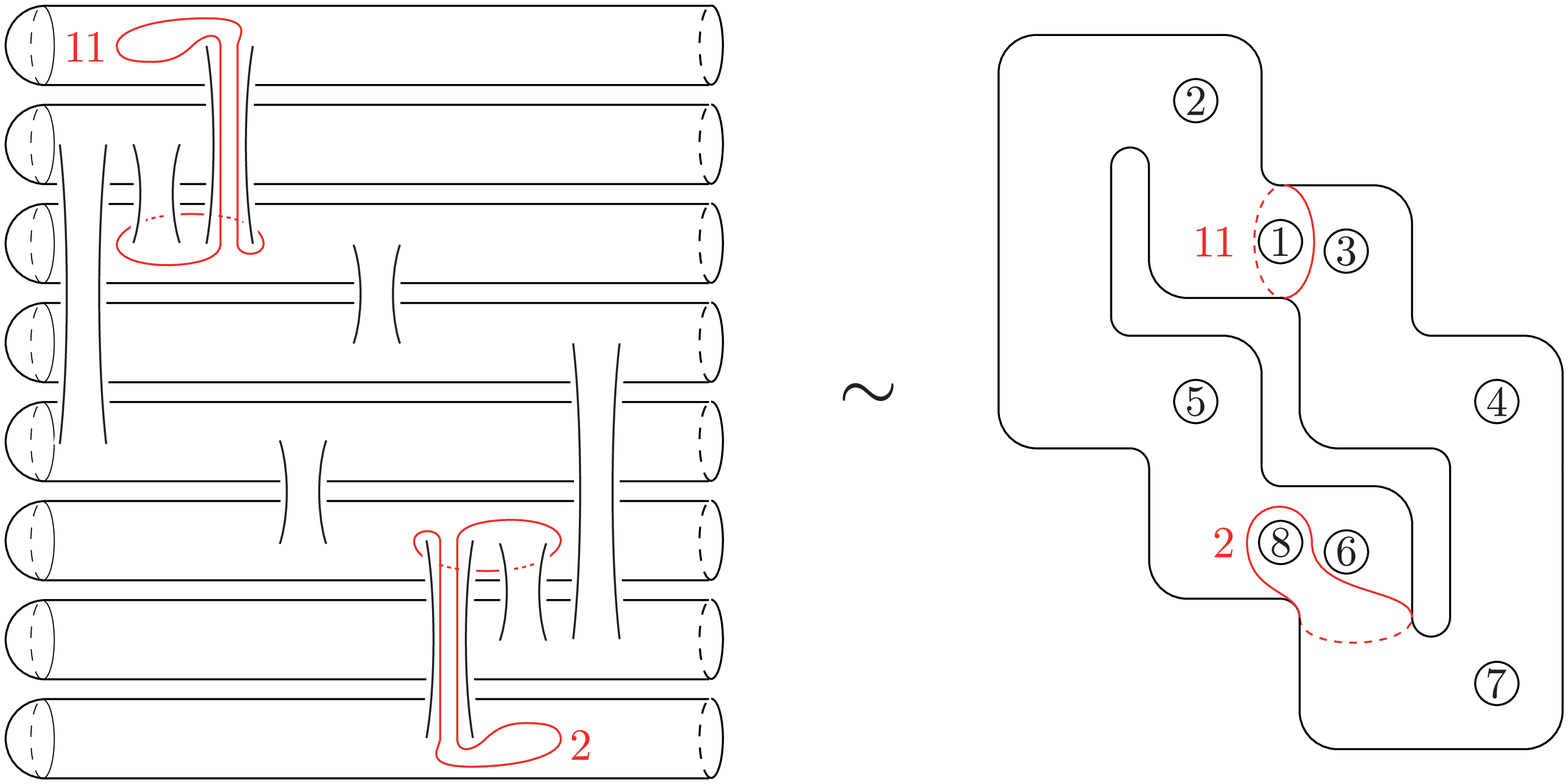}}
\subfigure[The vanishing cycle $c_3$.]{\includegraphics[width=90mm]{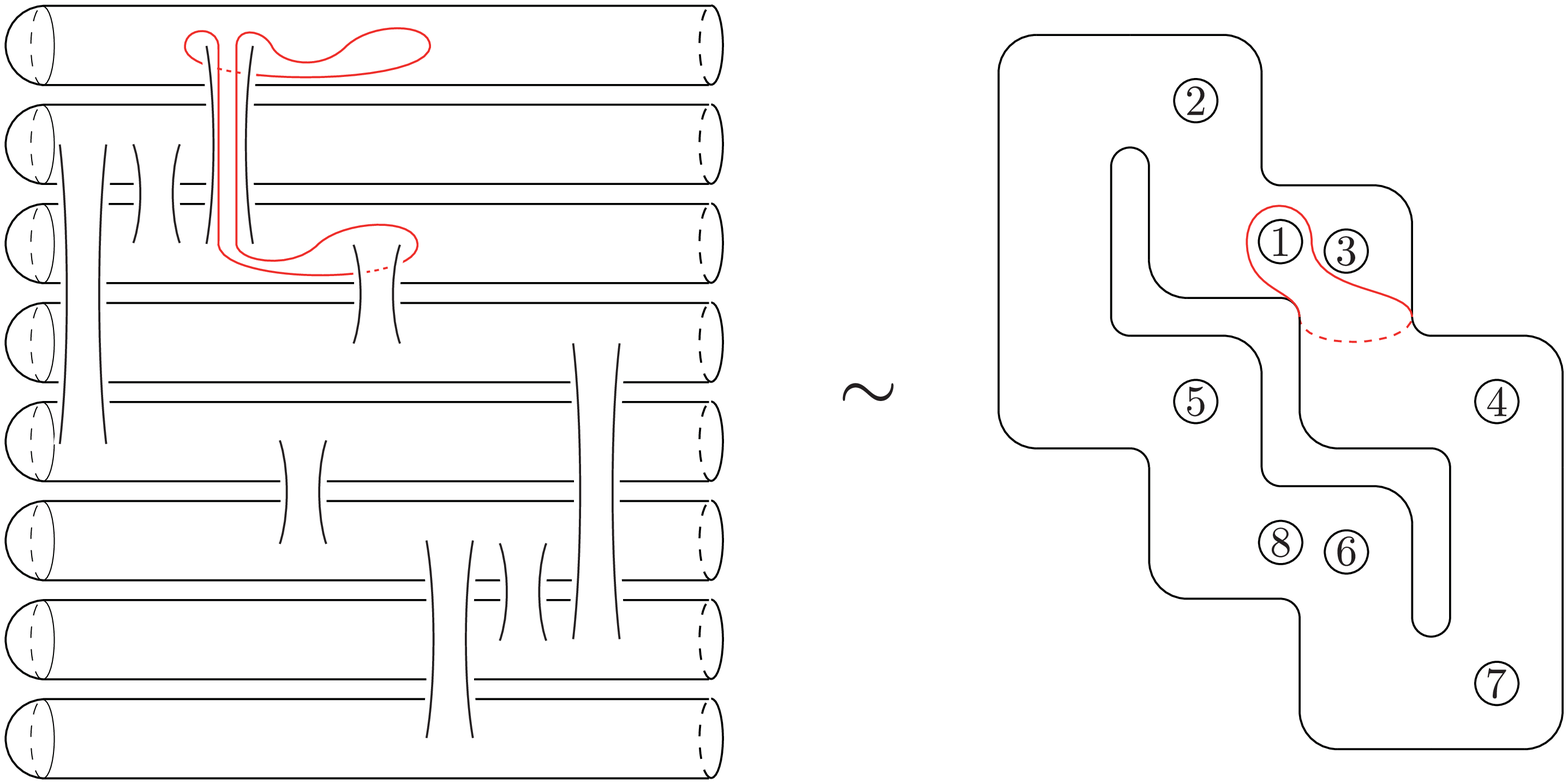}}
\subfigure[The vanishing cycle $c_4$.]{\includegraphics[width=90mm]{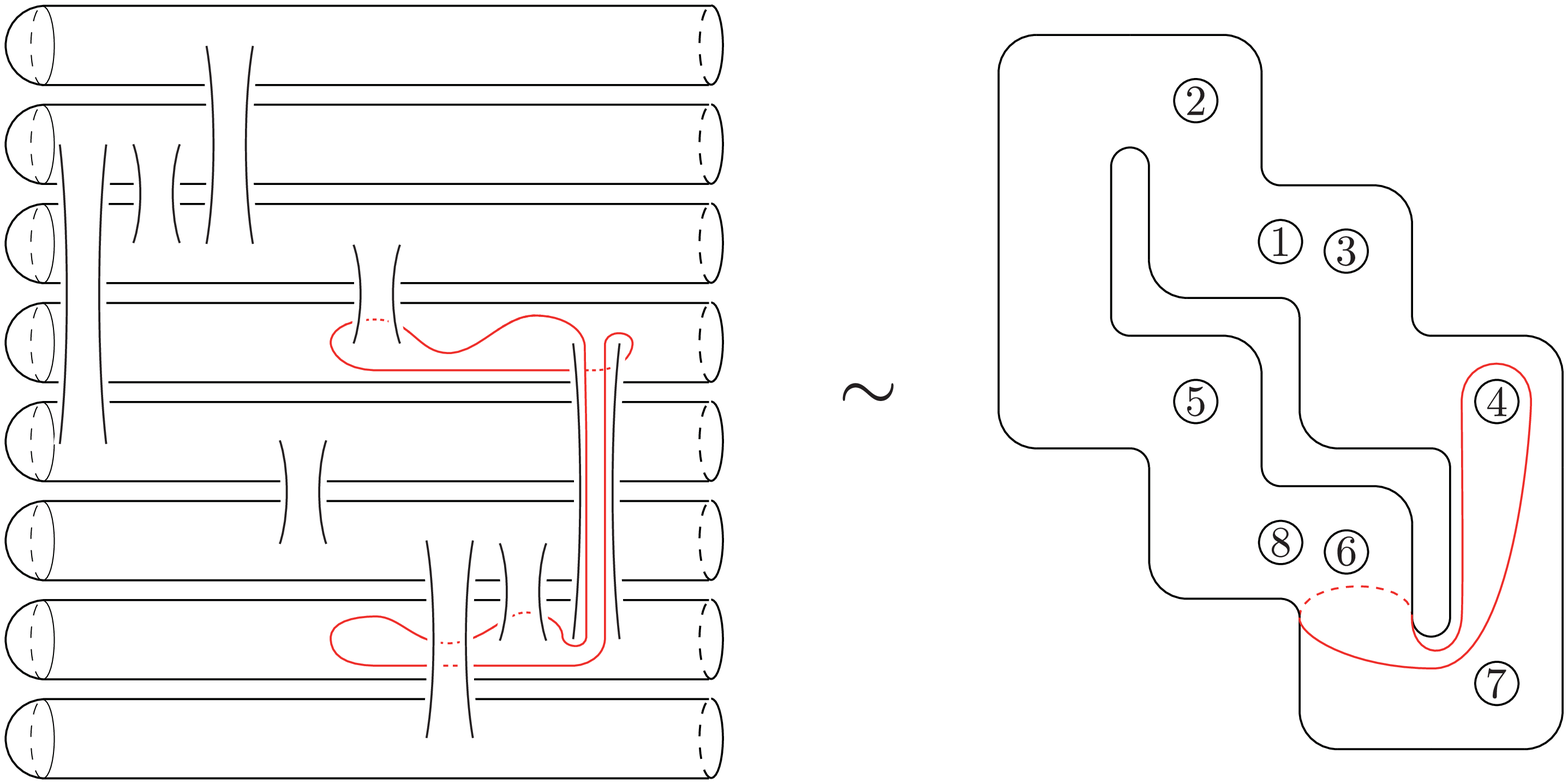}}
\caption{Vanishing cycles of the Lefschetz pencil $f_s$.}
\end{figure}
\addtocounter{figure}{-1}
\begin{figure}[htbp]
\centering
\subfigure{}
\setcounter{subfigure}{4}
\subfigure[The vanishing cycle $c_{5}$.]{\includegraphics[width=90mm]{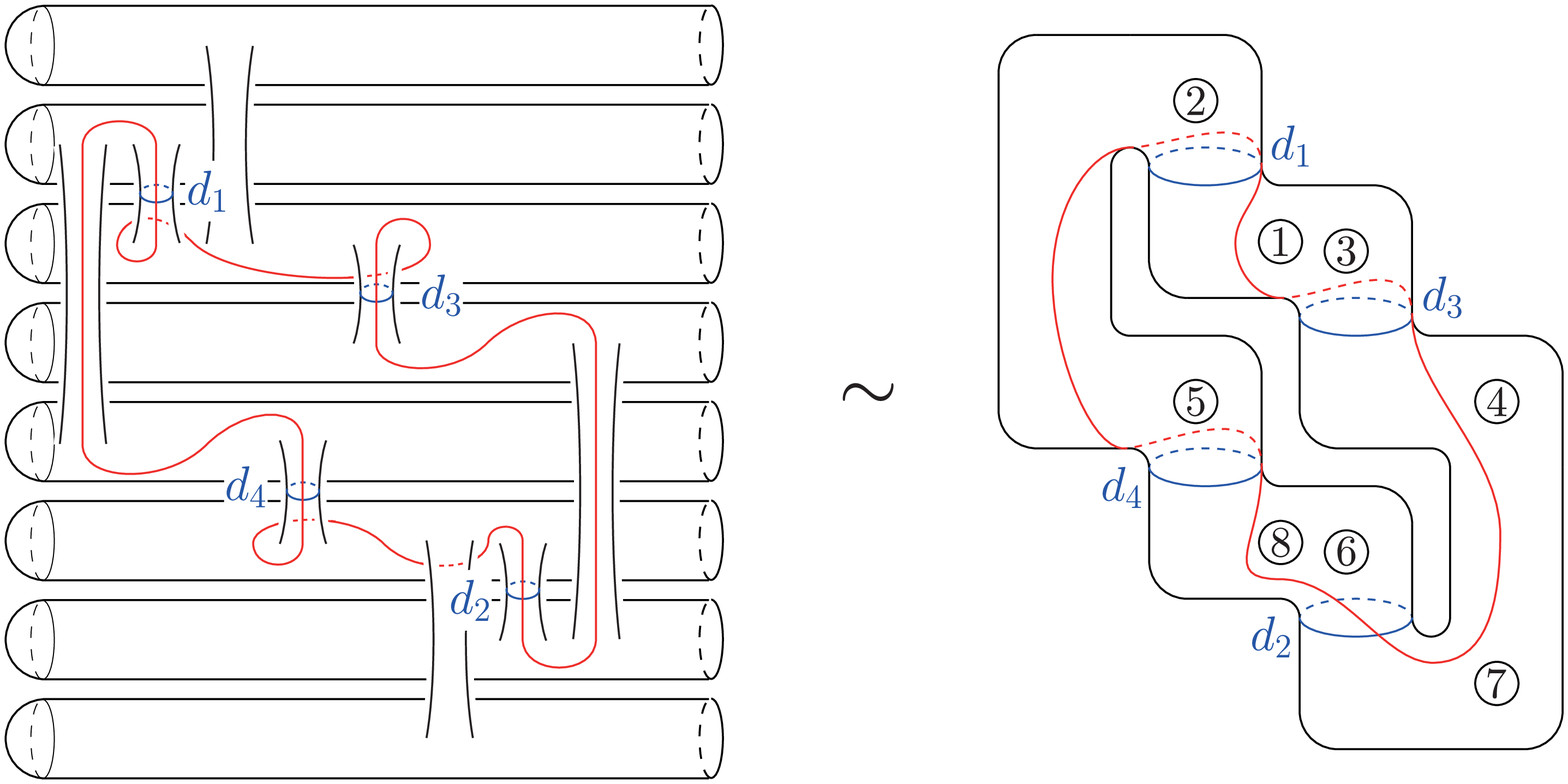}\label{F:vanishing cycle Us5}}
\subfigure[The vanishing cycle $c_{9}$.]{\includegraphics[width=90mm]{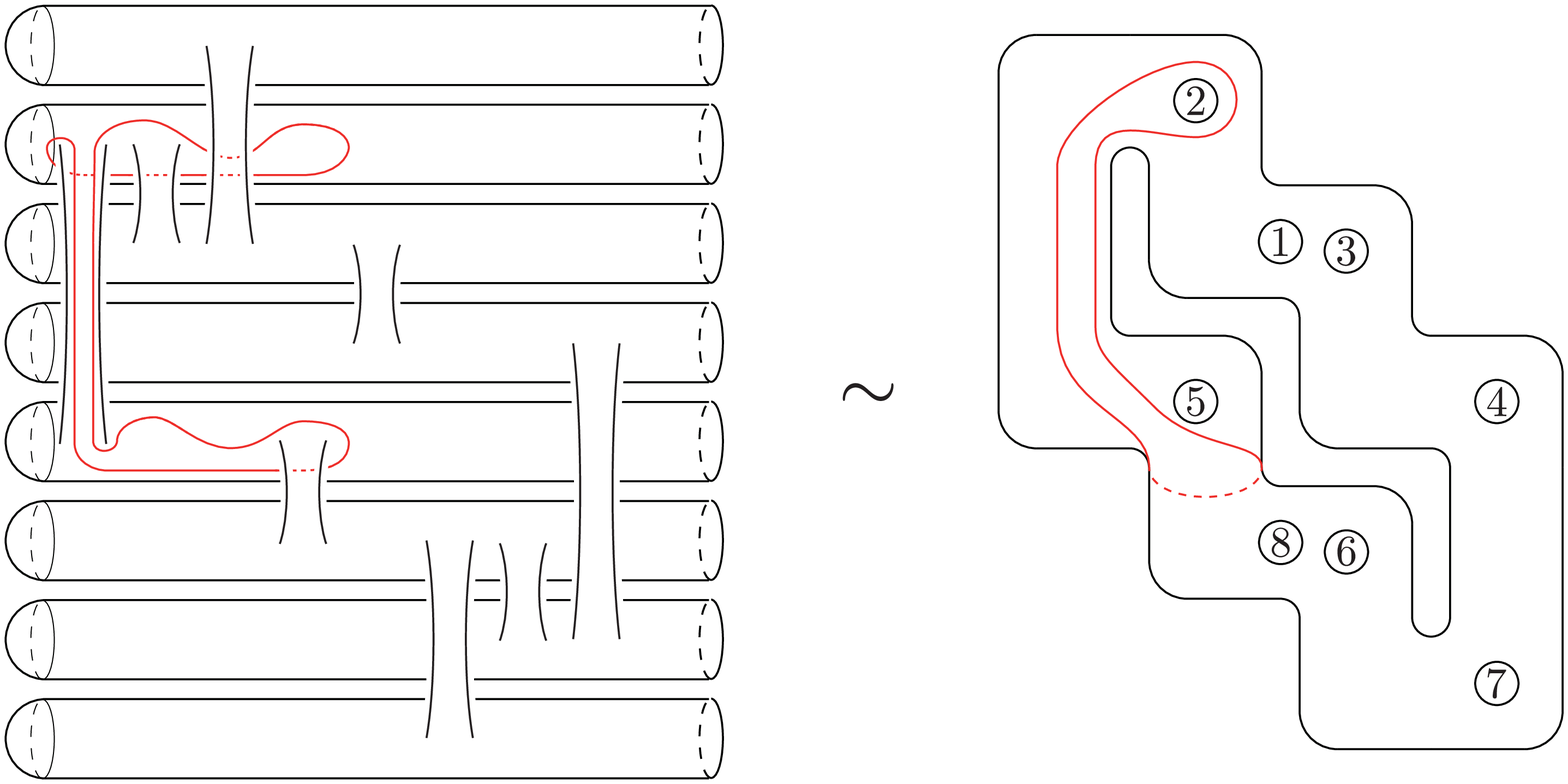}}
\subfigure[The vanishing cycle $c_{10}$.]{\includegraphics[width=90mm]{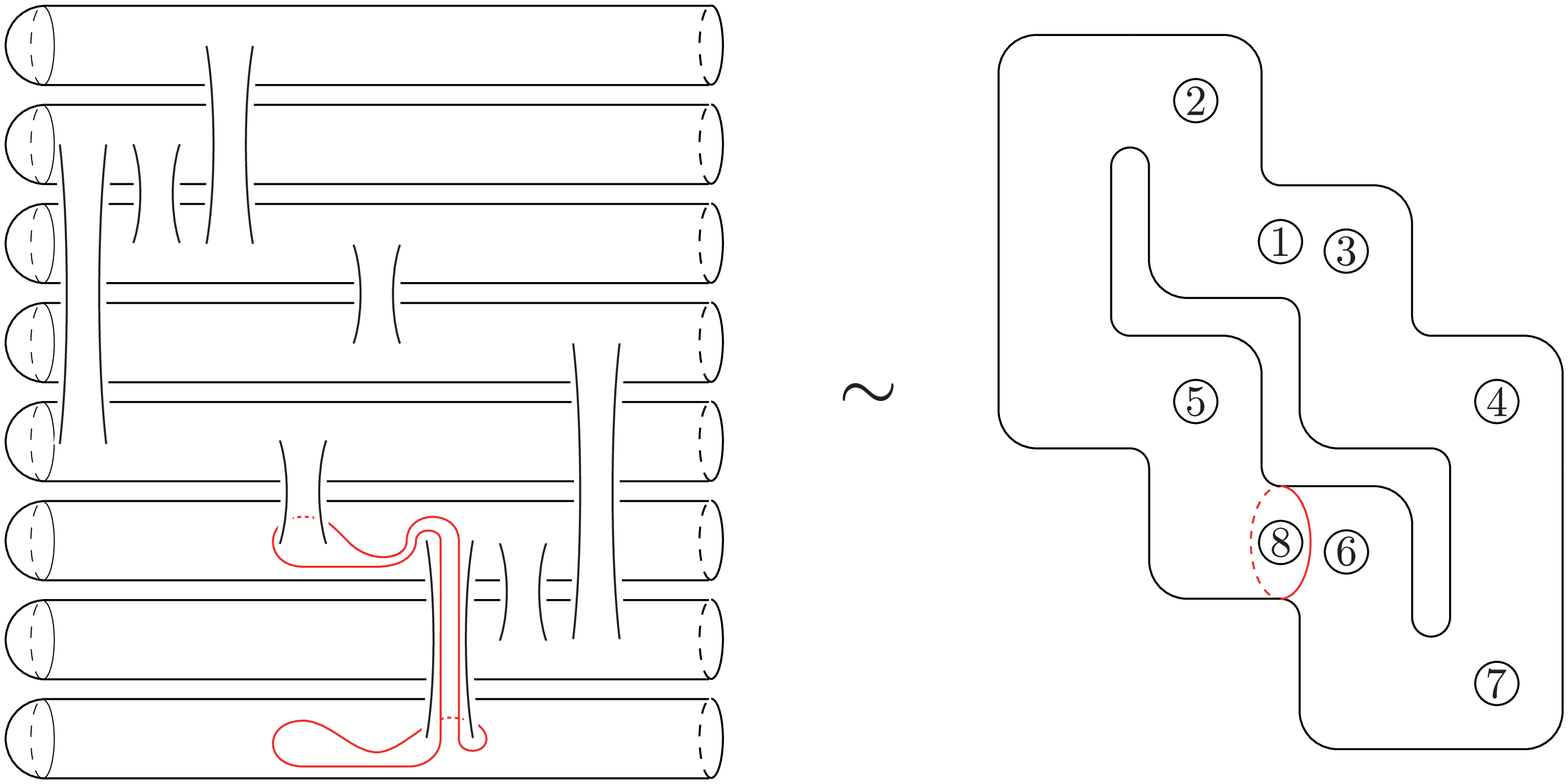}}
\caption{Vanishing cycles of the Lefschetz pencil $f_s$.}
\end{figure}

\end{document}